\documentclass{amsart}

\usepackage{amsmath, amssymb, amsthm, amsfonts, marginnote}

\input xy
\xyoption{all}

\usepackage{hyperref}

\swapnumbers
\numberwithin{equation}{section}

\theoremstyle{plain}
\newtheorem{theorem}[subsubsection]{Theorem}
\newtheorem{lemma}[subsubsection]{Lemma}
\newtheorem{prop}[subsubsection]{Proposition}
\newtheorem{cor}[subsubsection]{Corollary}
\newtheorem{conj}[subsubsection]{Conjecture}
\newtheorem*{claim}{Claim}

\theoremstyle{definition}
\newtheorem{defn}[subsubsection]{Definition}

\newtheorem{remark}[subsubsection]{Remark}

\newtheorem{ex}[subsubsection]{Example}

\def\AA{\mathbb{A}}

\def\CC{\mathbb{C}}
\def\DD{\mathbb{D}}

\def\GG{\mathbb{G}}

\def\PP{\mathbb{P}}
\def\QQ{\mathbb{Q}}
\def\RR{\mathbb{R}}

\def\VV{\mathbb{V}}

\def\XX{\mathbb{X}}

\def\ZZ{\mathbb{Z}}

\def\calB{\mathcal{B}}

\def\calD{\mathcal{D}}
\def\calE{\mathcal{E}}
\def\calF{\mathcal{F}}
\def\calG{\mathcal{G}}
\def\calH{\mathcal{H}}

\def\calK{\mathcal{K}}
\def\calL{\mathcal{L}}

\def\calN{\mathcal{N}}
\def\calO{\mathcal{O}}

\newcommand\cB{\mathcal{B}}

\newcommand\cE{\mathcal{E}}
\newcommand\cF{\mathcal{F}}

\newcommand\cH{\mathcal{H}}
\newcommand\cI{\mathcal{I}}

\newcommand\cK{\mathcal{K}}
\newcommand\cL{\mathcal{L}}
\newcommand\cM{\mathcal{M}}
\newcommand\cN{\mathcal{N}}
\newcommand\cO{\mathcal{O}}

\newcommand\cT{\mathcal{T}}

\newcommand\frX{\mathfrak{X}}

\newcommand\frg{\mathfrak{g}}

\newcommand{\Bun}{\textup{Bun}}

\newcommand\ev{\textup{ev}}

\newcommand{\Fl}{\textup{Fl}}
\newcommand\Forg{\textup{Forg}}

\newcommand{\Gr}{\textup{Gr}}

\newcommand{\Hecke}{\textup{Hecke}}

\newcommand\IC{\textup{IC}}

\renewcommand{\Im}{\textup{Im}}
\newcommand{\Ind}{\textup{Ind}}

\newcommand\Loc{\textup{Loc}}

\newcommand{\mult}{\textup{mult}}

\newcommand{\Perf}{\textup{Perf}}
\newcommand\Perv{\textup{Perv}}
\newcommand{\Pic}{\textup{Pic}}

\newcommand\Proj{\textup{Proj}}
\newcommand{\QCoh}{\textup{QCoh}}

\newcommand\Rep{\textup{Rep}}
\newcommand{\Res}{\textup{Res}}

\newcommand\Sat{\textup{Sat}}

\newcommand\Spec{\textup{Spec}}

\newcommand\St{\mathit{St}}

\newcommand{\Vect}{\textup{Vect}}

\newcommand\Aut{\textup{Aut}}
\newcommand\Hom{\textup{Hom}}
\newcommand\End{\textup{End}}

\newcommand\GL{\textup{GL}}

\newcommand\PGL{\textup{PGL}}

\newcommand\SL{\textup{SL}}
\renewcommand\sl{\mathfrak{sl}}

\newcommand{\Gm}{\GG_m}
\def\Ga{\GG_a}

\newcommand{\incl}{\hookrightarrow}
\newcommand{\isom}{\stackrel{\sim}{\to}}

\newcommand{\surj}{\twoheadrightarrow}
\renewcommand{\a}{\alpha}
\renewcommand{\b}{\beta}
\renewcommand{\d}{\delta}
\newcommand\D{\Delta}
\newcommand{\ep}{\epsilon}
\renewcommand{\l}{\lambda}
\renewcommand{\L}{\Lambda}
\newcommand{\om}{\omega}
\newcommand{\Om}{\Omega}
\renewcommand{\th}{\theta}

\DeclareMathOperator{\Eis}{Eis}

\newcommand{\const}[1]{\underline\QQ_{#1}}
\newcommand{\twtimes}[1]{\stackrel{#1}{\times}}
\newcommand{\jiao}[1]{\langle{#1}\rangle}
\newcommand{\wt}[1]{\widetilde{#1}}

\newcommand\quash[1]{}
\newcommand\mat[4]{\left(\begin{array}{cc} #1 & #2 \\ #3 & #4 \end{array}\right)}  % 2-by-2 matrix

\newcommand\ov{\overline}
\newcommand\bs{\backslash}

\newcommand{\cohoc}[2]{\textup{H}_{c}^{#1}({#2})}     % compact support
\newcommand\upH{\textup{H}}

\newcommand\tdN{\widetilde{\calN}^{\vee}}
\newcommand\dcN{\calN^{\vee}}

\newcommand\tcN{\widetilde{\calN}}

\newcommand\dG{G^{\vee}}
\newcommand\dB{B^{\vee}}
\newcommand\dT{T^{\vee}}
\newcommand\dN{N^{\vee}}

\newcommand\pt{\textup{pt}}
\newcommand\sgn{\textup{sgn}}
\newcommand\vn{\varnothing}

\newcommand\xr{\xrightarrow}

\newcommand\PS{\PP^{1}\setminus S}

\newcommand{\Wh}{\textup{Wh}}

\newcommand{\Maps}{\textup{Maps}}
\newcommand{\Avg}{\textup{Avg}}

\newcommand{\sph}{\mathit{sph}}
\newcommand{\asp}{\mathit{asph}}
\newcommand{\aff}{\mathit{aff}}

\newcommand{\Wf}{W_{f}}  % W^{f} usually denotes min representatives of W^{\aff}/W_{f}

\newcommand{\Lin}{\sum}
\newcommand{\dRLoc}{\textup{Conn}}

\newcommand{\Sh}{\mathit{Sh}}
\newcommand{\triv}{\mathit{triv}}
\newcommand{\alt}{\mathit{alt}}

\newcommand{\Coh}{\textup{Coh}}

\newcommand{\new}{\mathit{new}}
\newcommand{\old}{\mathit{old}}

\title[Geometric Langlands correspondence
for $\SL(2)$, $\PGL(2)$
over the pair of pants]{Geometric Langlands correspondence\\
for $\SL(2)$, $\PGL(2)$
over the pair of pants}

\dedicatory{}
\author{David Nadler}
\thanks{}
\address{Department of Mathematics, UC Berkeley, Evans Hall, Berkeley, CA 94720}
\email{nadler@math.berkeley.edu}
\author{Zhiwei Yun}
\thanks{}
\address{Department of Mathematics, Yale University,  10 Hillhouse Ave., New Haven, CT 06511}
\email{zhiwei.yun@yale.edu}
\date{}

\begin{document}

%%%%%%%%%%%%%%%%%%%%%%%%%%%%%%%%%%%%%%%%%%%%%%%%%%%%

\begin{abstract}
We establish the Geometric Langlands correspondence for rank one groups over the projective line
with three points of tame ramification. 
\end{abstract}

%%%%%%%%%%%%%%%%%%%%%%%%%%%%%%%%%%%%%%%%%%%%%%%%%%%%

\maketitle

\tableofcontents

%%%%%%%%%%%%%%%%%%%%%%%%%%%%%%%%%%%%%%%%%%%%%%%%%%%%
%%%%%%%%%%%%%%%%%%%%%%%%%%%%%%%%%%%%%%%%%%%%%%%%%%%%
%%%%%%%%%%%%%%%%%%%%%%%%%%%%%%%%%%%%%%%%%%%%%%%%%%%%

\section{Introduction}

%%%%%%%%%%%%%%%%%%%%%%%%%%%%%%%%%%%%%%%%%%%%%%%%%%%%

\subsection{Main result}

%
%On the automorphic side, we will work with moduli stacks defined over $\CC$
%and sheaves of $\QQ$-modules on them with respect to the classical topology.
%On the spectral side, we will work with coherent sheaves over  moduli stacks defined over $\QQ$.
%

Let $\PP^1$ denote the complex projective line, and fix the three-element subset $S=\{0, 1, \infty\}\subset \PP^1(\CC)$.

%\marginpar{Removed the first sentence about coeff fields.}
%Fix a  characteristic zero coefficient field $\QQ$, and a characteristic zero algebraically closed base field $\CC$.

Let $\Bun_{\PGL(2)}(\PP^1, S)$ denote the moduli stack (over $\CC$)  of $G=\PGL(2)$-bundles on $\PP^1$ with Borel reductions along $S$.
In more classical language, 
it classifies rank two vector bundles $\calE$ with lines in the fibers $\ell_s \subset \calE|_s$, $s\in S$, all up to tensoring with line bundles.
It is locally of finite type with discretely many isomorphism classes of objects.   

Let $\Sh_!(\Bun_{\PGL(2)}(\PP^1, S))$ be the $\QQ$-linear dg category of constructible complexes of $\QQ$-modules on 
$\Bun_{\PGL(2)}(\PP^1, S)$ that are extensions by zero off of finite type substacks.

Let $\Loc_{\SL(2)}(\PP^1, S)$
denote the moduli stack (over $\QQ$)  of $\SL(2)$-local systems on $\PP^1\setminus S$ equipped near $S$ with a Borel reduction with unipotent monodromy. Thus a point of $\Loc_{\SL(2)}(\PP^1, S)$ consists of triples of pairs $ (A_s,\ell_s)$, ${s\in S}$, consisting of a matrix $A_s\in \SL(2)$ and an eigenline $A_s(\ell_s) \subset \ell_s$ with trivial eigenvalue $A_s|_{\ell_s} = 1$, and the matrices satisfy the equation $A_0 A_1 A_\infty = 1$ inside of $\SL(2)$.  It admits the presentation
$$
\xymatrix{
\Loc_{\SL(2)}(\PP^1, S) \simeq (\tilde\calN^\vee)^{S,\prod=1} /\SL(2) 
}
$$
where $\tilde\calN^\vee \simeq T^*\PP^1$ denotes the Springer resolution of the unipotent variety $\cN^{\vee}$ of $\dG=\SL(2)$, and $(\tilde\calN^\vee)^{S,\prod=1} $ denotes the product of $S$ copies of $\tdN$ with the equation on the group elements $\prod = 1$  imposed inside of $\SL(2)$. Alternatively, it can be shown that $\Loc_{\SL(2)}(\PP^1, S) $ also admits a linear presentation
$$
\xymatrix{
\Loc_{\SL(2)}(\PP^1, S) \simeq T^* ((\PP^1)^S/\SL(2)) 
}
$$ 
where the the equation $\prod = 1$ is replaced by the zero-fiber of the moment map $\mu:T^* ((\PP^1)^S)\to \sl(2)^*$. 

Let $\Coh(\Loc_{\SL(2)}(\PP^1, S))$ be the $\QQ$-linear dg category of coherent complexes on 
$\Loc_{\SL(2)}(\PP^1, S)$.

One can similarly introduce the above objects with the roles of $\PGL(2)$ and $\SL(2)$ swapped.
We will also need the slight variation  where we write $\Coh^{\SL(2)-\alt}(\Loc_{\PGL(2)}(\PP^1, S))$ for the 
$\QQ$-linear dg category of $\SL(2)$-equivariant coherent complexes on $(\tilde\calN^\vee)^{S, \prod=1}$,
where the equation $\prod = 1$ is imposed inside of $\PGL(2)$, and the center $\mu_2 \simeq Z(\SL(2)) \subset \SL(2)$ acts on coherent complexes by the alternating
representation.

Our main theorem is the following Geometric Langlands correspondence 
with tame ramification.

\begin{theorem}\label{thm: intro}
There are  equivalences of dg categories
\begin{equation}\label{eq: 1}
\xymatrix{
 \Coh ( \Loc_{\SL(2)}(\PP^1, S)) \ar[r]^-\sim & \Sh_!(\Bun_{\PGL(2)}(\PP^1, S))
}
\end{equation}
\begin{equation}\label{eq: 2}
\xymatrix{
 \Coh^{\SL(2)-\alt} (\Loc_{\PGL(2)}(\PP^1, S)) \ar[r]^-\sim & \Sh_!(\Bun_{\SL(2)}(\PP^1, S))
}
  \end{equation}
compatible with Hecke modifications and parabolic induction. 
\end{theorem}

\begin{remark}
One can choose an equivalence
\begin{equation*}
\xymatrix{
 \Coh^{\SL(2)-\alt} (\Loc_{\PGL(2)}(\PP^1, S)) \simeq  \Coh (\Loc_{\PGL(2)}(\PP^1, S))
  }
\end{equation*}
for example by tensoring with a line bundle with an odd total twist, and thus reformulate  the second assertion
of the theorem in a more traditional form, but the formulation given in the theorem is more canonical
and independent of choices.
\end{remark}

\begin{remark}
It is also straightforward to use the theorem to deduce a similar result for $GL(2)$. 
\end{remark}

\begin{remark}
One can view the theorem as an instance of the traditional 
de Rham Geometric Langlands correspondence (see for example~\cite{BD}) or alternatively of the topological Betti Geometric Langlands correspondence (see~\cite{BNbetti} for an outline of  expectations).

On the automorphic side, the moduli $\Bun_{\PGL(2)}(\PP^1, S)$ has discretely many isomorphism classes
of objects, hence all of their codirections are nilpotent. Thus, if we work specifically with $\CC$-coefficients, via the Riemann-Hilbert correspondence, there is no difference in considering $\calD$-modules or complexes
of $\CC$-modules (with nilpotent singular support).
 
Similarly, on the spectral side, the Betti moduli $\Loc_{\SL(2)}(\PP^1, S))$ is algebraically isomorphic to the analogous de Rham moduli $\dRLoc_{\SL(2)}(\PP^1, S)$ of parabolic connections (see Corollary~\ref{c:spec mod same}).
Thus their  coherent complexes coincide.
\end{remark}

\begin{remark}
One can further impose nilpotent singular support (in the sense of ~\cite{AG}) on the coherent complexes on the spectral side. Under the equivalences of the theorem, this will correspond to requiring the stalks of the automorphic complexes to be torsion over the equivariant cohomology of automorphism groups. If one then passes to the ind-completions of these categories, what results are equivalences for all automorphic complexes without any constructibility or support restrictions
\begin{equation*}
\xymatrix{
 \Ind\Coh_\calN ( \Loc_{\SL(2)}(\PP^1, S)) \ar[r]^-\sim & \Sh(\Bun_{\PGL(2)}(\PP^1, S))
}
\end{equation*}
\begin{equation*}
\xymatrix{
\Ind \Coh_{\calN}^{\SL(2)-\alt} (\Loc_{\PGL(2)}(\PP^1, S)) \ar[r]^-\sim & \Sh(\Bun_{\SL(2)}(\PP^1, S))
}
  \end{equation*}

\end{remark}

\begin{remark}
One can also pass on the automorphic side to monodromic complexes of any specified monodromy at the three ramification points.
It is possible to find an equivalence with coherent complexes on the corresponding spectral stack of local systems with the same specified monodromy around the three ramification points. In the final section, we sketch the form this takes in the case
of unipotent monodromy at all three ramification points. For monodromy with more general semisimple part, the geometry only simplifies. 
\end{remark}

%%%%%%%%%%%%%%%%%%%%%%%%%%%%%%%%%%%%%%%%%%%%%%%%%%%%

\subsection{Sketch of proof}

We highlight here some of the key structures in the proof of Theorem~\ref{thm: intro}. The second equivalence~\eqref{eq: 2} follows closely from the first~\eqref{eq: 1} so we will focus on the first.

\subsubsection{Spectral action}

The category $ \Sh_!(\Bun_{\PGL(2)}(\PP^1, S))$ is naturally acted upon by a large collection of commuting Hecke operators.

First, at each unramified point $x\in \PP^1 \setminus S$, the symmetric monoidal Satake category $\Sat_{\PGL(2)}  \simeq \Rep(\SL(2))$
of spherical perverse sheaves on the affine Grassmannian $\Gr_{\PGL(2)}$ acts via bundle modifications. 
It is a simple verification that the action is locally constant in $x\in \PP^1 \setminus S$, and hence factors through
the chiral homology
\begin{equation*}
\xymatrix{
 \int_{\PP^1 \setminus S} \Rep(\SL(2)) \simeq \Perf(\Loc_{\SL(2)}(\PP^1 \setminus S))
}
  \end{equation*}

Second, at each point $s\in S$, the monoidal affine Hecke category %$\calH^\aff_{\PGL(2)}$ 
of Iwahori-equivariant constructible complexes on the affine flag variety
$\Fl_{\PGL(2)}$ acts  via modifications of bundles with flags.
In particular, its symmetric monoidal subcategory of Wakimoto operators acts, and hence via Bezrukavnikov's tame local Langlands correspondence~\cite{B}, the tensor category $\Perf(\tdN/\SL(2))$ of equivariant perfect complexes on the Springer resolution  acts at each point $s\in S$.

Thanks to the compatibility of Gaitsgory's central functor~\cite{G}, the above actions assemble into an action 
of the tensor category
of perfect complexes on the spectral stack
\begin{equation}\label{eq: action}
\xymatrix{
\Perf(\Loc_{\SL(2)}(\PP^1, S)) \otimes \Sh_!(\Bun_{\PGL(2)}(\PP^1, S))\ar[r] & \Sh_!(\Bun_{\PGL(2)}(\PP^1, S))
}
  \end{equation}
 By continuity, this can be extended to an action of quasi-coherent complexes
on all automorphic complexes and then further restricted to coherent complexes.

\begin{remark}
In the de Rham Geometric Langlands program, the construction of such an action is a deep ``vanishing theorem"~\cite{Gvanish}.
In the Betti Geometric Langlands program, it is a geometric consequence of requiring automorphic complexes to have nilpotent singular support~\cite{NYhecke}.
\end{remark}

\subsubsection{Whittaker sheaf}

To construct the functor~\eqref{eq: 1} from the  action~\eqref{eq: action}, 
we must choose an automorphic complex to act upon. It will be the object that the spectral structure sheaf $\calO \in \Perf(\Loc_{\SL(2)}(\PP^1, S))$  maps to, and there is a well known candidate given by the Whittaker sheaf $\Wh_{S}
\in  \Sh_!(\Bun_{\PGL(2)}(\PP^1, S))$.

In the situation at hand, the Whittaker sheaf takes the following simple form.
Consider the open substack of $\Bun_{\PGL(2)}(\PP^1, S)$ where the underlying bundle is $\calE \simeq \calO_{\PP^1}(1) \oplus \calO_{\PP^1}$. Consider the further open substacks where the lines take the form
\begin{equation*}
\xymatrix{
 \{\ell_s, s\in S, \mbox{ generic}\}  \ar@{^(->}[r]^-j & 
  \{\ell_s\not \subset \calO_{\PP^1}(1), s\in S\}    \ar@{^(->}[r]^-i & 
\Bun_{\PGL(2)}(\PP^1, S).
}\end{equation*}
Here ``generic'' in the first item means that, in addition to $\ell_{s}\not \subset \calO_{\PP^1}(1)$ for all $s\in S$, the three lines do not simultaneously lie in the image of any map $\cO_{\PP^{1}}\to \cE$. Then the Whittaker sheaf 
is given by the simple topological construction
\begin{equation*}
\xymatrix{
\Wh_{S} \simeq i_! j_* \QQ\in  \Sh_!(\Bun_{\PGL(2)}(\PP^1, S)).
}\end{equation*}

\begin{remark}
The most salient property of the Whittaker sheaf $\Wh_{S}$, and indeed the only property we use, is that it corepresents the functor of vanishing cycles
for a non-zero covector at the point given by the image of the natural induction map
\begin{equation*}
\xymatrix{
\Bun_B^{-1} (\PP^1)\ar[r]^-\sim &  \{\calE \simeq \calO_{\PP^1}(1) \oplus \calO_{\PP^1}, \ell_s   \subset \calO_{\PP^1}, s\in S\}     \subset   \Bun_G(\PP^1, S).
}
\end{equation*}
\end{remark}

\subsubsection{Compatibilities} 

With the functor~\eqref{eq: 1} in hand, to prove it is an equivalence, we first check that it behaves as expected with respect 
to certain distinguished objects.

First, we check that the functor~\eqref{eq: 1} is compatible with induction from two points of tame ramification. (In fact, we check that is equivariant for all affine Hecke symmetries at the ramification points.)
 Namely, for $s\in S$, we show that the functor~\eqref{eq: 1} fits as the top arrow in a commutative diagram
\begin{equation}\label{eq: intro change level}
\xymatrix{
    \Coh( \Loc_{\SL(2)}(\PP^1,S))  \ar[r] & 
      \Sh_!(\Bun_{\PGL(2)}(\PP^1, S))  
\\
  \ar[u]^-{\eta_s^\ell}  \Coh( \Loc_{\SL(2)}(\PP^1, S\setminus s))  \ar[r]^-\sim & 
      \Sh_!(\Bun_{\PGL(2)}(\PP^1, S\setminus s))  \ar[u]^-{\pi_s^*}
}
\end{equation}

Here the bottom arrow is  the Geometric Langlands correspondence
for  two points of tame ramification. A form of the Radon transform identifies it with Bezrukavnikov's tame local Langlands correspondence.  
The automorphic induction $\pi^*_s$ is pullback along the natural $\PP^1$-fibration
 \begin{equation*}
\xymatrix{ 
\pi_s: \Bun_{\PGL(2)}(\PP^1, S) \ar[r] & \Bun_{\PGL(2)}(\PP^1, S\setminus s)
}
\end{equation*}
where we forget the line at $s\in S$.
The spectral induction $\eta^*_s$ is the twisted integral transform 
 \begin{equation*}
\xymatrix{ 
 \eta^\ell_s(\calF) = q_{s*}(p_{s}^*(\calF) \otimes \calO_{\PP^1_s}(-1)[-1])
 }
\end{equation*}
associated to the correspondence
 \begin{equation*}
\xymatrix{ 
 \Loc_{\SL(2)}(\PP^1, S\setminus s) & \ar[l]_-{p_{s}}    \Loc_{\SL(2)}(\PP^1, S\setminus s)  \times_{\{s\}/\SL(2)} \PP_s^1/\SL(2)  \ar[r]^-{q_{s}} &  \Loc_{\SL(2)}(\PP^1, S) 
}
\end{equation*}
%with intermediate stack
%classifying 
%an $\SL(2)$-local system on $\PP^1 \setminus (S\setminus s)$ with $B^\vee$-reductions near $S\setminus s$ with unipotent
%monodromy, and an additional $B^\vee$-reduction at $s$.
Alternatively, under the identification $\Loc_{\SL(2)}(\PP^1, S) \simeq T^* ((\PP^1)^S/\SL(2))$, the correspondence is 
simply the Lagrangian correspondence associated  to the projection $(\PP^1)^S \to (\PP^1)^{S\setminus s}$.

Next,  we check that the functor~\eqref{eq: 1} is compatible with parabolic induction in the form of Eisenstein series.
Namely, on the spectral side,
consider  the natural induction map 
\begin{equation*}
\xymatrix{
\Loc_{B^\vee}(\PP^1, S) \ar[r]^-\sim  & 
\tdN_\Delta/\SL(2) \subset \Loc_{\SL(2)}(\PP^1, S)
}
\end{equation*}
with image the reduced total diagonal where all lines coincide.
On the automorphic side, 
consider  the natural induction map 
\begin{equation*}
\xymatrix{
q_{-1}:\Bun^{-1}_{B}(\PP^1) \ar[r]^-\sim &  \{\calE \simeq \calO_{\PP^1}(1) \oplus \calO_{\PP^1}, \ell_s   \subset \calO_{\PP^1}, s\in S\}   \subset \Bun_{\PGL(2)}(\PP^1, S).
}
\end{equation*}
Then we show the functor~\eqref{eq: 1} matches the objects
\begin{equation*}
\xymatrix{
\calO_{\tdN_\Delta/\SL(2)}\ar@{|->}[r] &  \Eis_{-1} := q_{-1!} \const{ \Bun^{-1}_B(\PP^1)}[-1].
%&
%n\geq -1
}\end{equation*}
By applying Wakimoto operators on both sides, it follows that the functor~\eqref{eq: 1} matches all Eisenstein objects $\calO_{\tdN_\Delta/\SL(2)}(n+1)\mapsto \Eis_{n}$, for all $n\in\ZZ$.

\subsubsection{New forms}

With the preceding compatibilities in hand, we are able to readily deduce that 
 the functor~\eqref{eq: 1} is an equivalence. The key idea is to focus on objects that are ``new forms" in that
 they do not come via  induction from two points of tame ramification.

We introduce the full subcategories of ``old forms" as the images
\begin{equation*}
\xymatrix{
C^\old = \langle \Im(\eta^\ell_s), s\in S\rangle \subset  \Coh ( \Loc_{\SL(2)}(\PP^1, S)) 
}
\end{equation*}
\begin{equation*}
\xymatrix{
 \Sh^\old = \langle \Im(\pi^*_s), s\in S\rangle\subset \Sh_!(\Bun_{\PGL(2)}(\PP^1, S))
}
\end{equation*}
and note that the compatibility~\eqref{eq: intro change level} implies the functor~\eqref{eq: 1} maps $C^{\old}$ essentially surjectively to  $\Sh^{\old}$. 

Thus to show that ~\eqref{eq: 1} is essentially surjective, it suffices to show it induces an essentially surjective functor on the quotient categories of new forms
\begin{equation*}
\xymatrix{
C^\new =  \Coh ( \Loc_{\SL(2)}(\PP^1, S))/C^\old 
}
\end{equation*}
\begin{equation*}
\xymatrix{
   \Sh^\new = \Sh_!(\Bun_{\PGL(2)}(\PP^1, S))/\Sh^\old. 
}\end{equation*}

To achieve this, we first check that $C^\new$ and $\Sh^\new$ are generated respectively by the Eisenstein objects $\cO_{\tdN/\SL(2)}(n+1)$ and $\Eis_{n}$ for $n\geq -1$. We do this by an explicit
parameterization of objects on both sides.

Finally, to show the functor~\eqref{eq: 1} is fully faithful, it suffices by evident Wakimoto  symmetries  and continuity to check it induces isomorphisms
\begin{equation*}
\xymatrix{
 \Hom^\bullet_{ \Coh(\Loc_{\SL(2)}(\PP^1, S))}(\calO, \calO_{\tdN_\Delta/\SL(2)}(n+1)) \ar[r]^-\sim & \Hom^\bullet_{ \Sh_!(\Bun_{\PGL(2)}(\PP^1, S))}(\Wh_{S},\Eis_{n}) & n\geq -1.
}
\end{equation*}
For $n\geq 0$, we observe that both sides vanish, and for $n=-1$, both sides are scalars in degree $0$ and the induced map is indeed
an isomorphism.

%
%
%\begin{remark}[de Rham vs Betti]
%
%Let us begin by recalling the general form of the de Rham and Betti Geometric Langlands correspondences with tame ramification. For the case we will consider in this paper, they coincide and our results draw upon intuitions from each.
%
%
%
%Fix an algebraically closed base field $k$ of characteristic zero.
%
%
%Let $G$ be a reductive group, with Borel subgroup $B\subset G$,
%and Langlands dual group $\dG$.
%
%
%Let $X$ be a smooth projective curve. 
%Fix a finite collection of points $S\subset X$.
%
%Let $\Bun_G(X, S)$ denote the moduli of $G$-bundles on $X$ with a $B$-reduction along $S$. 
%
%Let $\dRLoc_{\dG}(X, S)$
%denote the moduli of $\dG$-bundles on $X$ with a $B$-reduction along $S$  and a connection with simple poles along $S$ with compatible unipotent residue. 
%
%Let $\calD(\Bun_{G}(X, S))$ denote the dg category of $\calD$-modules on $\Bun_{G}(\PP^1, S)$.
%
%Let $\Ind\Coh_\calN(\dRLoc_{\dG}(X, S))$ denote the dg category of ind-coherent sheaves with nilpotent
%singular support on $\dRLoc_{\dG}(X, S)$.
%
%
%\begin{conj}[de Rham Geometric Langlands]
%There is an equivalence
%\begin{equation}
%\xymatrix{
% \Ind\Coh_\calN ( \Loc_{\dG}(X, S)) \ar[r]^-\sim & \calD(\Bun_{G}(X, S))
%}
%\end{equation}
%compatible with Hecke modifications and parabolic induction.
%\end{conj}
%
%
%%\end{remark}

%%%%%%%%%%%%%%%%%%%%%%%%%%%%%%%%%%%%%%%%%%%%%%%%%%%%

\subsection{Motivations}

While the results of this paper can be viewed as an instance of the traditional de Rham Geometric Langlands correspondence,
our initial motivations grew out of our interest in the topological Betti Geometric Langlands correspondence.

To recall the rough form of the Betti Geometric Langlands correspondence,  
let $X$ be a smooth projective curve,
 and $S\subset X$ be a finite collection of points.

Let $\Bun_G(X, S)$ denote the moduli of $G$-bundles on $X$ with a $B$-reduction along $S$. 

Let $\Sh_{\calN }(\Bun_{G}(X, S))$ denote the dg category of complexes of sheaves 
with nilpotent singular support on $\Bun_{G}(\PP^1, S)$.

Let $\Loc_{\dG}(X, S)$
denote the moduli of $\dG$-local systems on $X\setminus S$ equipped near $S$ with a $B^\vee$-reduction with  unipotent monodromy.

Let $\Ind\Coh_\calN(\Loc_{\dG}(X, S))$ denote the dg category of ind-coherent sheaves with nilpotent
singular support on $\Loc_{\dG}(X, S)$.

\begin{conj}[Rough form of Betti Geometric Langlands correspondence]
There is an equivalence
\begin{equation}\label{eq: betti conj}
\xymatrix{
 \Ind\Coh_\calN ( \Loc_{\dG}(X, S)) \ar[r]^-\sim & \Sh_\calN(\Bun_{G}(X, S))
}
\end{equation}
compatible with Hecke modifications and parabolic induction.
\end{conj}

Note that $\Loc_{\dG}(X, S)$ and hence the spectral side \eqref{eq: betti conj} depends only on the topological structure of the curve $X$ and not its algebraic structure. Thus the automorphic side of \eqref{eq: betti conj} is also conjecturally a topological invariant, and hence the fiber at $X$  of a locally constant family of categories over the moduli of curves. In particular, it makes sense to  try to 
produce a ``Verlinde formula" calculating   the automorphic side of \eqref{eq: betti conj} by degenerating to the boundary of the moduli of curves
and replacing $X$ with a nodal graph of genus zero curves. Such a gluing paradigm  for the spectral side of~\eqref{eq: betti conj} was established in~\cite{BNglue}.

Thus  the Betti Geometric Langlands correspondence admits the following two-step strategy:

\begin{enumerate}

\item Produce a ``Verlinde formula" describing  the automorphic category $\Sh_\calN(\Bun_{G}(X, S))$ in terms of the atomic building blocks where $X= \PP^1$, and $S$ comprises $0, 1, 2$, or $3$ points.

\item Establish the Betti Geometric Langlands correspondence for  the atomic building blocks where $X= \PP^1$, and $S$ comprises $0, 1, 2$, or $3$ points.\footnote{There is also a non-orientable version for real reductive groups which leads to the additional atomic building blocks
where $``X= \RR\PP^2"$, and $S$ comprises $0$ or $1$ point.}

\end{enumerate}

For $X= \PP^1$, and $S$ comprising $0, 1$, and $2$, the Betti Geometric Langlands correspondence is equivalent via Radon transforms with the derived Satake correspondence and Bezrukavnikov's tame local Langlands correspondence.  
Thus the remaining challenge for step (2) is to 
establish the Betti Geometric Langlands correspondence for the ``pair of pants" where  $X= \PP^1$, and $S$ comprises $3$ points. This was our original motivation for pursuing the results of this paper.

Independently of the above considerations,  the techniques of this paper also have immediate consequences for the 
 Geometric Langlands correspondence when $G= \PGL(2)$, $X=\PP^1$, and $S$ comprises four or more points.
 Note that $\Bun_{\PGL(2)}(\PP^1, S)$ and  $\Loc_{\SL(2)}(\PP^1, S)$ are of dimension  $\#S-3$ and $2(\#S-3)$ respectively,
 and  when $\#S\geq 4$, there are continuous  moduli of objects  within $\Bun_{\PGL(2)}(\PP^1, S)$ and nontrivial global functions on $\Loc_{\SL(2)}(\PP^1, S)$.  The techniques of this paper most directly apply 
 to the expected correspondence between the full subcategory of $\Sh_{\calN}(\Bun_{\PGL(2)}(\PP^1, S))$
 generated by complexes with unipotent monodromies,
 and  the full subcategory of $\Coh_{\calN}(\Loc_{\SL(2)}(\PP^1, S))$ of coherent complexes supported on local systems with
 global unipotent reductions. We hope to expand upon this in a subsequent paper.

%%%%%%%%%%%%%%%%%%%%%%%%%%%%%%%%%%%%%%%%%%%%%%%%%%%%

\subsection{Conventions}

%Let $\QQ$ be a  characteristic zero coefficient field, and $\CC$ a characteristic zero algebraically closed base field.

On the automorphic side, we will work with moduli stacks defined over the complex numbers $\CC$
and sheaves of $\QQ$-modules on them with respect to the classical topology.

Given a stack  $\frX$ over $\CC$, we write $\Sh(\frX)$, respectively $\Sh_c(\frX)$, for the $\QQ$-linear dg category of  complexes,
respectively constructible complexes,
of $\QQ$-modules on $\frX$. When $\frX$ is  locally of finite-type, we write $\Sh_!(\frX)$ for the $\QQ$-linear dg category of constructible complexes
of $\QQ$-modules on $\frX$ that are extensions by zero off of finite-type substacks.
Given an ind-stack $\frX$, we write $\Sh_c(\frX)$ for the $\QQ$-linear dg category of constructible complexes
of $\QQ$-modules on $\frX$ that are extensions by zero off of substacks.

On the spectral side, we will work with coherent sheaves over  moduli stacks defined over $\QQ$.
All of our categories will be stable (=pre-triangluated) $\QQ$-linear dg categories, and all of our functors will be derived.

\subsection{Acknowledgements} 
We thank David Ben-Zvi for sharing his ideas, and Dennis Gaitsgory for pointing out the role of the Whittaker sheaf.

DN is grateful for the support of NSF grant DMS-1502178.
ZY is  grateful for the support of NSF grant DMS-1302071 and the Packard Foundation.

%%%%%%%%%%%%%%%%%%%%%%%%%%%%%%%%%%%%%%%%%%%%%%%%%%%%
%%%%%%%%%%%%%%%%%%%%%%%%%%%%%%%%%%%%%%%%%%%%%%%%%%%%
%%%%%%%%%%%%%%%%%%%%%%%%%%%%%%%%%%%%%%%%%%%%%%%%%%%%

\section{General constructions}\label{prelims}

In this section, we collect  standard structures from the Geometric Langlands program. Most of the materials in this section are known to experts.

%%%%%%%%%%%%%%%%%%%%%%%%%%%%%%%%%%%%%%%%%%%%%%%%%%%%

\subsection{Group theory}

Let $G$ be a reductive group, $B\subset G$ a Borel subgroup, $N \subset B$ its unipotent radical,
and $T = B/N$ the universal Cartan. 
Let $\calB\simeq G/B$ be the flag variety of $G$.

%\marginnote{Switched $R^{\vee}_{+}$ and $R_{+}$; calling $R_{+}$ coroots sounds a bit weird.}

Let $(\Lambda_T, R^{\vee}_+, \Lambda_T^\vee, R_+)$ be the associated based root datum, where 
$\Lambda_T =\Hom(\GG_m, T)$ is the coweight lattice,
$R^{\vee}_+\subset \Lambda_T$ the positive coroots, $\Lambda_T^\vee = \Hom(T, \GG_m)$ the weight lattice, and 
$R_+\subset \Lambda_T^\vee$ the positive roots.
Let $\Wf$ denote the Weyl group of $G$, and $W^\aff \simeq \Wf\ltimes \Lambda_T$ its affine Weyl group. Let $\rho\in \L^{\vee}_{T}$ (resp. $\rho^{\vee}\in \L_{T}$) be half of the sum of elements in $R_{+}$ (resp. $R^{\vee}_{+}$).

%\marginpar{Changed $W^{f}$ to $W_{f}$: my impression is that $W^{f}$ is used to denote the min representatives in $W^{\aff}/W_{f}$.}

Form the dual based root datum
$(\Lambda_T^\vee, R_+, \Lambda_T, R^{\vee}_+)$, and
construct the Langlands dual group $\dG$, with Borel subgroup $B^\vee\subset \dG$,
unipotent radical $N^\vee\subset B^\vee$,
and dual universal Cartan $T^\vee = B^\vee/N^\vee$. 
Let $\calB^\vee\simeq \dG/B^\vee$ be the flag variety of $\dG$.

 Let $\calN^\vee$ be the nilpotent cone in the Lie algebra $\frg^{\vee}$.
We identify $\calN^\vee $ with the unipotent elements in $\dG$ via the exponential map.
% $$\xymatrix{
% \exp:\calN^\vee\ar[r]^-\sim &  \cU^\vee
% }
% $$
 
Let $\mu:  \tilde \calN^\vee \to \calN^\vee$ be the Springer resolution. Recall that $\tilde \calN^\vee\subset \dG \times \calB^\vee$ classifies pairs $(g,\dB_{1})$
such that the class $g$ lies in the unipotent radical of $\dB_{1}$.  Note the isomorphism of adjoint quotients $N^\vee/B^\vee \simeq \tilde \calN^\vee/\dG$. 

%%%%%%%%%%%%%%%%%%%%%%%%%%%%%%%%%%%%%%%%%%%%%%%%%%%%%%%

\subsection{Hecke kernels}

%%%%%%%%%%%%%%%%%%%%%%%%%%%%%%%%%%%%%%%%%%%%%%%%%%%%%%%

\subsubsection{Satake category}

Let $D = D_-= D_+ = \Spec k[[t]]$ be copies of the formal disk, $D^\times  = \Spec k((t)) \subset D, D_-, D_+$ the punctured formal disk, 
and $\DD = D_-  \coprod_{D^\times} D_+$ the non-separated disk with two zeros $0_- \in D_-, 0_+\in D_+$.

Let $\Bun_G(\DD)$ be the moduli of $G$-bundles on $\DD$. 
%It is the increasing union of stacks
%$$
%\xymatrix{
%\Bun_G(\DD) = \bigcup_{\lambda\in \Lambda^\dom}  \Bun^\lambda_G(\DD)
%}
%$$

Introduce the Laurent series loop group $G((t))= \Maps(D^\times, G)$, with its parahoric arc subgroup $G[[t]] = \Maps(D, G)$,
and affine Grassmannian $ \Gr_G = G((t))/G[[t]]$.
The gluing presentation $\DD = D_-\coprod_{D^\times} D_+$ induces a double-coset presentation 
$$
\xymatrix{
\Bun_G(\DD) \simeq G[[t]]\backslash G((t))/G[[t]] \simeq G[[t]]\backslash \Gr_G.
}
$$

Let $\calH^\sph_G = \Sh_c(\Bun_G(\DD))$ be the dg spherical Hecke category  of  constructible complexes on $\Bun_G(\DD)$ with proper support, or equivalently $G[[t]]$-equivariant constructible complexes on $\Gr_G$
with proper support.
Convolution and fusion equips $\calH^\sph_G$ with an $E_3$-monoidal structure which
 preserves 
 the heart $\calH^{\sph}_{G, \heartsuit} \subset \calH^\sph_G $ with respect
to the perverse $t$-structure. The $E_3$-monoidal structure on $\calH^{\sph}_{G, \heartsuit}$ naturally lifts to a symmetric monoidal structure. Though we mention it  for clarity, we will not need 
the $E_3$-monoidal structure on $\calH^\sph_G$ but only the symmetric monoidal structure on $\calH^{\sph}_{G, \heartsuit}$.

The geometric Satake correspondence \cite{MV, Gi} provides a symmetric monoidal equivalence 
\begin{equation}\label{Sat}
\xymatrix{
\Phi^{\sph}: \Rep(\dG) \simeq \calH^{\sph}_{G, \heartsuit, \rho^\vee} 
}
\end{equation}
where $\calH^{\sph}_{G, \heartsuit, \rho^\vee}$ denotes the same monoidal category $\calH^{\sph}_{G, \heartsuit}$
but with its twisted commutativity constraint.
There is also a derived geometric Satake correspondence but we will not need it.
%%%%%%%%%%%%%%%%%%%%%%%%%%%%%%%%%%%%%%%%%%%%%%%%%%%%%%%

\subsubsection{Affine Hecke category}

Let $\Bun_{G}(\DD, \{0_-, 0_+\})$ be the moduli of  $G$-bundles  on $\DD$ with $B$-reductions  at the points $0_-, 0_+\in \DD$.
The natural projection $\Bun_{G}(\DD, \{0_-, 0_+\})\to \Bun_G(\DD)$ is a $\calB\times \calB$-fibration.

Let $I\subset G[[t]]$ be the Iwahori subgroup given by the inverse image of $B\subset G$ under the evaluation map at $0\in D$,
and  $ \Fl_G = G((t))/I$ the corresponding affine flag variety.
The gluing presentation $\DD = D_- \coprod_{D^\times} D_+$ induces a double-coset presentation 
$$
\xymatrix{
\Bun_{G}(\DD, \{0_-, 0_+\}) \simeq I\backslash G((t))/I \simeq I\backslash \Fl_G.
}
$$

Let $\calH_G^\aff = \Sh_c(\Bun_{G}(\DD, \{0_-, 0_+\}))$  be the dg affine Hecke category
of constructible complexes on $\Bun_{G}(\DD, \{0_-, 0_+\})$  with proper support,
or equivalently $I$-equivariant constructible complexes on $\Fl_G$
with proper support.
Convolution  equips $\calH^\aff_G$ with a monoidal structure.

Recall we write $\mu:\tilde \calN^\vee\to \calN^\vee$ for the Springer resolution,
and identify $\calN^\vee $ with the unipotent elements in $\dG$ via the exponential map.
The Steinberg variety $\St_{\dG}$  is the {\em derived scheme} given by the {\em derived} fiber product
$$
\xymatrix{
\St_{\dG} = \tilde \calN^\vee \times_{\dG} \tilde \calN^\vee.
 }
$$
Passing to adjoint quotients, we have
$$
\xymatrix{
\St_{\dG}/\dG = (\tdN \times_{\dG} \tdN)/\dG
\simeq
 \tdN/\dG\times_{\dG/\dG} \tdN/\dG.
 }
$$

Let $ \Coh^{\dG}(\St_\dG)$ be the dg derived category of  coherent complexes
on $\St_{\dG} /\dG$, or equivalently
$\dG$-equivariant  coherent complexes
on $\St_{\dG}$.
Convolution  equips it with a monoidal structure.

Bezrukavnikov's theorem \cite[Theorem 1]{B} provides a monoidal equivalence
\begin{equation}\label{Bez}
\xymatrix{
\Phi^{\aff}: \Coh^{\dG}(\St_{\dG})\ar[r]^-\sim & \calH_G^\aff.
}
\end{equation}

%%%%%%%%%%%%%%%%%%%%%%%%%%%%%%%%%%%%%%%%%%%%%%%%%%%%%%%

\begin{ex}[Wakimoto sheaves, see {\cite[Section 3.3]{B}}]\label{ex:Waki}
For $\l\in\L_{T}  = \{1\} \ltimes \Lambda_T \subset \Wf\ltimes \Lambda_T =  W^\aff$, we have the $\dG$-equivariant line bundle $\cO_{\calB^{\vee}}(\l)$ on the flag variety $\calB^{\vee} = \dG/B^\vee$. It pulls back under the natural projection
$\pi:  \tdN\to \cB^{\vee}$ to a $\dG$-equivariant line bundle $\cO_{\tdN}(\l) = \pi^*\cO_{\calB^\vee}(\l)$. 

Let $\Delta:\tdN\to \St_{\dG}$ be the diagonal map. Under the equivalence $\Phi^{\aff}$, the coherent sheaf $\D_{*}\cO_{\tdN}(\l)$ corresponds to the {\em Wakimoto sheaf} $J_{\l}$, which can be explicitly constructed as follows. 
Let $j_\l: \Fl_G^{\l}\incl \Fl_G $ be the $I$-orbit indexed by $\l$ regarded in $W^{\aff}$. When $\l$ is dominant, 
$J_{\l}\simeq j_{\l*}\QQ[\jiao{2\rho, \l}]$; when $\l$ is anti-dominant, $J_{\l}\simeq j_{\l!}\QQ[\jiao{2\rho, -\l}]\simeq \DD_{\Fl_G}  \iota J_{-\l}$, where $\iota$ denotes the involution of $I\backslash \Fl_G$ induced by the inverse of $G$.
In general, writing $\l$ as $\l_{1}-\l_{2}$ where $\l_{1}$ and $\l_{2}$ are both dominant, we have 
$J_{\l}\simeq J_{\l_{1}}J_{-\l_{2}}$ independently of the expression of $\l$ as $\l_{1}-\l_{2}$. 
One can check geometrically that the assignment $\l\mapsto J_{\l}$ gives a map of monoids $\L_{T}\to \cH^{\aff}_{G}$.
\end{ex}

%We highlight some distinguished affine Hecke objects in  the following examples.

%%%%%%%%%%%%%%%%%%%%%%%%%%%%%%%%%%%%%%%%%%%%%%%%%%%%%%%

%%%%%%%%%%%%%%%%%%%%%%%%%%%%%%%%%%%%%%%%%%%%%%%%%%%%%%%

\begin{ex}[$G=\SL(2)$]\label{ex:aff hecke for sl2}
The affine Weyl group $W^{\aff}$ can be identified with the infinite dihedral group acting on the real line $\RR$ with fundamental domain $[0,1]$. For $x\in\ZZ$, let $r_{x}\in W^{\aff}$ be the reflection with center $x$, then $W^{\aff}=\jiao{r_{0},r_{1}}$.   

Correspondingly there are two standard monoidal generators $T_{0*}, T_{1*}$ for $\cH^{\aff}$ given by the $*$-extensions of $\const{\Fl^{r_{0}}}[1]$ and $\const{\Fl^{r_{1}}}[1]$. Similarly define $T_{0!}$ and $T_{1!}$ using $!$-extensions instead of $*$-extensions. Then we have monoidal inverses $T_{0*}^{-1} \simeq T_{0!}$, $T_{1*}^{-1} \simeq T_{1!}$.

For $k\in \ZZ$, the Wakimoto sheaf can be expressed as
$
J_{2k} \simeq (T_{0*} T_{1*})^k
$
which corresponds under $\Phi^{\aff}$ to the twist of the
structure sheaf of the relative diagonal  $\calO_{\tdN}(2k)$.

The  finite braid operator  $T_{0*}$ corresponds under $\Phi^{\aff}$ to the {\em classical} structure sheaf  $\calO_{\St_{\dG}}^{cl}$.
Its inverse $T_{0*}^{-1} \simeq T_{0!}$ corresponds to the twist $\calO_{\St_{\dG}}^{cl}(-1,-1)$. 

The affine braid operator $T_{1*}  $ corresponds under $\Phi^{\aff}$ to the twisted classical structure sheaf  $\calO_{\St_{\dG}}^{cl}(-1, 1)$.
Its inverse $T_{1*}^{-1} \simeq T_{1!}$ corresponds to the twist   $\calO_{\St_{\dG}}^{cl}(-2, 0)$.
The conjugate $T_{0*} T_{1*} T_{0*}^{-1} \simeq J_2 T_{1*} J_2^{-1}$ corresponds 
to the twist  $\calO_{\St_{\dG}}^{cl}(1, -1)$, and its inverse $T_{0*} T_{1!} T_{0*}^{-1}\simeq J_2 T_{1!} J_2^{-1}$
to the twist  $\calO_{\St_{\dG}}^{cl}(0, -2)$.

%Let $z:(\PP^1 \times \PP^1)/\PGL(2) \to \St_{\dG}$ be the inclusion.
Let $\Avg$ be the IC-sheaf of the closure of $\Fl^{r_{0}}$. Then $\Avg$ corresponds under $\Phi^{\aff}$ to $\calO_{\PP^1 \times \PP^1}(-1, -1)$.
The natural  distinguished triangles 
$$
\xymatrix{
\Avg \ar[r] & T_{0*} \ar[r] & \delta
&
\delta \ar[r] & T_{0!} \ar[r] & \Avg 
}
$$
correspond to the natural  distinguished triangles 
$$
\xymatrix{
\calO_{\PP^1\times \PP^1}(-1, -1) \ar[r] & \calO_{\St_{\dG}}^{cl} \ar[r] & \calO_{\tilde \calN^\vee}
&
 \calO_{\tilde \calN^\vee} \ar[r] & \calO_{\St_{\dG}}^{cl} (-1, -1) \ar[r] & \calO_{\PP^1 \times \PP^1}(-1, -1)
}
$$

\end{ex}

\begin{ex}[$G=\PGL(2)$]\label{ex:aff hecke for pgl2}
The morphism $\SL(2)\to \PGL(2)$ induces a canonical monoidal functor $\cH^{\aff}_{\SL(2)}\to \cH^{\aff}_{\PGL(2)}$. We use the same notation introduced in Example~\ref{ex:aff hecke for sl2} for objects in $\cH^{\aff}_{\SL(2)}$ to denote their images in $\cH^{\aff}_{\PGL(2)}$. The description of $\Phi^{\aff}(\cF)$ for $\cF\in \cH^{\aff}_{\SL(2)}$ given in Example \ref{ex:aff hecke for sl2} is still valid in the case of $\PGL(2)$ for the same-named sheaf $\cF$ but viewed as in $\cH^{\aff}_{\PGL(2)}$ (note that the Steinberg variety is the same for $\SL(2)$ and $\PGL(2)$).

Now $W^{\aff}$ can be identified with the infinite dihedral group generated by $r_{0}$ and $r_{1/2}$ (reflection with center $1/2$).  Correspondingly, $\cH^{\aff}_{G}$ two standard monoidal generators $T_{0*}$ and $T_{1/2}$, where $T_{1/2}$, the {\em Atkin-Lehner involution}, is the skyscraper sheaf of the point $I$-orbit $\Fl^{r_{1/2}}$, and $T_{1/2}^{-1} \simeq T_{1/2}$.

%We have  with respective monoidal inverses $T_{0*}^{-1}\simeq T_{0!}$, $, T_{1/2}^{-1} \simeq T_{1/2}$.

For $k\in \ZZ$,  the Wakimoto sheaf can be expressed as
$
J_k \simeq (T_{0*} T_{1/2})^k
$
which corresponds to the twist of the
structure sheaf of the relative diagonal  $\calO_{\tilde \calN^\vee}(k)$.

%The finite braid operator $T_{0*}  $ corresponds to the classical structure sheaf  $\calO_{\St_{\dG}}^{cl}$.
%Its inverse $T_{0*}^{-1} \simeq T_{0!}$ corresponds to the twist $\calO_{\St_{\dG}}^{cl}(-1,-1)$. 

The Atkin-Lehner involution $T_{1/2}$  corresponds under $\Phi^{\aff}$ to the twisted classical structure sheaf   $\calO_{\St_{\dG}}^{cl}(-1, 0)$. The conjugate $T_{0*} T_{1/2} T_{0*}^{-1} \simeq J_{1} T_{1/2} J_1^{-1}$ corresponds 
to  the twist $\calO_{\St_{\dG}}^{cl}(0, -1)$.

%Let $z:(\PP^1 \times \PP^1)/\PGL(2) \to \St_{\dG}$ be the inclusion.
%The finite averaging operator $\Avg$ corresponds to the twist of the zero-section $\calO_{\PP^1 \times \PP^1}(-1, -1)$.
%The natural  distinguished triangles  (where $\delta$ is the monoidal unit)
%$$
%\xymatrix{
%\Avg \ar[r] & T_{0*} \ar[r] & \delta
%&
%\delta \ar[r] & T_{0!} \ar[r] & \Avg 
%}
%$$
%correspond to the natural  distinguished triangles 
%$$
%\xymatrix{
%\calO_{\PP^1\times \PP^1}(-1, -1) \ar[r] & \calO_{\St_{\dG}}^{cl} \ar[r] & \calO_{\tilde \calN^\vee}
%&
% \calO_{\tilde \calN^\vee} \ar[r] & \calO_{\St_{\dG}}^{cl} (-1, -1) \ar[r] & \calO_{\PP^1 \times \PP^1}(-1, -1)
%}
%$$

\end{ex}

%%%%%%%%%%%%%%%%%%%%%%%%%%%%%%%%%%%%%%%%%%%%%%%%%%%%

\subsubsection{Compatibilty}

Gaitsgory's nearby cycles construction provides a central functor
$$
\xymatrix{
Z:\calH^{\sph}_{G} \ar[r] & \calH_G^\aff. 
}
$$
Under the Satake equivalence \eqref{Sat} and Bezrukavnikov's equivalence $\Phi^{\aff}$, the central functor becomes the natural functor
$$
\xymatrix{
\Rep (\dG) \ar[r] & \Coh^{\dG}(\St_{\dG})
}
$$
given by pullback along the projection $\St_{\dG}/\dG \to \pt/\dG$.  Its monodromy automorphism corresponds to the universal unipotent automorphism of the pullback.

%%%%%%%%%%%%%%%%%%%%%%%%%%%%%%%%%%%%%%%%%%%%%%%%%%%%

\subsubsection{Finite and aspherical Hecke categories} 
Let $\cH^{f}_{G}=\Sh_{c}(B\bs G/B)$ be the finite Hecke category of $B$-equivariant constructible complexes on 
the flag variety $\cB= G/B$, with monoidal structure defined by convolution.  
Pushforward along the closed embedding $\cB=G/B\incl G((t))/I = \Fl_G$ gives a fully fiathful monoidal functor $\cH^{f}_{G}\to \cH^{\aff}_{G}$.

Let $\Xi\in \Perv_{N}(\cB) \subset \Sh_c( G/B)$ be the tilting extension of the shifted constant sheaf $\const{\cB^{w_0}}[\dim \cB]$ on the  open $N$-orbit  $\cB^{w_0}\subset \cB$. 
Equivalently, in the abelian category $\Perv_{N}(\cB)$,
it is also the projective cover of the skyscraper sheaf on the closed $N$-orbit.  

Consider the functor 
\begin{equation*}
\xymatrix{
\VV= \Hom_{\Sh_c(\cB)}(\Xi, q^*(-))
:     \cH^{f}_{G}\ar[r] &   \Vect 
}
\end{equation*}
where we first forget $B$-equivariance via the pullback $q^*:\cH^{f}_{G} \to\Sh_c(\cB)$ 
along $q:G/B \to B\backslash G/B$.

The functor $\VV$ calculates the vanishing cycles at a generic covector at the closed $N$-orbit.
It is the universal quotient of $\calH^f_G$ with kernel the full monoidal ideal $\jiao{ \IC_w \vert w\not = 1\in \Wf}$
generated by $\IC$-sheaves of  $N$-orbits  $\cB^{w}\subset \cB$, for $w\not = 1\in \Wf$, that are not closed. 
It can be equipped with a monoidal structure (for the usual tensor product on $\Vect$).

The aspherical affine Hecke category is defined to be the tensor product
\begin{equation*}
\cH^{\asp}_{G}:=\cH^{\aff}_{G}\otimes_{\cH^{f}_{G}}\Vect
\end{equation*}
where the $\cH^{f}_{G}$-module structure on $\Vect$ is given by $\VV$.
It has a natural $\cH^{\aff}_{G}$-module structure via convolution on the left.

When the base field has positive characteristic, Bezrukavnikov~\cite{B} realizes $\cH^{\asp}_{G}$ as the dg category of Iwahori-Whittaker sheaves on the affine flag variety with the help of an Artin-Schreier sheaf. By \cite[Theorem 2]{B}, there is an equivalence of dg categories
\begin{equation}\label{Phi asp}
\xymatrix{ \Phi^{\asp}:  \Perf(\tdN/\dG)=\Coh^{\dG}(\tdN)\ar[r]^-\sim & \cH^{\asp}_{G}.
}
\end{equation}
Moreover, the  $\cH^{\aff}_{G}$-action on the left hand side gets intertwined with the $\Coh^{\dG}(\St_{\dG})$-action on the right hand side by left convolution via the equivalence $\Phi^{\aff}$. 

The above equivalence also holds when the base field is $\CC$. One way to see this is to work with $D$-modules (where the exponential $D$-module plays the role of an Artin-Schreier sheaf) to obtain an equivalence between the $\CC$-linearizations of the two sides of \eqref{Phi asp}, and then descend it to $\QQ$. Another way is to use a $\Gm$-averaged version of an Artin-Schreier sheaf, as we do when introducing the Whittaker sheaf in Section~\ref{sss:Gm AS} .

%%%%%%%%%%%%%%%%%%%%%%%%%%%%%%%%%%%%%%%%%%%%%%%%%%%%

\subsection{Hecke modifications}\label{ss:Hk mod}

Let $X$ be a connected smooth projective curve of genus $g$, and $S\subset X$ a finite subset.

Let $\Bun_G(X, S)$ be the moduli stack of $G$-bundles on $X$ with $B$-reductions at $S$. This is an algebraic stack locally of finite type. Later we will focus on the case $G=\PGL(2)$ and $\SL(2)$. For more concrete modular interpretations of $\Bun_{G}(X,S)$ in these cases, see Section~\ref{ss:mod bun}.

Let $\Sh(\Bun_G(X, S))$ be the dg derived category of all complexes on $\Bun_G(X, S)$.
We will abuse terminology and use the term sheaves
to refer to its objects.

Let $\Sh_{!}(\Bun_G(X, S)) \subset \Sh(\Bun_G(X, S))$ be the full dg subcategory of constructible complexes
 that are extensions by zero off of finite-type substacks. 

Introduce copies of the curve $X = X_- = X_+$, and for any $x\in X$, introduce the non-separated curve 
\begin{equation*}
\xymatrix{
\XX_x = X_- \coprod_{ X \setminus \{x\}} X_+
}
\end{equation*}
with the two distinguished points $x_-\in X_-$, $x_+\in X_+$, and the natural embeddings
\begin{equation*}
\xymatrix{
&  \XX_x &  \\
X_- \ar[ur]^-{i_-}  & \DD_x \ar[u]_-j &\ar[ul]_-{i_+}  X_+  }
\end{equation*}
where $\DD_x = D_{x_-} \coprod_{D^\times_x} D_{x_+}$ is the formal neighborhood of $\{x_- ,x_+\} \subset X$.
Note that for the choice of a local coordinate, we can identify $\DD_x$ with the standard model  $\DD$.

%%%%%%%%%%%%%%%%%%%%%%%%%%%%%%%%%%%%%%%%%%%%%%%%%%%%

\subsubsection{Spherical Hecke action}

For $x\in X \setminus S$, we may define the moduli stack $\Bun_{G}(\XX_{x}, S)$ of $G$-bundles on $\XX_{x}$ with $B$-reductions at $S$. We have a diagram 
\begin{equation*}
\xymatrix{
& \ar[dl]_-{p_-} \Bun_G(\XX_x, S) \ar[dr]^-{p_+} \ar[d]^-\kappa& \\
\Bun_G(X, S)  & \Bun_G(\DD_x) &   \Bun_G(X, S)  }
\end{equation*}

Passing to sheaves, and choosing a local coordinate to identify $\Bun_{G}(\DD_{x})$ with $G[[t]]\bs G((t))/G[[t]]$, one obtains the spherical Hecke modifications 
\begin{equation*}
\xymatrix{
\Hecke^{\sph}_x: \calH_G^\sph
 \otimes  \Sh(\Bun_G(X, S))
 \ar[r] & 
\Sh(\Bun_G(X, S))}
\end{equation*}
\begin{equation*}
\xymatrix{
\Hecke^{\sph}_x(\calK, \calF) = (p_{+})_!((p_{-})^*\calF \otimes \kappa^*(\calK)). 
%\in \Sh_!(\Bun_G(X, S))
%&
%\calF \in \Sh_!(\Bun_G(X, S)), \calK \in \calH_G^\sph
}
\end{equation*}
It evidently preserves
the full dg subcategory $\Sh_{!}(\Bun_G(X, S)) \subset \Sh(\Bun_G(X, S))$.

Natural generalizations of the above constructions provide $\Sh(\Bun_G(X, S))$ the requisite coherences of 
an $ \calH_G^\sph$-module.

Restricting to the heart of $\cH^{\sph}_{G}$, one obtains a tensor action 
\begin{equation*}
\xymatrix{
\Rep(\dG) 
 \otimes  \Sh(\Bun_G(X, S))
 \ar[r] & 
\Sh(\Bun_G(X, S)).
}
\end{equation*}

\begin{remark}
It is straightforward to generalize  the above from a point $x\in X\setminus S$ to a family of points $Y \subset X \setminus S$ 
%with a coordinate
to obtain a functor
\begin{equation*}
\xymatrix{
\Hecke^{\sph}_Y: \calH_G^\sph
 \otimes  \Sh(\Bun_G(X, S))
 \ar[r] & 
\Sh(\Bun_G(X, S) \times Y).
}
\end{equation*}

\end{remark}

%%%%%%%%%%%%%%%%%%%%%%%%%%%%%%%%%%%%%%%%%%%%%%%%%%%%%%%%

\subsubsection{Affine Hecke action}

For $s\in S$, let $S_{\pm}=S\coprod_{S\setminus\{s\}}S\subset \XX_{s}$. We may similarly define the moduli stack $\Bun_G(\XX_s, S_{\pm})$ of $G$-bundles on $\XX_{s}$ with $B$-reductions at $S_{\pm}$, and obtain a diagram 
\begin{equation}\label{eq: aff hecke diag}
\xymatrix{
& \ar[dl]_-{p_-} \Bun_G(\XX_s, S_{\pm}) \ar[dr]^-{p_+} \ar[d]^-\kappa& \\
\Bun_G(X, S)  & \Bun_G(\DD_s, \{s_{-}, s_{+}\}) &   \Bun_G(X, S)  
}
\end{equation}
%Note that for the choice of a local coordinate, we can identify $\DD_x$ with the standard model  $\DD$.

Passing to sheaves, and choosing a local coordinate to identify $\Bun_G(\DD_s, \{s_{-}, s_{+}\})$ with $I\bs G((t))/I$, one obtains the affine Hecke modifications 
\begin{equation}\label{eq: aff hecke action}
\xymatrix{
\Hecke^{\aff}_s:\calH_G^\aff
 \otimes  \Sh(\Bun_G(X, S))
 \ar[r] & 
\Sh(\Bun_G(X, S))}
\end{equation}

\begin{equation*}
\xymatrix{
\Hecke^{\aff}_s(\calK, \calF) = (p_{+})_!((p_{-})^*\calF \otimes \kappa^*(\calK)).
%\in \Sh_!(\Bun_G(X, S))
%&
%\calF \in \Sh_!(\Bun_G(X, S)), \calK \in \calH_G^\aff
}
\end{equation*}
More often, we will use the binary notation $\star_{s}$ to denote the affine Hecke action
\begin{equation*}
\cK\star_{s}\cF:=\Hecke^{\aff}_s(\calK, \calF).
\end{equation*}
%For $\cK\in\cH^{\aff}_{G}$, the action $\cK\star_{x}$ 
It evidently preserves
the full dg subcategory $\Sh_!(\Bun_G(X, S)) \subset \Sh(\Bun_G(X, S))$.

Natural generalizations of the above constructions provide $\Sh(\Bun_G(X, S))$ the requisite coherences of an $ \calH_G^\aff$-module
structure. For different $s\in S$, the resulting $ \calH_G^\aff$-actions on $\Sh(\Bun_G(X, S))$ commute with each other.

In particular, restricting the action of $\cH^{\aff}_{G}$ to $\Perf(\tdN/\dG)$ via the monoidal functor
\begin{equation}\label{eq:Waki to Haff}
\xymatrix{ \Perf(\tdN/\dG) \ar[r]^{\D_{*}} &\Coh^{\dG}(\St_{\dG})\ar[r]^-{\Phi^{\aff}} & \cH^{\aff}_{G}}
\end{equation}
where $\D: \tdN/\dG\to \St_{\dG}/\dG$ is the diagonal map,  one obtains commuting tensor actions 
\begin{equation}\label{eq:PerfN S action}
\xymatrix{
\Perf(\tdN/\dG)^{\otimes S}
\otimes  \Sh(\Bun_G(X, S))
\ar[r] & \Sh(\Bun_G(X, S))
}
\end{equation}

%%%%%%%%%%%%%%%%%%%%%%%%%%%%%%%%%%%%%%%%%%%%%%%%%%%%%%%%

\subsubsection{Compatibility}

For $s\in S$, with punctured disk $D^\times_s \subset X$,
there is a natural equivalence of actions
\begin{equation}\label{eq:Z comp}
\xymatrix{
\Hecke_{s}^\aff  \circ Z
\simeq 
\Psi_s \circ \Hecke_{D^\times_s}^\sph
}
\end{equation}
where $\Psi_{s}: \Sh(\Bun_G(X, S) \times D^\times_s) \to \Sh(\Bun_G(X, S))$ denotes nearby cycles towards the central fiber of $\Bun_{G}(X,S)\times D_{s}\to D_{s}$.
Moreover, the monodromy of the central functor $Z$ coincides with the monodromy of $\Psi_{s}$.

%%%%%%%%%%%%%%%%%%%%%%%%%%%%%%%%%%%%%%%%%%%%%%%%%%%%

\subsection{Eisenstein series}

Consider the induction diagram
\begin{equation}\label{B ind all}
\xymatrix{
\Bun_T(X) & \ar[l]_-p \Bun_B(X)  \ar[r]^-q & \Bun_{G}(X, S) 
}\end{equation}
where $p$ is the usual projection, and $q$ assigns to a $B$-bundle the induced $G$-bundle with its given $B$-reduction
remembered along $S$. Since $\Bun_{T}(X)\simeq \L_{T}\otimes_{\ZZ}\Pic(X)$, for each $\l\in \L_{T}$ we have a corresponding component $\Bun^{\l}_{T}(X)$ of $\Bun_{T}(X)$. Let $\Bun^{\l}_{B}(X)$ be the preimage of $\Bun^{\l}_{T}(X)$ under $p$. Restricting the diagram \eqref{B ind all} to the $\l$-component we get
\begin{equation}\label{B ind lam}
\xymatrix{
\Bun^{\l}_T(X) & \ar[l]_-{p_{\l}} \Bun_B^{\l}(X)  \ar[r]^-{q_{\l}} & \Bun_{G}(X, S) 
}\end{equation}

\begin{ex}[$G=\PGL(2)$] In this case, $T=\Gm$, with $\Lambda_T \simeq \ZZ$, and therefore  $\Bun_T(X)\simeq \Pic(X)$. An object of $\Bun_B(X)$ is an inclusion $(\calL \subset \calE)$ of a line bundle into a rank 2 vector bundle on $X$ up to simultaneous tensoring with a line bundle. Then $p$ is given by $(\cL\subset\cE)\mapsto\calL^{\otimes 2} \otimes (\det \calE)^{-1}$, and $q$ is given by $(\cL\subset\cE)\mapsto(\cE, \calL|_S \subset \calE|_{S})$.
An object $(\calL \subset \calE)\in \Bun_{B}(X)$ lies in the component $\Bun^{n}_{B}(X)$
 if and only if $2\deg(\calL) - \deg(\calE)=n$. 
%Note that the image of $q_{n}$ lies in $\Bun^{\bar n}_{\PGL(2)}(X, S)$ where $\bar n=n\mod 2$.

%Recall that $p$ induces an isomorphism 
%\begin{equation}
%\xymatrix{
%\pi_0(\Bun_B(\PP^1)) \ar[r]^-\sim & \pi_0(\Bun_T(\PP^1)) \simeq   \Lambda_T \simeq \ZZ
%}
%\end{equation}

%For $n\in \ZZ$, consider the subdiagram
%\begin{equation}
%\xymatrix{
%\Bun^n_T(\PP^1) & \ar[l]_-{p_n} \Bun^n_B(\PP^1)  \ar[r]^-{q_n} & \Bun^{\bar n}_{\PGL(2)}(\PP^1, S)
%}\end{equation}
%where we fix $n= 2\deg(\calL) - \deg(\calE)$, and the symbol $\bar n = n \mod 2$ stands for  $ev$ when $n$ is even, and $odd$ when $n$ is odd.
\end{ex}

\begin{defn} For $\l\in\L_{T}$, we define the (unipotent) {\em Eisenstein sheaf} to be
\begin{equation*}
\xymatrix{
\Eis_\l = q_{\l!} \const{\Bun^\l_B(X)}[\dim B\cdot (g-1)-\jiao{2\rho, \l}] \in \Sh_{!}(\Bun_{G}(X,S)).
}
\end{equation*}
Note that the shift $\dim B\cdot (g-1)-\jiao{2\rho, \l}$ is the dimension of $\Bun^{\l}_{B}(X)$.
\end{defn}

\begin{ex}[$X=\PP^{1}, \l=0$]\label{ex:Eis 0} In this case, using that $\upH^{1}(\PP^{1},\cO_{\PP^{1}})=0$, we see that $\Bun^{0}_{B}(\PP^{1})\simeq \pt/B$. The map $q_0:\Bun^{0}_{B}(\PP^{1})\to \Bun_{G}(\PP^{1},S)$ is an isomorphism to its image, which is the point classifying the trivial $G$-bundle over $\PP^{1}$ with the same $B$-reduction at all $s\in S$. The
Eisenstein series sheaf 
 $\Eis_{0}$ is the constant sheaf $\QQ[-\dim B]$ on this point extended by zero. 
\end{ex}

The next lemma shows that the Eisenstein sheaves are translated by Wakimoto sheaves.

\begin{lemma}\label{l:J on Eis}
For $\l,\mu\in \L_{T}$, $s\in S$, there is a canonical isomorphism
\begin{equation*}
J_{\mu}\star_{s}\Eis_{\l}\simeq  \Eis_{\mu+\l}.
\end{equation*} 
\end{lemma}
\begin{proof}
We first treat the case when $\mu$ is anti-dominant. To make notation more convenient, let $\mu$ be dominant and consider the action of $J_{-\mu}$ on $\Eis_{\l}$. By definition, $J_{-\mu}$ is the $!$-extension of the constant sheaf $\QQ[\jiao{2\rho, \mu}]$ on $\Fl^{-\mu}_G$. Unravelling the definitions, in particular of the action~\eqref{eq: aff hecke action}, we may describe the Hecke operator $J_{-\mu}\star_{s}$ using the Hecke correspondence 
\begin{equation}\label{Ga n}
\xymatrix{\Bun_{G}(X,S) & \Gamma_{-\mu} \ar[l]_-{{p_-}}\ar[r]^-{{p_+}}  & \Bun_{G}(X,S)}
\end{equation}
given by the subdiagram of \eqref{eq: aff hecke diag} where $\Gamma_{-\mu} \subset  \Bun_G(\XX_s, S)$ classifies pairs of points in $\Bun_{G}(X,S)$ with relative position $-\mu$ at the point~$s$. By definition, we have
\begin{equation}\label{J -mu action}
\xymatrix{
J_{-\mu}\star_{s}\cF={ p}_{+!}{p_-}^{*}\cF[\jiao{2\rho,\mu}], & \textup{ for } \cF\in \Sh(\Bun_{G}(X,S)) 
}\end{equation}

We first assume the following
\begin{claim}
We have a commutative diagram
\begin{equation*}
\xymatrix{  \Bun^{\l}_{B}(X)\ar[d]^{q_{\l}}     &   {}^{\l}\Gamma'_{-\mu}\ar[l]_-{{\gamma_-}}\ar[r]^-{{\gamma_+}}\ar[d]^{h} & \Bun^{\l-\mu}_{B}(X) \ar[d]^{q_{\l-\mu}} \\
\Bun_{G}(X,S) & \Gamma_{-\mu} \ar[l]_-{{p_-}}\ar[r]^-{{p_+}}  & \Bun_{G}(X,S)}
\end{equation*}
with the left square Cartesian and $\gamma_+$ a homeomorphism. 
\end{claim}

From the claim and \eqref{J -mu action}, we can conclude
\begin{eqnarray}
\notag J_{-\mu}\star_{s}\Eis_{\l}&=&{p}_{+!}{p_-}^{*}q_{\l,!}\QQ[\dim B\cdot (g-1)-\jiao{2\rho,\l}][\jiao{2\rho,\mu}]\\
\notag &\simeq& {p_+}_{!}h_{!}\QQ[\dim B\cdot (g-1)-\jiao{2\rho,\l-\mu}]\\
\notag &\simeq & q_{\l-\mu !}{\gamma}_{+!}\QQ[\dim B\cdot (g-1)-\jiao{2\rho,\l-\mu}]\\
\label{J neg}&\simeq & \Eis_{\l-\mu}
\end{eqnarray}
This proves the lemma for $\mu$ anti-dominant.  

Since $J_{\mu}\star_{s}$ is the inverse to $J_{-\mu}\star_{s}$, from \eqref{J neg} we obtain
\begin{equation}\label{J pos}
\xymatrix{
J_{\mu}\star_{s}\Eis_{\l'}\simeq \Eis_{\mu+\l'}  &  \textup{ for } \mu \textup{ dominant, } \l'\in\L_{T}
}\end{equation}

Finally, for general $\mu$, write $\mu$ as $\mu_{1}-\mu_{2}$ where $\mu_{1},\mu_{2}$ are both dominant. Using \eqref{J neg} and \eqref{J pos}, we conclude
\begin{equation*}
J_{\mu}\star_{s}\Eis_{\l}\simeq J_{\mu_{1}}\star_{s}(J_{-\mu_{2}}\star_{s}\Eis_{\l})\simeq J_{\mu_{1}}\star_{s}\Eis_{\l-\mu_{2}}\simeq \Eis_{\l-\mu_{2}+\mu_{1}}=\Eis_{\l+\mu}
\end{equation*}

Now to prove the lemma, it remains to prove the claim. With the choice of $s\in S$, we claim there is  a canonical  morphism
\begin{equation}\label{Bun B add mu}
\xymatrix{
b_{\mu}: \Bun^{\l-\mu}_{B}(X)\ar[r] &  \Bun^{\l}_{B}(X)
}
\end{equation}
Once  this is in hand, a local calculation shows there is a homeomorphism
\begin{equation*}
\xymatrix{
\gamma_+: {}^{\l}\Gamma'_{-\mu} := \Bun^{\l}_{B}(X)\times_{\Bun_{G}(X,S)}\Gamma_{-\mu}
\ar[r] &  \Bun^{\l-\mu}_{B}(X)
}
\end{equation*}
respecting the maps to $\Bun_G(X, S)$.

Thus it remains construct the map~\eqref{Bun B add mu}.  

First, recall the following ``pushout'' construction for filtered vector bundles. Suppose $\cE$ is a vector bundle over $X$ with a finite decreasing filtration $\{F^{i}\cE\}_{i\in \L}$ by subbundles indexed by $i$ in some poset $\L$. Let $i\mapsto \cL_{i}$ be a functor $\L\to \Pic(X)^{\incl}$, where $\Pic(X)^{\incl}$ is the category of line bundles on $X$ with injective sheaf maps as morphisms.  Then there is a canonical vector bundle $\cE'$ equipped with a decreasing filtration $\{F^{i}\cE'\}_{i\in \L}$ such that
\begin{equation*}
\xymatrix{
\Gr^{i}_{F}\cE'\simeq \Gr^{i}_{F}\cE\otimes\cL_{i} &  \textup{ for all } i\in\L.
}
\end{equation*}
The construction is by induction on the number of steps in the filtration, and we omit the details.

Next,  the fiber of the natural projection $\Bun_{B}(X)\to \Bun_{T}(X)$ above a point 
$\calL\in \Bun_T(X)$ classifies the  following data:
\begin{itemize}
\item A tensor functor $\calE:\Rep(G)\to \Vect(X)$ (the tensor category of vector bundles on $X$) denoted by $V\mapsto \cE_{V}$.
\item For $V\in \Rep(G)$, a decreasing filtration $\{F^{\b}\cE_{V}\}_{\b\in\L^{\vee}_{T}}$ indexed by the 
poset
$\L^{\vee}_{T}$ (where $\b\le\b' \in \Lambda_T^\vee$ iff $\b'-\b$ is a $\ZZ_{\ge0}$-combination of simple roots), along with isomorphisms $\Gr^{\b}_{F}\cE_{V}\simeq\cL^{\oplus \dim V(\b)}_{\b}$ (where $V(\b)$ denotes the $\b$-weight space of $V$,  and $\cL_{\b} \in \Pic(X)$ the induction of $\cL \in \Bun_T(X)$ along $\b: \L_{T}\to \ZZ$).
\item Moreover, the filtrations $\{F^{\b}\cE_{V}\}_{\b\in \Lambda_T^\vee}$ and the tensor structure of $V\mapsto \cE_{V}$ are compatible in the following sense: if $V,V'\in \Rep(G)$, then under the isomorphism $\cE_{V\otimes V'}\simeq \cE_{V}\otimes\cE_{V'}$, we have $F^{\b''}\cE_{V\otimes V'}=\sum_{\b+\b'\ge \b''}F^{\b}\cE_{V}\otimes F^{\b'}\cE_{V'}$.
\end{itemize}

Now we are ready to define the map \eqref{Bun B add mu}. Starting with a point $(\cE_{V}; F^{\b}\cE_{V})_{V\in \Rep(G)}$ of $\Bun^{\l-\mu}_{B}(X)$. Let $\cE'_{V}$ be the pushout of $\cE_{V}$ with respect to the line bundles $\b\mapsto \cO_{X}(\jiao{\b,\mu}\cdot s)$. Since $\mu$ is dominant, for $\b\le \b'\in \L^{\vee}_{T}$, we have $\jiao{\b,\mu}\le\jiao{\b',\mu}$ hence a natural inclusion $\cO_{X}(\jiao{\b,\mu}\cdot s)\incl \cO_{X}(\jiao{\b',\mu}\cdot s)$, therefore the pushout is defined. The data $(\cE'_{V}; F^{\b}\cE'_{V})_{V\in\Rep(G)}$ then defines a point in $\Bun^{\l}_{B}(X)$.

This completes the proof of the claim and thus that of the lemma.
\end{proof}

\begin{ex}[$G=\PGL(2)$]  We explain the stacks that appear in the proof above in the case $G=\PGL(2)$. Let $\mu=n\ge0$. The Hecke correspondence  $\Gamma_{-n}$ in the proof above can be described as follows. Let $\wt\Gamma_{-n}$ be  the moduli stack of $(\cE_{-1}\incl \cE_{0}\incl\cdots\incl \cE_{n}; \{\ell_{s'}\}_{s'\in S\setminus\{s\}})$ where each $\cE_{i}$ is a rank two vector bundle on $X$, each arrow $\cE_{i}\incl \cE_{i+1}$ is an upper modification of degree $1$ at $s$, such that $\cE_{i-1}(s)\ne \cE_{i+1}$ for $i=0,1,\cdots, n-1$; finally, for $s'\ne s$, $\ell_{s'}$ is a line of the fiber of $\cE_{0}$ at $s'$. Then we define $\Gamma_{-n}=\wt\Gamma_{-n}/\Pic(X)$ where $\Pic(X)$ acts by simultaneous tensoring on $\cE_{i}$. The map ${p_-}$ sends $(\cE_{-1}\incl \cE_{0}\incl\cdots\incl \cE_{n}; \{\ell_{s'}\}_{s'\in S\setminus \{s\}})$ to $(\cE_{0};\{\ell_{s'}\}_{s'\in S})$ where $\ell_{s}$ is the image of $\cE_{-1}$ in the fiber of $\cE_{0}$ at $s$. The map ${p_+}$ sends $(\cE_{-1}\incl \cE_{0}\incl\cdots\incl \cE_{n}; \{\ell_{s'}\}_{s'\in S\setminus \{s\}})$ to $(\cE_{n}, \{\ell'_{s'}\}_{s'\in S})$ where $\ell'_{s}$ is the image of $\cE_{n-1}$ in the fiber of $\cE_{n}$ at $s$, $\ell'_{s'}$ for $s'\ne s$ is induced from $\ell_{s'}$ after identifying $\cE_{0}|_{X\setminus\{s\}}$ and $\cE_{n}|_{X\setminus \{s\}}$.  

Let $\l=m\in\ZZ$. The stack ${}^{m}\Gamma'_{-n}$ defined in the proof above has the following moduli interpretation. It classifies $(\cL\subset \cE_{0}\incl\cdots\incl \cE_{n})$ where the chain $\cE_{0}\incl \cdots\cE_{n}$ is as before, $\cL$ is a line subbundle of $\cE_{0}$ which is also saturated in $\cE_{1}$ (the last condition is equivalent to $\cE_{-1}(s)\ne \cE_{1}$, if we define $\cE_{-1}$ to be the lower modification of $\cE_{0}$ at $s$ determined by the line $\cL_{s}$ of the fiber of $\cE_{0}$ at $s$).  It is easy to see inductively that $\cL$ is saturated in $\cE_{2},\cdots, \cE_{n}$. Therefore $(\cL\subset \cE_{n})$ defines a point in $\Bun^{m-n}_{B}(X)$. This gives the map $\gamma_+: {}^{m}\Gamma'_{-n}\to \Bun^{m-n}_{B}(X)$ which is an isomorphism: the pair $\cL\subset \cE_{n}$ determines the chain $\cE_{0}\incl\cdots\incl \cE_{n}$ because $\cE_{i-1}$ can be inductively identified with the pullback of $(\cE_{i}/\cL)(-s)$ under the quotient $\cE_{i}\surj \cE_{i}/\cL$. 
\end{ex}

%%%%%%%%%%%%%%%%%%%%%%%%%%%%%%%%%%%%%%%%%%%%%%%%%%%%

\subsection{Whittaker sheaf}

In this subsection, we assume in addition that 
$\rho^{\vee}\in\L_{T}$, for example $G$ is adjoint.

\subsubsection{Twisted $N$-bundles}
Consider the distinguished $T$-bundle
\begin{equation*}
\om(S) :=\rho^{\vee}\otimes\om_{X}(S) \in\Bun_{T}(X) \simeq \L_{T}\otimes_{\ZZ}\Pic(X).
\end{equation*}

Define the moduli 
\begin{equation*}
\Bun^{\om(S)}_{N}(X, S)=\Bun^{\om(S)}_{N}(X)\times_{\Bun_{G}(X)}\Bun_{G}(X, S)
\end{equation*}
classifying triples $(\calE_{B},\tau, \{\calF_{s}\}_{s\in S})$, where $\calE_{B}$ is a $B$-torsor, 
 $\tau: \calE_{B}/N\to \om(S)$ is an isomorphism of $T$-torsors,  and $\calF_{s}$ is a $B$-reduction of the fiber $\calE_{G}|_s$ of the $G$-bundle induced by $\calE_B$. In other words, the choice of $\calF_{s}$
 is equivalent to the choice of a point of the twisted flag variety
 $\cB_{\calE_B|_s}=\cE_{B}|_{s}\twtimes{B}\cB$.

Observe that there is an open substack
\begin{equation*}
\Bun^{\om(S),\circ}_{N}(X, S)
\subset \Bun^{\om(S)}_{N}(X, S)
\end{equation*}
where  the $B$-reductions $\calE_B|_s$ and $\calF_s$ of the fiber $\calE_G|_s$
are transverse, for each $s\in S$. If we let $\cB^{\circ}\subset \cB$ be the  open $B$-orbit, then  the choice of $\calF_{s}$ is now  equivalent to the choice of a point of the twisted open cell
\begin{equation*}
\cB^{\circ}_{\calE_B|_x}
\subset
\cB_{\calE_B|_s}
\end{equation*}
Note since $B\backslash \cB^{\circ} \to T\backslash pt$ is an equivalence,  the choice of such $\calF_{s}$ is in turn equivalent to a splitting of $\calE_B|_s\to \omega(S)|_s$.
 
Thus the abelianization map $N\to N/[N,N]\simeq  \prod_{i=1}^{r}\GG_{a}$, where $r$ is the rank,
 induces a map 
 \begin{equation}\label{eq:N to Ga}
 \xymatrix{
 \Bun^{\om(S), \circ}_{N}(X, S)\ar[r] & \prod_{i=1}^{r}\Bun^{\om_X(S)}_{\GG_{a}, S}(X)
 }
 \end{equation}
where $\Bun^{\om_X(S)}_{\GG_{a}, S}(X)$ classifies extensions $\om_X(S) \to \calE\to \calO_X $ with a splitting at each $s\in S$.

Pushout of extensions along the inclusion $\omega_X \to\omega_X(S)$ provides a canonical equivalence 
 \begin{equation}\label{eq:GaS}
\xymatrix{
\Bun^{\om_X}_{\GG_{a}}(X) \ar[r]^-\sim&  \Bun^{\om_X(S)}_{\GG_{a}, S}(X) 
&
(\om_X \to \calE\to \calO_X)  \ar@{|->}[r] & (\om_X(S) \to \calE' \to \calO_X) 
}
 \end{equation}
 since  the inclusion $\calE_{x}\to \calE'_{x}$ factors through $\calE_{x}\to \calO_{X, x}$, and hence its image gives a splitting of $\calE_x'\to \calO_{X, x}$.

Composing ~\eqref{eq:N to Ga} with the inverse of ~\eqref{eq:GaS} and taking the sum of the canonical evaluations 
\begin{equation*}
\xymatrix{
\Bun^{\om_X}_{\GG_{a}}(X)\simeq H^{1}(X,\om_X)\simeq \GG_{a}
}\end{equation*}
we obtain the total evaluation 
\begin{equation*}
\xymatrix{
\ev:\Bun^{\om(S), \circ}_{N}(X, S)\ar[r] & \GG_{a}.
}\end{equation*}
Note that the total evaluation is $\GG_m$-equivariant for the action on $\Bun^{\om(S), \circ}_{N}(X, S)$ induced via $\rho^\vee:\GG_m\to T$ from the adjoint $T$-action
and the usual rotation action on $\GG_a$.
Therefore it descends to a map
\begin{equation*}
\xymatrix{
\ov{\ev}: \Bun^{\om(S),\circ}_{N}(X)/\Gm\ar[r] &\GG_{a}/\Gm.
}
\end{equation*}

We also have the natural induction map
\begin{equation*}
\xymatrix{
p: \Bun^{\om(S),\circ}_{N}(X, S) \ar[r] & \Bun_{G}(X, S)
}
\end{equation*}
which descends to a map
\begin{equation*}
\xymatrix{
\ov{p}: \Bun^{\om(S),\circ}_{N}(X, S)/\Gm  \ar[r] & \Bun_{G}(X, S)
}
\end{equation*}
where again the $\GG_m$-action on $\Bun^{\om(S), \circ}_{N}(X, S)$ is induced via $\rho^\vee:\GG_m\to T$ from the adjoint $T$-action.

\subsubsection{$\Gm$-averaged Artin-Schreier sheaf}\label{sss:Gm AS} 
Let us write $j:pt = \Gm/\Gm\to \Ga/\Gm$ for the open inclusion. Let
\begin{equation*}
\Psi:=j_{*}\const{\pt}[-1] \in D_{\Gm}(\Ga).
\end{equation*}
This sheaf should be thought of as a $\Gm$-equivariant version of an Artin-Schreier sheaf over $\Ga$ if we worked over a base field of finite characteristic, or a $\Gm$-equivariant version  of the exponential $D$-module over $\Ga$ if we worked in the $D$-modules setting.

\begin{defn} 
The Whittaker sheaf is the object
\begin{equation*}
\xymatrix{
\Wh_{S}= \ov{p}_{!}\ov{\ev}^{*}\Psi[-d_{S}] \in \Sh_{!}(\Bun_G(X, S))
}
\end{equation*}
where
\begin{equation*}
d_{S}=\dim B\cdot (g-1)+\jiao{2\rho, \rho^{\vee}}(2g-2+\#S)
\end{equation*}
is the dimension of $\Bun^{-\om(S)}_{B}(X)$.
\end{defn}

\begin{ex}[$X=\PP^{1}, S=\{0,\infty\}$]\label{ex:Wh 2 pts}
In this case, we have $\omega(S) \simeq \calO_{\PP^1}$, and hence the Whittaker sheaf is supported on the open locus
where the underlying $G$-bundle is semistable or equivalently trivializable
\begin{equation*}
\xymatrix{
\Bun^\triv_G(\PP^1, \{0, \infty\}) \simeq G\bs (\cB\times \cB). 
}
\end{equation*}
On the other hand, let $\cB^{\circ}$ be the open $N$-orbit in $\cB$, then we have
\begin{equation*}
\Bun^{\om(S),\circ}_{N}(\PP^1, \{0, \infty\}) \simeq N\bs (\cB^{\circ}\times \cB^{\circ}).
\end{equation*}
If we choose a point $B^{-}\in \cB^{\circ}$ represented by a Borel opposite to $B$, then we have $G\bs (\cB\times \cB)\simeq B^{-}\bs \cB$ by fixing the first coordinate to be $B^{-}$; similarly, we have $N\bs (\cB^{\circ}\times \cB^{\circ})\simeq \cB^{\circ}$ by fixing the first coordinate to be $B^{-}$. Under the above isomorphisms, the map $p : \Bun^{\om(S),\circ}_{N}(\PP^1, \{0, \infty\}) \to \Bun_{G}(\PP^1, \{0, \infty\})$ is the evident composition
\begin{equation*}
\xymatrix{
\calB^{\circ} \ar@{^(->}[r]^-i & \calB \ar@{->>}[r]^-q & B^{-}\backslash \calB
}
\end{equation*}

%Writing $\calB_{w_0} \subset \calB$ for the open $N$-orbit, we also have an isomorphism
%\begin{equation}
%\xymatrix{
%\Bun^{\om(S),\circ}_{N}(\PP^1, \{0, \infty\}) \simeq \calB_{w_0}  
%}
%\end{equation}
%such that the map $p : \Bun^{\om(S),\circ}_{N}(\PP^1, \{0, \infty\}) \to \Bun_{G}(\PP^1, \{0, \infty\})$ is the evident composition
%\begin{equation}
%\xymatrix{
%\calB_{w_0}  \ar@{^(->}[r]^-i & \calB \ar@{->>}[r]^-q & B\backslash \calB
%}
%\end{equation}
%and its descent  $\ov{p} $ is the evident composition
%\begin{equation}
%\xymatrix{
%\rho^\vee(\GG_m)\backslash \calB_{w_0}  \ar@{^(->}[r] & \rho^\vee(\GG_m) \backslash\calB \ar@{->>}[r] & B\backslash \calB
%}
%\end{equation}

Let $\Xi\in \Perv_{N}(\cB)$ be the tilting extension  to $\cB$ of the constant perverse sheaf $\const{\cB^{\circ}}[\dim \cB^{\circ}]$. We claim that
\begin{equation}\label{Wh 2 pts}
\Wh_{\{0,\infty\}}\simeq u_{!}q_{!}\Xi[\dim B].
\end{equation}
Here $u: B^-\bs \cB\simeq \Bun^{triv}_{G}(\PP^1, \{0, \infty\})\incl \Bun_{G}(\PP^1, \{0, \infty\})$ is the open inclusion. To see this, we only need to note that both sides of \eqref{Wh 2 pts}, up to appropriate shifts, corepresent the functor of vanishing cycles at a generic covector at the image of  $\Bun^0_B(\PP^1) \to \Bun_G(\PP^1, \{0, \infty\})$.
\end{ex}

The Whittaker sheaf $\Wh_{S}$ enjoys an asphericity property, as we spell out now. For $s\in S$ and a parabolic subgroup $P\subset G$, we may define a moduli stack $\Bun_{G}(X,S)_{s,P}$ where the level structure at $s$ is changed to a $P$-reduction. We have a proper smooth projection
\begin{equation*}
\pi_{s,P}:\Bun_{G}(X,S)\to \Bun_{G}(X,S)_{s,P}
\end{equation*}
which induces adjoint functors
\begin{equation*}
\xymatrix{
\Sh_!(\Bun_{G}(X, S))
\ar[rr]^-{\pi_{s,P,*}=\pi_{s,P!}} &&
 \Sh_{!}(\Bun_{G}(X, S)_{s,P}) \ar@<-2.5ex>[ll]_-{\pi^{*}_{s,P}}\ar@<2.5ex>[ll]_-{\pi^{!}_{s,P}}}
\end{equation*}

\begin{lemma}\label{l:Avg Wh 0}
Let $s\in S$ and $P\subset G$ be a parabolic subgroup which is not a Borel. Then
\begin{equation*}
\pi_{s,P!}\Wh_{S}\simeq 0.
\end{equation*}
\end{lemma}
\begin{proof} Let $P_{i}$ be the standard parabolic whose Levi only has simple root $\a_{i}$. Then each $P$ which is not a Borel contains some $P_{i}$, and $\pi_{s,P}$ factors as
\begin{equation*}
\Bun_{G}(X,S)\xr{\pi_{s,P_{i}}} \Bun_{G}(X,S)_{s, P_{i}}\to \Bun_{G}(X,S)_{s,P}.
\end{equation*}
Therefore it suffices to show that $\pi_{s,P_{i}!}\Wh_{S}\simeq 0$, for each $P_{i}$.
 
We denote $\Bun_{G}(X,S)_{s,P_{i}}$ simply by $\Bun_{G}(X,S)_{s,i}$, and denote  $\pi_{s,P_{i}}$ similarly by $\pi_{s,i}$, which is a $\PP^{1}$-fibration.

Let us  extend the maps in the  definition of $\Wh_S$
 to a commutative (but not Cartesian) diagram
\begin{equation*}
\xymatrix{
\pt = \GG_m/\GG_m \ar@{^(->}[r]^-j & \GG_{a}/\Gm
 & \ar[l]_-{\ov{\ev}} \Bun^{\om(S),\circ}_{N}(X, S)/\Gm  \ar[r]^-{\ov{p}} \ar[d]^{\pi'_{s, i}} & \Bun_{G}(X, S)
 \ar[d]^{\pi_{s, i}}\\
 &&\Bun^{\om(S),\circ}_{N}(X, S)_{s, i}/\Gm  \ar[r]^-{\ov{p}_{s,i}} &  \Bun_{G}(X, S)_{s, i}
 }
\end{equation*}
 where we denote by $\Bun^{\om(S),\circ}_{N}(X, S)_{s, i}/\Gm$
 the moduli where we replace the $B$-reduction at $s$ with a $P_i$-reduction in general position with the given $N$-structure, and 
  $\pi'_{s, i}$ is the natural $\AA^{1}$-fibration where we forget the $B$-reduction at $s$ to a $P_i$-reduction.

Now returning to the definition of $\Wh_S$, we have 
\begin{equation*}
\xymatrix{
 \pi_{s, i!}\Wh_S =  \pi_{s, i!} \ov{p}_!\ov{\ev}^*j_*\const{\pt}[-1-d_S] \simeq \ov{p}_{s,i!} \pi'_{s, i!}\ov{\ev}^*j_*\const{\pt}[-1-d_S] 
 }
\end{equation*}
and so it suffices to show 
\begin{equation*}
\xymatrix{
\pi'_{s, i!}\ov{\ev}^*j_*\const{\pt}\simeq 0.
 }
\end{equation*}

Fix a point $\xi:\pt\to \Bun^{\om(S),\circ}_{N}(X, S)_{s, i}$, and consider the base-changed Cartesian diagram
\begin{equation*}
\xymatrix{
\Gm \ar@{^(->}[r]^-j & \GG_{a}
 & \ar[l]_-{{\ev}} \Bun^{\om(S),\circ}_{N}(X, S) \ar[d]^{\pi'_{s, i}} & \ar[l]_-{\wt \xi} \AA^{1}
 \ar[d] \\
 &&\Bun^{\om(S),\circ}_{N}(X, S)_{s, i}  & \ar[l]_-{\xi} \pt 
 }
\end{equation*}
Then it suffices to show 
\begin{equation*}
\xymatrix{
 \xi^*\pi'_{s, i!}{\ev}^*j_*\const{\Gm}\simeq 0.
}
\end{equation*}

%\marginpar{Fixed a small gap here: $\ev\circ \wt\xi :\AA^{1}\to \GG_a$ is not necessarily an isom of group schemes, but only an affine isom.}

Finally, observe that $\ev\circ \wt\xi :\AA^{1}\to \GG_a$ is an isomorphism of schemes, and so 
\begin{equation*}
\xymatrix{
\wt\xi^*\ev^*j_*\const{\Gm}\simeq j'_*\const{U}.
 }
\end{equation*} 
where $j': U=(\ev\circ\wt \xi)^{-1}(\Gm)\incl \AA^{1}$ is the complement of one point in $\AA^{1}$.  Thus we have the required vanishing 
\begin{equation*}
\xymatrix{
 \xi^*\pi'_{s, i!}\ev^*j_*\const{\Gm} \simeq
 \pi'_{s, i!}\wt\xi^*\ev^*j_*\const{\Gm} \simeq
 \cohoc{*}{\AA^{1}, j'_*\const{U}}\simeq 0.
}
\end{equation*}
%since we are calculating the compactly supported sections over $\GG_a$ of a constant sheaf over $\GG_m$ that is $*$-extended over $0$.
\end{proof}

\begin{cor}\label{c:Wh asp} Let $s\in S$.
\begin{enumerate}
\item For any $w\not = 1\in \Wf$, we have $\IC_{w}\star_{s}\Wh_{S}\simeq 0$.
\item The action of $\cH^{f}_{G} \subset \cH^\aff_G$  on $\Sh(\Bun_G(X, S))$ by Hecke modification at $s$  factors through the monoidal functor $\VV:\cH^f_G \to \Vect$ in that for any $\cK\in \cH^{f}_{G}$, there is a canonical isomorphism
\begin{equation*}
\cK\star_{s}\Wh_{S}\simeq \VV(\cK)\otimes \Wh_{S}
\end{equation*}
compatible with the monoidal structures in the obvious sense.
\end{enumerate}
\end{cor}
\begin{proof}
(1) Since any $w\not =1$ can be written as a product of simple reflections $\sigma_i$,  it suffices to show $\IC_{\sigma_i} \star_s \Wh_S \simeq 0$, for the simple reflections $\sigma_i\in  \Wf$. Let $P_{i}$ be the standard parabolic of $G$ whose Levi has only simple root $\a_{i}$. Then 
\begin{equation*}
\IC_{\sigma_{i}}\star_{s}\Wh_{S}\simeq \pi^{*}_{s,P_{i}}\pi_{s,P_{i},!}\Wh_{S}[1]
\end{equation*}
which vanishes by Lemma \ref{l:Avg Wh 0}. Therefore (1) is proved.

Since $\VV:\calH^f_G\to \Vect$ is monoidal and the universal quotient functor with kernel the monoidal ideal $\langle \IC_w \vert w\not = 1 \in \Wf\rangle$, (2) follows from (1). 
\end{proof}

%%%%%%%%%%%%%%%%%%%%%%%%%%%%%%%%%%%%%%%%%%%

\subsubsection{Wakimoto action on Whittaker sheaf} 

%Via the diagonal map $\D:\tcN\to \St_{\dG}$, we have a monoidal functor $\D_{*}: \Perf(\tdN/\dG)\to \Coh^{\dG}(\St_{\dG})$, where  the monoidal structure on $\Perf(\tdN/\dG)$ is given by tensor product. On the other hand, $\Perf(\tdN/\dG)$ is a left module category for $\Coh^{\dG}(\St_{\dG})$ by convolution.

For $s\in S$, recall we have an action of $\Perf(\tdN/\dG)$ on $\Sh_{!}(\Bun_{G}(X,S))$ as the restriction of the affine Hecke action at $s$, see ~\eqref{eq:Waki to Haff} and ~\eqref{eq:PerfN S action}. By acting on $\Wh_{S}$, we  obtain a functor

%For $s\in S$, we have the $ \Coh^{\dG}(\St_{\dG})$-action on $\Sh_{!}(\Bun_{G}(X,S))$
%via the monoidal equivalence $\Phi^{\aff}$ and the Hecke action~\eqref{eq: aff hecke action} denoted by $\star_s$.
%By restricting to a $ \Perf^{\dG}(\tdN)$-action along $\D_*$ and acting on $\Wh_{S}$, 

\begin{equation}\label{eq: act on wh functor}
\xymatrix{
\a_{s}: \Perf(\tdN/\dG) \ar[r] & \Sh_{!}(\Bun_{G}(X,S)) 
}
\end{equation}
such that line bundles go to translations of $\Wh_S$ by Wakimoto operators 
\begin{equation*}
\xymatrix{
\a_{s}( \calO_{\tdN}(\lambda))  = J_\lambda\star_s \Wh_S, & \l\in\L_{T}.
}
\end{equation*}

\begin{prop}\label{p:Hk equiv} The functor $\a_{s}$ intertwines the action of $\Coh^{\dG}(\St_{\dG})$ on the left side and the $\star_{s}$-action of $\cH^{\aff}_{G}$ on the right side under the monoidal equivalence $\Phi^{\aff}$.
\end{prop}

\begin{proof}
By Corollary \ref{c:Wh asp}, the $\star_{s}$-action of $\cH^{\aff}_{G}$ on the object $\Wh_{S}$ factors through the aspherical quotient $\cH^{\asp}_{G}$, or  in other words, we have a functor
\begin{equation*}
\xymatrix{
\a'_{s}: \cH^{\asp}_{G}\ar[r] &  \Sh_{!}(\Bun_{G}(X,S))
}
\end{equation*} 
and a canonical equivalence $\cK\star_{s}\Wh_{S}\simeq \a'_{s}(\ov \cK)$, where $\cK\in \cH^{\aff}_{G}$, and $\ov\cK\in\cH^{\asp}_{G}$ is its image. 
By construction, the functor $\a'_{s}$ is a $\cH^{\aff}_{G}$-module map.

Now we claim that  $\a_{s}$ and $\a'_{s}$ are the same functors under the equivalence  $\Phi^{\asp}$. By the construction in \cite{B},  $\Phi^{\asp}$ is  the composition $\Perf(\tdN/\dG) \to \cH_G^\aff\to  \cH_G^{\asp}$
given by $\calK \mapsto \ov{\Delta_*\calK}$.
Thus we have canonical equivalences
\begin{equation*}
\xymatrix{
\a_{s}(\cK)\simeq (\D_{*}\cK)\star_{s}\Wh_{S} \simeq \a'_{s}(\ov{\D_{*}\cK}), & \calK\in \Perf(\tdN/\dG).
}\end{equation*}

Finally, since $\Phi^{\asp}$ intertwines the $\cH^{\aff}_{G}$-action and the $\Coh^{\dG}(\St_{\dG})$-action via the monoidal equivalence $\Phi^{\aff}$, and $\a'_{s}$ is a $\cH^{\aff}_{G}$-module map, the lemma follows. 
\end{proof}

\subsection{Two point ramification}\label{ss:2 pts}
In this section, we specialize to the case $X=\PP^{1}$ and $S=\{0,\infty\}$. We elaborate on the principle that ``$\Sh_{!}(\Bun_G(\PP^{1}, \{0,\infty\}))$ is the same as $\cH^{\aff}_G$''. 

%Let $\Bun_{G}(\PP^1, \{0, \infty\})$ be the moduli of $G$-bundles on $\PP^1$ with $B$-reductions at $0, \infty\in \PP^1$. 
%Let $\Sh_{!}(\Bun_{G}(\PP^1, \{0, \infty\})$ be the dg derived category of constructible complexes that are extensions by zero off of finite-type substacks. 

We have the two commuting actions $\star_0, \star_\infty$  of $\cH^{\aff}_{G}$ on $\Sh_{!}(\Bun_G(\PP^{1}, \{0,\infty\}))$ by Hecke modifications at respectively $0$, $\infty$. We have the Eisenstein series sheaf $\Eis_{0}$ described in Example~\ref{ex:Eis 0}. 
%given by the extension by zero of $\QQ[-\dim B]$ from the point of $\Bun_{G}(\PP^{1},\{0,\infty\})$ corresponding to the trivial $G$-bundle with identical $B$-reductions at $0$ and $\infty$.
Acting by   $\cH^{\aff}_{G}$ on $\Eis_{0}$   at $0$, we obtain a functor
\begin{equation*}
\xymatrix{
\Phi'_{0,\infty}: \cH^{\aff}_{G} \ar[r] &  \Sh_{!}(\Bun_{G}(\PP^1, \{0, \infty\}) &
\Phi'_{0,\infty}(\cK) =  \cK\star_{0}\Eis_{0}.
}
\end{equation*}

\begin{lemma}\label{l: radon}
$\Phi'_{0,\infty}$ is an equivalence.
\end{lemma}

\begin{proof}
Let us relate $\Phi'_{0,\infty}$ to the Radon transform.

Let $j: \pt/T\incl \Bun_G(\PP^{1}, \{0,\infty\})$ be the open substack where the underlying bundle is trivial and the two Borel reductions at $0,\infty$ are transverse.  Acting by   $\cH^{\aff}_{G}$ on $j_{!}\QQ[-\dim T]$   at $0$
we recover the  Radon transform
\begin{equation*}
\xymatrix{
R: \cH^{\aff}_{G}\ar[r] &  \Sh_{!}(\Bun_{G}(\PP^1, \{0, \infty\}) & 
R(\cK)  = \cK\star_{0} j_{!}\QQ[-\dim T].
}
\end{equation*}
It is well-known that $R$ is an equivalence (see \cite[Corollary 4.1.5 and Section 5.2]{Ytilt} for example). Let $T_{w_{0}*}\in \cH^{f}_{G}$ denote the perverse sheaf which is the $*$-extension of the shifted constant sheaf from the open $B$-orbit in $\cB$. Then $T_{w_{0}*}\star_{0}j_{!}\QQ[-\dim T]\simeq \Eis_{0}$. Therefore
\begin{equation*}
\Phi'_{0,\infty}(\cK)=\cK\star_{0}\Eis_{0}\simeq (\cK\star T_{w_{0}*})\star_{0}j_{!}\QQ[-\dim T] = R(\cK\star T_{w_{0}*}).
\end{equation*}
In other words, $\Phi'_{0,\infty}$ is the composition of first convolution on $\cH^{\aff}$ on the right by $T_{w_{0}*}$ (which is an equivalence with inverse given by  convolution on the right by $T_{w_{0}!}$),  and then the Radon transform $R$ (which is again also an equivalence). This shows that $\Phi'_{0,\infty}$ is an equivalence.
\end{proof}

Let $\Loc_{\dG}(\PP^{1},\{0,\infty\})$ denote the (derived) moduli stack (over $\QQ$) of $\dG$-local systems on $\PP^{1}\setminus \{0,\infty\}$ equipped near $\{0, \infty\}$ with a Borel reduction with unipotent monodromy. Then $\Loc_{\dG}(\PP^{1},\{0,\infty\})$ admits  the  presentation as the substack of $(\tdN\times \tdN)/\dG$  given by imposing on pairs $(\tilde A_0, \tilde A_\infty)\in\tdN\times\tdN$ the equation $A_{0}A_{\infty}=1$ on the underlying group elements  inside of $\dG$. Therefore we have an isomorphism
\begin{equation*}
\xymatrix{
\iota: \St_{\dG}/\dG = (\tdN \times_{\dG} \tdN)/\dG \ar[r]^-\sim &  \Loc_{\dG}(\PP^{1},\{0,\infty\})
&
\iota(\wt A_{0}, \wt A_{\infty}) =(\wt A_{0}, \wt A_{\infty}^{-1})  
}\end{equation*}
where $\wt A_{\infty}^{-1}$ means we invert the group element $A_{\infty}$ while keeping the Borel containing it unchanged. 

Now introduce the equivalence given by the composition of equivalences
\begin{equation*}
\xymatrix{
\Phi_{0,\infty}: \Coh(\Loc_{\dG}(\PP^{1},\{0,\infty\}))\ar[r]^-{\iota^{*}} & \Coh^{\dG}(\St_{\dG})\ar[r]^-{\Phi^{\aff}}  & \cH^{\aff}_{G} 
% \xr{\Phi'_{0,\infty}} \Sh_{!}(\Bun_{G}(\PP^1, \{0, \infty\}).
}
\end{equation*}
$$
\xymatrix{
%\Phi_{0,\infty}: \Coh(\Loc_{\dG}(\PP^{1},\{0,\infty\}))\ar[r]^-{\iota^{*}} & \Coh^{\dG}(\St_{\dG})\ar[r]^-{\eqref{Bez}} & \cH^{\aff}_{G} 
\ar[r]^-{\Phi'_{0,\infty}} & \Sh_{!}(\Bun_{G}(\PP^1, \{0, \infty\}).
}
$$

By construction,  $\Phi_{0,\infty}$ intertwines the $\Coh^{\dG}(\St_{\dG})$-action on $\Coh(\Loc_{\dG}(\PP^{1},\{0,\infty\}))$ by convolution at $0$  and the $\cH^{\aff}_{G}$-action on $\Sh_{!}(\Bun_{G}(\PP^1, \{0, \infty\})$ by the Hecke modifications $\star_{0}$, under the monoidal equivalence $\Phi^{\aff}$.
One can also show that $\Phi_{0,\infty}$ similarly  intertwines the the $\Coh^{\dG}(\St_{\dG})$-action on $\Coh(\Loc_{\dG}(\PP^{1},\{0,\infty\}))$ by convolution at $\infty$ and the  $\cH^{\aff}_{G}$-action on $\Sh_{!}(\Bun_{G}(\PP^1, \{0, \infty\})$ by the
Hecke modifications $\star_{\infty}$. We will not use this statement in the rest of the paper,
 only the following compatibilities.

\begin{lemma}\label{l:2 pts cal}
\begin{enumerate}
\item Let $\D^{-}: \tdN/\dG\to \Loc_{\dG}(\PP^{1},\{0,\infty\})$ be the anti-diagonal $\Delta^-(\wt A_{0}) = (\wt A_{0}, \wt A_{0}^{-1})$.  Then we have
%There is a canonical equivalence
%\begin{equation}
%\xymatrix{
%\Phi_{0,\infty}\circ \D^{-}_{*}  \simeq \alpha_s:\Perf(\tdN/\dG) \ar[r] &  \Sh_{!}(\Bun_{G}(\PP^1, \{0, \infty\}).
%}\end{equation}
 \begin{equation*}
 \xymatrix{
\Phi_{0,\infty} (\Delta^-_* \cO_{\tcN})\simeq \Eis_{0}
}
 \end{equation*}

\item For 
 $\cO_{\Loc}$  the derived structure sheaf of $\Loc_{\dG}(\PP^{1},\{0,\infty\})$, 
 we have 
 \begin{equation*}
 \xymatrix{
 \Phi_{0,\infty}(\cO_{\Loc})\simeq \Wh_{0,\infty}
 }
 \end{equation*}
\end{enumerate}
\end{lemma}
\begin{proof}

(1) Under the equivalence $\Phi^{\aff}$, the monoidal unit $\d \in \calH^\aff_G$, given by the constant sheaf on the closed $I$-orbit in $\Fl_G =  G((t))/I$, 
corresponds to $\D_{*}\cO_{\tdN}\in \Coh^{\dG}(\St_\dG)$ (see Section \ref{ex:Waki}). By construction, we also have $\iota^{*}(\D_{*}\cO_{\tdN})\simeq \D^{-}_{*}\cO_{\tdN}$, and $\Phi'_{0,\infty}(\d)=\d\star_{0}\Eis_{0}\simeq \Eis_{0}$, therefore $\Phi_{0,\infty}(\D^{-}_{*}\cO_{\tcN})\simeq \Eis_{0}$.

(2)  First, we claim that under the equivalence $\Phi^{\aff}$, the derived structure sheaf $\cO_{\St_{\dG}} \in \Coh^{\dG}(\St_\dG)$ corresponds to $q_{!}\Xi[2\dim B]\in \cH^{f}_{G} \subset \cH^\aff_G$ (see Example \ref{ex:Wh 2 pts} for notation). To see this, note that by \cite[Theorem 1 and beginning of Section 6]{B}, there is an equivalence 
\begin{equation*}
\xymatrix{
\Phi_{I^{0}I}: \Sh_{c}(I^{0}\bs \Fl_G)  \ar[r]^-\sim & \Coh^{\dG}(\St'_{\dG})
}
\end{equation*}
where $I^0 = I\times_B N \subset I$, and $\St'_{\dG}=\wt{\dG}\times_{\dG}\tdN$. Moreover, $\Phi_{I^0 I}(\cO_{\St'}) = \Xi$. On the other hand,  the equivalences $\Phi_{I^{0}I}$ and $\Phi^{\aff}$ are compatible:  the forgetful functor $\Forg:\cH^{\aff}_{G}\to \Sh_{c}(I^{0}\bs \Fl_G)$ corresponds to pushforward along $i:\St_{\dG}\incl \St'_{\dG}$. Therefore, $i^{*}$ corresponds to the left adjoint of $\Forg$, and this is given by $q_{!}[2\dim B]$ when restricted to $\Sh_{c}(N\bs \cB)\subset \Sh_{c}(I^{0}\bs \Fl_G)$. Hence $\cO_{\St_{\dG}}\simeq i^{*}\cO_{\St'_{\dG}}$ corresponds to $q_{!}\Xi[2\dim B]$
under the equivalence $\Phi^{\aff}$.

Therefore we have
\begin{equation*}
\Phi_{0,\infty}(\cO_{\Loc})=\Phi'_{0,\infty}(\Phi^{\aff}(\cO_{\St_{\dG}}))\simeq \Phi'_{0,\infty}(q_{!}\Xi[2\dim B])=q_{!}\Xi[2\dim B]\star_{0}\Eis_{0}.
\end{equation*}

Finally, if we view $q_{!}\Xi[2\dim B]\star_{0}\Eis_{0}$  as an object of $\Sh_{c}(B\bs G/B) \xr{u_{!}} \Sh_{!}(\Bun_{G}(\PP^1, \{0, \infty\})$,  it is equivalent to $q_{!}\Xi[2\dim B]\star \d[-\dim B]\simeq q_{!}\Xi[\dim B]$. Thus  $\Phi_{0,\infty}(\cO_{\Loc}) \simeq u_{!}q_{!}\Xi[\dim B]$,
and in turn  $u_{!}q_{!}\Xi[\dim B] \simeq \Wh_{0,\infty}$ as seen in \eqref{Wh 2 pts}.
\end{proof}

%%%%%%%%%%%%%%%%%%%%%%%%%%%%%%%%%%%%%%%%%%%%%%%%%%%%%%%
%%%%%%%%%%%%%%%%%%%%%%%%%%%%%%%%%%%%%%%%%%%%%%%%%%%%%%%
%%%%%%%%%%%%%%%%%%%%%%%%%%%%%%%%%%%%%%%%%%%%%%%%%%%%%%%
%%%%%%%%%%%%%%%%%%%%%%%%%%%%%%%%%%%%%%%%%%%%%%%%%%%%%%%

\section{Automorphic side: $\PP^1$, 3 ramification points, $G = \PGL(2), \SL(2)$}\label{automorphic}

Let $\PP^1 =\Proj \CC[x, y]$ be the projective line with  homogeneous coordinates $[x, y]$ and 
  coordinate $t = y/x$. 
 
 Fix the three points $S = \{0, 1, \infty \} \subset \PP^1$ where the coordinate $t$ takes the respective value.

%%%%%%%%%%%%%%%%%%%%%%%%%%%%%%%%%%%%%%%%%%%%%%%%%%%%

\subsection{Moduli of bundles}\label{ss:mod bun}

%%%%%%%%%%%%%%%%%%%%%%%%%%%%%%%%%%%%%%%%%%%%%%%%%%%%

Let $\Pic(\PP^1) \simeq \Bun_{\GL(1)}(\PP^1)$ denote the Picard stack of line bundles on $\PP^1$, and $\Vect_2(\PP^1) \simeq \Bun_{\GL(2)}(\PP^1)$ the moduli of rank $2$ vector bundles on $\PP^1$.

\subsubsection{$G= \PGL(2)$}\label{sss:Bun PGL2 pts}

By the exact sequence $1 \to \GL(1) \to \GL(2) \to \PGL(2) \to 1$  and the vanishing of the Brauer group of a curve over $\CC$, we have an isomorphism 
 $$
 \xymatrix{
\Vect_2(\PP^1)/\Pic(\PP^1) \ar[r]^-\sim &  \Bun_{\PGL(2)}(\PP^1). 
 }
 $$
Thus we can  represent $\PGL(2)$-bundles  by rank $2$ vector bundles up to
tensoring with a line bundle. 
There is a disjoint union decomposition 
$$
\xymatrix{
\Bun_{\PGL(2)}(\PP^1) = \Bun^{\ov{0}}_{\PGL(2)}(\PP^1)  \coprod \Bun^{\ov{1}}_{\PGL(2)}(\PP^1)
}
$$
given by the parity of the degree of a rank $2$ vector bundle.

%For $k\in \ZZ$, we will always regard the line bundle $\calO_{\PP^1}(k)$ as embedded into the sheaf of meromorphic functions by using the divisor given by $k$ times the point $0\in \PP^1$.

The stack $\Bun_{\PGL(2)}(\PP^1, S)$ is the moduli of $\PGL(2)$-bundles on $\PP^1$ with $B$-reductions at the points
of $S = \{0, 1, \infty\}$.
 We can  represent objects of $\Bun_{\PGL(2)}(\PP^1, S)$ by $(\cE, \{\ell_{s}\}_{s\in S})$, where $\cE$ is a rank $2$ vector bundles on $\PP^{1}$ up to
tensoring with a line bundle, and $\ell_{s}$ is a line in the fiber $\cE_{s}$ for each  $s\in S$.

Let us list the isomorphism classes of objects of $\Bun_{\PGL(2)}(\PP^1, S)$. For each isomorphism class of $\cE\in \Bun_{\PGL(2)}(\PP^{1})$, we describe the poset of points in $\Bun_{\PGL(2)}(\PP^1, S)$ over it, where an arrow $x\to y$ means $y$ lies in the closure of $x$.

\begin{enumerate}

\item $\calE =  \calO_{\PP^1} \oplus \calO_{\PP^1}$,
$\Aut(\calE) \simeq \PGL(2)$, with  the
poset of configurations of lines
$$
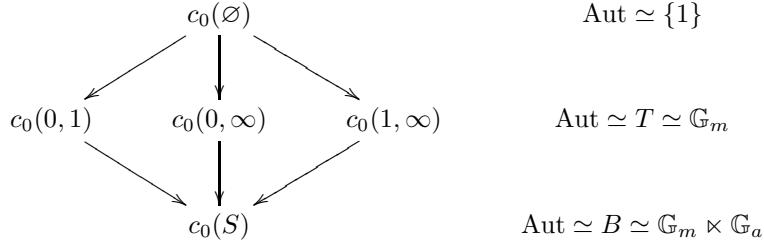
\xymatrix{
& \ar[dl]\ar[d] c_0(\vn) \ar[dr] &&  \Aut \simeq \{1\} \\
\ar[dr] c_0(0, 1)   &\ar[d] c_0(0, \infty)  & \ar[dl] c_0(1, \infty) & \Aut \simeq T\simeq \GG_m   \\
&  c_0(S) & & \Aut \simeq B\simeq \GG_m \ltimes \GG_a
}
$$
where $c_0(R)$ denotes where the lines $\ell_r$ coincide, for $r\in R \subset S$.

%
%\begin{enumerate}
%
%\item $c_0(S) = \{ \mbox{$\ell_0, \ell_1, \ell_\infty$ distinct}\}, \Aut \simeq \{1\}$
%
%\item $c_0(0) = \{\ell_0 \not = \ell_1 = \ell_\infty\}, \Aut \simeq T $
%
%\item $c_0(1) = \{\ell_1 \not = \ell_0 = \ell_\infty\}, \Aut \simeq T $
%
%\item $c_0(\infty) = \{\ell_\infty \not = \ell_0 = \ell_1\}, \Aut \simeq T $
%
%\item $c_0 = \{\ell_0  = \ell_1 = \ell_\infty\}, \Aut \simeq B $
%
%\end{enumerate}

%
%$$
%\xymatrix{
%& \ar[dl]\ar[d] \{ \mbox{all lines distinct}\}, \Aut \simeq \{1\} \ar[dr] &\\
%\ar[dr] \{\ell_0 = \ell_1 \not = \ell_\infty\}, \Aut \simeq T &\ar[d] \{\ell_1 = \ell_\infty \not = \ell_0\}, \Aut \simeq T & \ar[dl]\{\ell_\infty = \ell_0 \not = \ell_1\}, \Aut \simeq T\\
%& \{\ell_0 = \ell_1  = \ell_\infty\}, \Aut \simeq B&
%}
%$$

\item $\calE =  \calO_{\PP^1}(1) \oplus \calO_{\PP^1}$, $ \Aut(\calE) \simeq  \GG_m \ltimes \GG_a^2$, with the
poset of  configurations of lines

$$
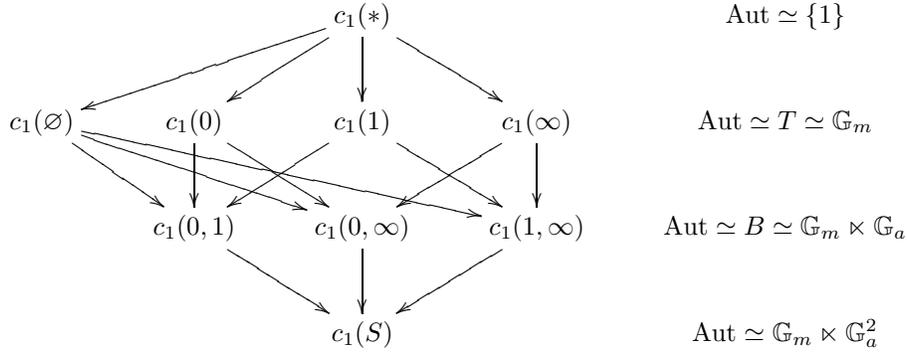
\xymatrix{
& & \ar[dll]\ar[dl]\ar[d] c_1(*) \ar[dr] &&  \Aut \simeq \{1\} \\
\ar[dr]\ar[drr]\ar[drrr]c_1(\vn) & \ar[dr]\ar[d] c_1(0)   &\ar[dr]\ar[dl] c_1(1)  & \ar[dl]\ar[d] c_1(\infty)&  \Aut \simeq T \simeq \GG_m  \\
& \ar[dr] c_1(0, 1)   &\ar[d] c_1(0, \infty)  & \ar[dl] c_1(1, \infty) &  \Aut \simeq B \simeq \GG_m \ltimes \GG_a \\
& &  c_1(S) & &  \Aut \simeq   \GG_m \ltimes \GG_a^2
}
$$
where $c_1(R)$ denotes where the lines $\ell_r$ lie in the summand $\calO_{\PP^1}(1)$,  for $r\in R \subset S$,
and in the summand $\calO_{\PP^1}$,  for $r\not \in R \subset S$. The generic configuration  $ c_1(*)$
denotes where none of the lines $\ell_s$ lie in $\calO_{\PP^1}(1)$, for $s\in S$, and also  they do not all lie in the image of any map $\cO_{\PP^{1}}\to \cE$ (as in the configuration $c_1(\vn)$).

\item $k\geq 2$, $\calE =  \calO_{\PP^1}(k) \oplus \calO_{\PP^1}$,  and  we have an exact  sequence
$$
\xymatrix{
1 \ar[r] & \GG_a^{k-2} \ar[r] & \Aut(\calE) \ar[r]^-{\ev_S} &  \GG_m \ltimes (\GG_a)^S \ar[r] & 1
}
$$
The poset of  configurations of lines is the product
$$
\xymatrix{
\prod_{s \in S} (\{\ell_s \subset  \calO_{\PP^1}\} \ar[r] &  \{\ell_s \subset \calO_{\PP^1}(k)\} )
}
$$
with automorphisms 
$$
\xymatrix{
1 \ar[r] & \GG_a^{k-2} \ar[r] & \Aut \ar[r]^-{\ev_S} &  \GG_m \ltimes (\GG_a)^{R} \ar[r] & 1
& R = \{s\in S \mid \ell_s \subset \calO_{\PP^1}(k)\}.
}
$$
Let us denote by $c_k(R)$ where
  the lines $\ell_r$ lie in the summand $\calO_{\PP^1}(k)$,  for $r\in R \subset S$,
  and in the summand $\calO_{\PP^1}$,  for $r\not \in R \subset S$.

\end{enumerate}

\subsubsection{$G= \SL(2)$}

Note that $1 \to \SL(2) \to \GL(2) \to \GL(1) \to 1$ allows us to represent $\SL(2)$-bundles  by rank $2$ vector bundles 
with trivialized determinant.

Let $\Bun_{\SL(2)}(\PP^1, S)$ denote the moduli of $\SL(2)$-bundles on $\PP^1$ with $B$-reductions at the points
of $S = \{0, 1, \infty\}$.
   We can  represent objects of $\Bun_{\SL(2)}(\PP^1, S)$ by  $(\cE,\tau, \{\ell_{s}\}_{s\in S})$ where $\cE$ is a rank $2$ vector bundle on $\PP^{1}$, $\tau:\cO_{\PP^{1}}\isom \det(\cE)$, and $\ell_{s}$ is a line in the fiber $\cE_{s}$ for $s\in S$.

Let us list the isomorphism classes of objects of $\Bun_{\SL(2)}(\PP^1, S)$ according to the isomorphism type of the underlying rank 2 bundles.

\begin{enumerate}

\item $\calE =  \calO_{\PP^1} \oplus \calO_{\PP^1}$
$\Aut(\calE) \simeq \SL(2)$, with the poset of configurations of lines
$$
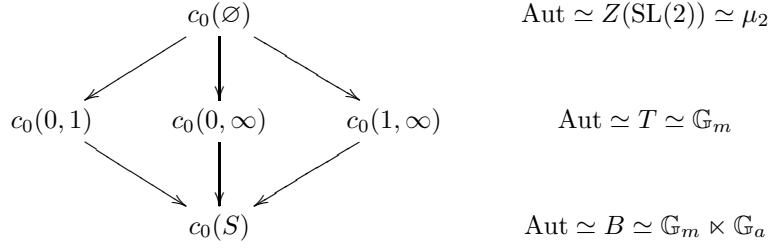
\xymatrix{
& \ar[dl]\ar[d] c_0(\vn) \ar[dr] &&  \Aut \simeq Z(\SL(2)) \simeq \mu_2 \\
\ar[dr] c_0(0, 1)   &\ar[d] c_0(0, \infty)  & \ar[dl] c_0(1, \infty) & \Aut \simeq T\simeq \GG_m   \\
&  c_0(S) & & \Aut \simeq B\simeq \GG_m \ltimes \GG_a
}
$$
where $c_0(R)$ denotes where the lines $\ell_r$ coincide, for $r\in R \subset S$.

\item $k\geq 1$, $\calE =  \calO_{\PP^1}(k) \oplus \calO_{\PP^1}(-k)$,  and we have an exact sequence
$$
\xymatrix{
1 \ar[r] & \GG_a^{2k-2} \ar[r] & \Aut(\calE) \ar[r]^-{\ev_S} &   \GG_m \ltimes (\GG_a)^S \ar[r] & 1
}
$$
The poset of  of configurations of lines is product
$$
\xymatrix{
\prod_{s \in S} (\{\ell_s \subset \calO_{\PP^1}(-k)\} \ar[r] &  \{\ell_s \subset \calO_{\PP^1}(k)\} )
}
$$
with automorphisms
$$
\xymatrix{
1 \ar[r] & \GG_a^{2k-2} \ar[r] & \Aut \ar[r]^-{\ev_S} &  \GG_m \ltimes (\GG_a)^R \ar[r] & 1
& R= \{s\in S \mid \ell_s \subset \calO_{\PP^1}(k)\}.
}
$$
Let us denote by $c_{2k}(R)$ where
  the lines $\ell_r$ lie in the  summand $\calO_{\PP^1}(k)$,  for $r\in R \subset S$,
  and  in the summand $\calO_{\PP^1}(-k)$,  for $r\not \in R \subset S$.

%\marginpar{Changed $c_{k}(R)$ to $c_{2k}(R)$, to be consistent with the $\PGL(2)$ notation.}

\end{enumerate}

%%%%%%%%%%%%%%%%%%%%%%%%%%%%%%%%%%%%%%%%%%%%%%%%%%%%

\subsection{Coarse symmetries}

\subsubsection{Atkin-Lehner modifications for $G= \PGL(2)$}

Atkin-Lehner modifications provide involutions exchanging the two connected components of $\Bun_{\PGL(2)}(\PP^1, S)$.
For $r\in S$, define the involution
$$
\xymatrix{
AL_r:\Bun_{\PGL(2)}(\PP^1, S) \ar[r] &\Bun_{\PGL(2)}(\PP^1, S) & AL_r(\calE, \{\ell_s\}_{s\in S})  = (\calE', \{\ell'_s\}_{s\in S})  
}
$$
where $\calE' \subset \calE$ is the  lower modification at $r\in \PP^1$ so that $\ell_r \subset \calE_r$ factors through $\calE'_r \subset \calE_r$,  the resulting map $\calE\to \calE_r$ induces  an isomorphism 
$$
\xymatrix{
\calE/\calE' \ar[r]^-\sim &  \calE_r/\ell_r
}$$
and $\ell'_r \subset \calE'_r$ is the image of the map $\calE(- r)_r  \to \calE'_r$,
and the other lines are unchanged $\ell'_s = \ell_s \subset \calE'_s = \calE_s$, for $s\not = r\in S$ .
Note the involution $AL_r$ exchanges the open points
$$
\xymatrix{
c_0(\vn) \ar@{<->}[r] & c_1(*).
}
$$

The Atkin-Lehner modifications generate a group $(\ZZ/2\ZZ)^{S}$  of order $8$. 
For $R\subset S$ of even size, the Atkin-Lehner modifications
$AL_{R}=\prod_{r\in R} AL_{r}$ preserve the two connected components, and generate a subgroup 
$(\ZZ/2\ZZ)^{S,ev}$ of order $4$.

%%%%%%%%%%%%%%%%%%%%%%%%%%%%%%%%%%%%%%%%%%%%%%%%%%%%

\subsubsection{Central automorphisms for $G= \SL(2)$}

The inclusion $\mu_2 \simeq Z(\SL(2)) \subset \SL(2)$ of the center induces an automorphisms of the identity functor 
 of $\Bun_{\SL(2)}(\PP^1, S)$.

%%%%%%%%%%%%%%%%%%%%%%%%%%%%%%%%%%%%%%%%%%%%%%%%%%%%

\subsection{Constructible sheaves}

%
%Given a stack $Z$ with discretely many isomorphism classes of objects, let us write $\Sh(Z)$ to denote the dg derived category of all complexes
%on $Z$. We will abuse terminology and use the term sheaves
%to refer to its objects.
%
%
%
%
% Let us write $\Sh_!(Z) \subset \Sh(Z)$ to denote the full dg subcategory of constructible complexes
% that are extensions by zero off of finite-type substacks. 
%%
% Let us write $\Sh_\dag(Z) \subset \Sh(Z)$ to denote the full dg subcategory of compact objects,
% or equivalently, objects  that are extensions by zero off of finite-type substacks,
% and whose stalks are torsion over the equivariant cohomology of automorphisms.
%
%
%
%Let us record some basic features of $\Sh_!(\Bun_{\PGL(2)}(\PP^1, S))$
%and $\Sh_!(\Bun_{\SL(2)}(\PP^1, S))$. 
%By a basis of objects, we will mean a minimal set of objects such that every object
%is a finite complex built out of shifts of the basis.

%%%%%%%%%%%%%%%%%%%%%%%%%%%%%%%%%%%%%%%%%%%%%%%%%%%%

\subsubsection{$G= \PGL(2)$}

Recall the points of  $\Bun_{\PGL(2)}(\PP^1, S)$ are discretely parameterized and
their automorphism groups are connected. We have the corresponding basis of objects of $\Sh_!(\Bun_{\PGL(2)}(\PP^1, S))$
given by the respective extensions by zero of constant sheaves:
\begin{equation*}
\xymatrix{
\calF_0(R) = j_! \const{c_0(R)} & R = \vn, \{0,1\}, \{0, \infty\}, \{1, \infty\}, S;
}
\end{equation*}
\begin{equation*}
\xymatrix{
\calF_1(R) = j_! \const{c_1(R)} & R \subset S \mbox{ or } R = *;
}
\end{equation*}
\begin{equation*}
\xymatrix{
\calF_k(R) = j_! \const{c_k(R)} & R \subset S, k\geq 2.
}\end{equation*}

We will also use the following notation
\begin{equation}\label{eq:IC k R}
\IC_{k}(R):=j_{!*}\const{c_{k}(R)}[-\dim \Aut(c_{k}(R))]
\end{equation}
to denote the IC-sheaf of the closure of $c_{k}(R)$. 

The decomposition 
into connected components
$$
\xymatrix{
\Bun_{\PGL(2)}(\PP^1, S) = \Bun^{\ov{0}}_{\PGL(2)}(\PP^1, S)  \coprod \Bun^{\ov{1}}_{\PGL(2)}(\PP^1, S)
}
$$
provides a direct sum decomposition
$$
\xymatrix{
\Sh_!(\Bun_{\PGL(2)}(\PP^1, S)) \simeq \Sh_!(\Bun^{\ov 0}_{\PGL(2)}(\PP^1, S))  \oplus \Sh_!(\Bun^{\ov 1}_{\PGL(2)}(\PP^1, S))
}
$$
The above basis of objects $\cF_{k}(R)$ belongs to $\Sh_!(\Bun^{\ov k}_{\PGL(2)}(\PP^1, S))$, where $\ov k=k\mod 2$.
For $r\in S$, note the Atkin-Lehner involution $AL_r$ exchanges the basis elements
 $$
\xymatrix{
\calF_0(\vn) \ar@{<->}[r] & \calF_1(*).
}
$$

\subsubsection{Whittaker sheaf for $G=\PGL(2)$} Let us record the form of the Whittaker sheaf.
Consider the open substacks of the odd component
\begin{equation*}
\xymatrix{
 c_1(*)    \ar@{^(->}[r]^-j & 
  c_1(*) \cup c_1(\vn)   \ar@{^(->}[r]^-i & 
\Bun^{\ov 1}_{\PGL(2)}(\PP^1, S)
}\end{equation*}
classifying respectively bundles $\calE \simeq \calO_{\PP^1}(1) \oplus \calO_{\PP^1}$ with generic lines $\ell_0, \ell_1, \ell_\infty$,
and more generally, lines $\ell_0, \ell_1, \ell_\infty$ with none
contained within $\calO_{\PP^1}(1)$. Then the Whittaker sheaf 
is given by
\begin{equation*}
\xymatrix{
\Wh_{S} = i_! j_* \const{ c_1(*)} \in  \Sh_!(\Bun^{\ov 1}_{\PGL(2)}(\PP^1, S)).
}\end{equation*}
Note the twist in the definition of  $\Wh_{S}$ disappears because $d_{S}=-1$ in this situation.

\subsubsection{$G= \SL(2)$}\label{sss:auto SL2}

Recall the points of  $\Bun_{\SL(2)}(\PP^1, S)$ are discretely parameterized and
their automorphism groups are connected except for the configuration $c_0(\vn)$ with $\Aut \simeq  Z(\SL(2)) \simeq \mu_2$. 
Let $\QQ^{\alt}_{c_{0}(\vn)}$ denote the rank one local system on $c_{0}(\vn)$ where the automorphism group $\mu_{2}$ acts by the sign character. We have the corresponding basis of objects of $\Sh_!(\Bun_{\PGL(2)}(\PP^1, S))$
given by the respective extensions by zero of constant sheaves and one additional sheaf $\QQ^{\alt}_{c_{0}(\vn)}$:
\begin{equation*}
\xymatrix{
\calF_0(R) = j_! \const{c_0(R)} & R = \vn, \{0,1\}, \{0, \infty\}, \{1, \infty\}, S;
}\end{equation*}
\begin{equation*}
\xymatrix{
\calF_0(\vn)^{\alt} = j_! \QQ^{\alt}_{c_0(\vn)}; 
}\end{equation*}
\begin{equation*}
\xymatrix{
\calF_{2k}(R) = j_! \const{c_{2k}(R)} & R \subset S, k\geq 1.
}\end{equation*}

%\marginpar{Changed $\cF_{k}(R)$ to $\cF_{2k}(R)$.}

The canonical automorphisms of the identity functor of $\Bun_{\SL(2)}(\PP^1, S)$ given by $\mu_2 \simeq Z(\SL(2))$ 
provides a direct sum decomposition
$$
\xymatrix{
\Sh_!(\Bun_{\SL(2)}(\PP^1, S)) \simeq \Sh^{\triv}_!(\Bun_{\SL(2)}(\PP^1, S))  \oplus \Sh^{\alt}_!(\Bun_{\SL(2)}(\PP^1, S))
}
$$
determined by whether the induced action of $\mu_2 \simeq Z(\SL(2))$ on sheaves is trivial or alternating.

The second summand admits an equivalence
$$
\xymatrix{
\Sh^{\alt}_!(\Bun_{\SL(2)}(\PP^1, S))\simeq \Vect.
}
$$
since all of its objects are finite complexes built out of shifts of  $\calF_0(\vn)^{\alt}$ whose automorphisms are scalars.

%%%%%%%%%%%%%%%%%%%%%%%%%%%%%%%%%%%%%%%%%%%%%%%%%%%%

\subsubsection{Relation between $G= \PGL(2)$ and $G=\SL(2)$}\label{sss:cons shv rel}

The natural map $\SL(2) \to \PGL(2)$ induces a map 
\begin{equation*}
\xymatrix{
p:\Bun_{\SL(2)}(\PP^1, S)\ar[r] &  \Bun^{\ov 0}_{\PGL(2)}(\PP^1, S)\subset \Bun_{\PGL(2)}(\PP^1, S)
}
\end{equation*}
which sends $c_{2k}(R)\in \Bun_{\SL(2)}(\PP^1, S)$ to the same-named point $c_{2k}(R)$ in $\Bun^{\ov 0}_{\PGL(2)}(\PP^1, S)$, for any $k\ge0$ and $R\subset S$.
%whose base-change to $\Bun_{\PGL(2)}(\PP^1, S)_{ev}$ is a $B\mu_2$-torsor, or in other words a $\mu_2$-gerbe.

Pullback provides an equivalence
\begin{equation}\label{eq:PGL ev SL2}
\xymatrix{
p^*:\Sh_!(\Bun^{\ov{0}}_{\PGL(2)}(\PP^1, S))\ar[r]^-\sim &  \Sh^{\triv}_!(\Bun_{\SL(2)}(\PP^1, S))
}
\end{equation}
that acts on the above basis by
\begin{equation*}
\xymatrix{
p^*(\calF_0(R)) \simeq \calF_0(R) & R = \vn, \{0,1\}, \{0, \infty\}, \{1, \infty\}, S;
}
\end{equation*}
\begin{equation*}
\xymatrix{
p^*(\calF_{2k}(R)) \simeq \calF_{2k}(R) & R\subset S, k\geq 1.
}
\end{equation*}

Thus using the  prior decompositions and Atkin-Lehner involutions, we see that to understand any of the above categories, it suffices to understand for example  $Sh_!(\Bun^{\ov{1}}_{\PGL(2)}(\PP^1, S))$. We prefer the odd component of $\Bun_{\PGL(2)}(\PP^1, S)$ since it supports the Whittaker sheaf.

%%%%%%%%%%%%%%%%%%%%%%%%%%%%%%%%%%%%%%%%%%%%%%%%%%%%

%%%%%%%%%%%%%%%%%%%%%%%%%%%%%%%%%%%%%%%%%%%%%%%%%%%%
%%%%%%%%%%%%%%%%%%%%%%%%%%%%%%%%%%%%%%%%%%%%%%%%%%%%
%%%%%%%%%%%%%%%%%%%%%%%%%%%%%%%%%%%%%%%%%%%%%%%%%%%%

\section{Spectral side: $\PP^1$, 3 ramification points, $\dG = \SL(2), \PGL(2)$}

Continue with $\PP^1 =\Proj (k[x, y])$  the projective line with  homogeneous coordinates $[x, y]$ and 
  coordinate $t = y/x$, and
  the three points $S = \{0, 1, \infty \} \subset \PP^1$ where the coordinate $t$ takes the respective value.

\subsection{Moduli of local systems}

%%%%%%%%%%%%%%%%%%%%%%%%%%%%%%%%%%%%%%%%%%%%%%%%%%%%

\subsubsection{General definition}
We start with a general reductive group $\dG$ over $\QQ$. Let $\Loc_{\dG}(\PP^{1}, S)$ be the Betti moduli of $\dG$-local systems on $\PP^1 \setminus S$ with $B^\vee$-reductions near $S$ with trivial induced $T^\vee$-monodromy.
By choosing a point $u_{0}$ in $\PP^1 \setminus S$ and a based loop $\gamma_{s}$ around $s\in S$ for each $s\in S$ such that $\gamma_{0}\gamma_{1}\gamma_{\infty}=1$ in $\pi_{1}(\PP^1 \setminus S, u_{0})$, we obtain the presentation
$$
\xymatrix{
\Loc_{\dG}(\PP^1, S) \simeq (\tilde\calN^\vee)^{S,\prod=1} /\dG.
}
$$
Here, $(\tilde\calN^\vee)^{S,\prod=1}$ is the derived fiber of $1$ of the map
\begin{equation}\label{mult S}
\xymatrix{(\tdN)^{S}\ar[r]^-{\mu^{S}} & (\dcN)^{S}\ar[r]^-{\mult} & \dG} 
\end{equation}
and the map ``mult'' takes $(A_{0},A_{1}, A_{\infty})$ to $A_{0}A_{1}A_{\infty}$.

\subsubsection{$\dG= \SL(2)$}\label{sss:Loc comp}
In this case, $(\tilde\calN^\vee)^{S,\prod=1}$ is the derived subscheme of $(\tdN)^{S}$ classifying triples of pairs $(A_s,\ell_s)_{s\in S}$ consisting of a matrix $A_s\in \SL(2)$ and an eigenline $A_s(\ell_s) \subset \ell_s$ with trivial eigenvalue $A_s|_{\ell_s} = 1$, and the matrices satisfy the equation $A_0 A_1 A_\infty = 1$ inside of $\SL(2)$.

%where $\prod$ denotes the product of the group elements, and the equation $\prod = 1$ is imposed inside of $\SL(2)$.
%Thus a point of $\Loc_{\SL(2)}(\PP^1, S)$ is a triple of pairs 

To write explicit local equations for $(\tilde \calN^\vee)^{S, \prod = 1}$, we may apply the $\SL(2)$-symmetry to assume without loss of generality that $\ell_\infty = [1:0]$, $\ell_{0}=[1:x]$ and $\ell_{1}=[1:y]$. Then the  three matrices take the form
\begin{equation*}
\xymatrix{
A_0=\mat{1-ax}{a}{-ax^{2}}{1+ax} &  A_1=\mat{1-by}{b}{-by^{2}}{1+by} 
}
\end{equation*}
\begin{equation*}
\xymatrix{
A_\infty^{-1}= A_0 A_1 = \mat{1-ax-by+aby(x-y)}{a+b-ab(x-y)}{-ax^{2}-by^{2}+abxy(x-y)}{1+ax+by-abx(x-y)}
}
\end{equation*}
such that $A_\infty$ is of the form $\mat{1}{*}{0}{1}$. Since $\det(A_\infty)=1$, we need only impose the equations
\begin{equation*}
\xymatrix{
1-ax-by+aby(x-y)=1,
&
-ax^{2}-by^{2}+abxy(x-y)=0.
}\end{equation*}
These in turn are equivalent to the equations
\begin{equation*}
\xymatrix{
ax+by=0, & ax^{2}+by^{2}=0.
}\end{equation*}

We conclude that $(\tilde \calN^\vee)^{S, \prod = 1}$ is a lci classical scheme (i.e., not derived) with 5 irreducible components:
\begin{enumerate}
\item $A_0=A_1=A_\infty=1$. This component is isomorphic to $\PP^{1}\times\PP^{1}\times\PP^{1}$. Local equation: 
$$a=b=0.$$
\item  $A_0=1$ (hence $A_1=A_\infty^{-1}$) and $\ell_{1}=\ell_{\infty}$. This component is isomorphic to $\PP^{1}\times\tdN$. Local equation:
$$a=0, y=0.$$
\item  $A_1=1$  (hence $A_0=A_\infty^{-1}$) and $\ell_{0}=\ell_{\infty}$. This component is isomorphic to $\PP^{1}\times\tdN$. Local equation:
$$b=0, x=0.$$
\item  $A_\infty=1$  (hence $A_0=A_1^{-1}$) and $\ell_{0}=\ell_{1}$. This component is isomorphic to $\PP^{1}\times\tdN$. Local equation:
$$a+b=0, x=y.$$
\item  $A_0, A_1, A_\infty$ all lie in a single Borel. Note this does not mean that $\ell_{0},\ell_{1}$, $\ell_{\infty}$ are the same; in fact, this component is non-reduced since $A_0,A_1,A_\infty$ fix  $\ell_{0},\ell_{1}$, $\ell_{\infty}$ respectively.  
Local equation:
$$x^{2}=0, y^{2}=0, xy=0, ax+by=0.$$ 
Note for $a,b$ not both zero, there is a unique infinitesimal direction for $(x,y)$.  The reduced structure of this component is isomorphic to the total space of $\cO(-2)\oplus\cO(-2)$ over $\PP^{1}$, and we denote it as $\tdN_{\Delta}$.
\end{enumerate}

If we view $(\tdN)^{S}$ as the cotangent bundle of $(\PP^{1})^{S}$, the 5 components listed above, after passing to reduced structures, are exactly the conormal bundles of various partial diagonals in $(\PP^{1})^{S}$. For this reason, we introduce the following notation. For a subset $R\subset S$ with $\#R\ne1$, we denote by $\Delta_{R}$ the partial diagonal of $(\PP^{1})^{S}$ where the $R$-components are equal. For example, $\Delta_{\varnothing}=(\PP^{1})^{S}$. Let $\L_{R}\subset T^{*}(\PP^{1})^{S}\simeq (\tdN)^{S}$ be the conormal bundle of $\Delta_{R}$. Then the reduced structure of the 5 components of $((\tdN)^{S})^{\prod=1}$ are, in the order listed above, $\L_{\varnothing}, \L_{1,\infty}, \L_{0,\infty}, \L_{0,1}$ and $\L_{S}=\tdN_{\Delta}$.

%%%%%%%%%%%%%%%%%%%%%%%%%%%%%%%

%\marginpar{Use superscript $\ov 0$ and $\ov 1$ to denote the two components of $\Loc_{\PGL(2)}$ instead of subscripts $even/odd$.} 

\subsubsection{$\dG=\PGL(2)$} The stack $\Loc_{\PGL(2)}(\PP^{1}, S)$ has two connected components. In fact, for $\dG=\PGL(2)$, the Springer resolution $\tdN$ and the unipotent variety $\dcN$ are the same as those of $\SL(2)$. Therefore the map ``mult'' in ~\eqref{mult S} factorizes as
\begin{equation*}
\xymatrix{(\dcN)^{S}\ar[r]^-{\wt\mult} & \SL(2)\ar[r] & \PGL(2).}
\end{equation*}
Hence, according to whether the product of three elements in $\dcN$ is $1$ or $-1$ in $\SL(2)$, we have a decomposition of $\Loc_{\PGL(2)}(\PP^{1}, S)$
\begin{equation}\label{Loc 2 cpts}
\xymatrix{
\Loc_{\PGL(2)}(\PP^1, S) = \Loc^{\ov 0}_{\PGL(2)}(\PP^1, S)  \coprod \Loc^{\ov 1}_{\PGL(2)}(\PP^1, S)
}\end{equation}
where
\begin{equation*}
\xymatrix{\Loc^{\ov 0}_{\PGL(2)}(\PP^1, S)=(\tdN)^{S, \wt \prod=1}/\PGL(2)}
\end{equation*}
\begin{equation*}
\xymatrix{\Loc^{\ov 1}_{\PGL(2)}(\PP^1, S)=(\tdN)^{S, \wt \prod=-1}/\PGL(2).}
\end{equation*}

The natural map $\SL(2) \to \PGL(2)$ induces a map 
\begin{equation}\label{eq:Loc p}
\xymatrix{
p:\Loc_{\SL(2)}(\PP^1, S)\ar[r] &  \Loc^{\ov 0}_{\PGL(2)}(\PP^1, S)\subset \Loc_{\PGL(2)}(\PP^1, S)
}
\end{equation}
which in turn induces an equivalence
\begin{equation*}
\xymatrix{
\Loc_{\SL(2)}(\PP^1, S)\ar[r]^-\sim &  \Loc^{\ov 0}_{\PGL(2)}(\PP^1, S) \times_{\pt/\PGL(2)} (\pt/\SL(2))
}
\end{equation*}

The odd component of $\Loc_{\PGL(2)}(\PP^1, S)$ actually reduces to a single point.

%\marginpar{Added following lemma and remark.}

\begin{lemma}\label{l:Loc odd} The derived scheme $(\tdN)^{S, \wt \prod=-1}$ is a trivial torsor for $\PGL(2)$. In particular, 
$$\Loc^{\ov 1}_{\PGL(2)}(\PP^1, S)\cong \Spec \QQ.$$
\end{lemma}
\begin{proof}
Let $(A_{s},\ell_{s})_{s\in S}$ be a point of $(\tdN)^{S, \wt \prod=-1}$. We view $A_{s}$ as unipotent elements in $\SL(2)$, then $A_{0}A_{1}A_{\infty}=-1\in \SL(2)$. It is easy to see that none of $A_{s}$ can be $1$, hence each line $\ell_{s}$ is determined by $A_{s}$. It is also easy to see that no two lines are equal, hence using the $\PGL(2)$-action we may arrange $\ell_{0}=[1:0], \ell_{1}=[0:1]$ and using the remaining $\dT$-conjugacy we may arrange uniquely 
\begin{equation*}
\xymatrix{A_{0}=\mat{1}{1}{0}{1} & A_{1}=\mat{1}{0}{c}{1}. }
\end{equation*}
Then we have
\begin{equation*}
\xymatrix{A_{\infty}=-A_{1}^{-1}A_{0}^{-1}=\mat{-1}{1}{c}{-1-c}}
\end{equation*}
which is unipotent if and only if $c=-4$. This shows that $(\tdN)^{S, \wt \prod=-1}$ is a torsor for $\PGL(2)$ with a rational point.
\end{proof}

\begin{remark}\label{r:Legendre} The unique point in $\Loc^{\ov 1}_{\PGL(2)}(\PP^1, S)$ corresponds to a rank 2 local system on $\PP^{1}\setminus S$ with nontrivial unipotent monodromy at $0$ and $1$, and monodromy with a single Jordan block of eigenvalue $-1$ at $\infty$. This local system arises from the universal Tate module of the Legendre family of elliptic curves over $\PP^{1}\setminus\{0,1,\infty\}$ given by $y^{2}=x(x-1)(x-t), t\in \PP^{1}\setminus\{0,1,\infty\}$.
\end{remark}

%%%%%%%%%%%%%%%%%%%%%%%%%%%%%%%%%%%%%%%%%%%%%%%%%

%\marginpar{Now this is a separated subsection.}

\subsection{Comparison with linear and de Rham moduli} In this subsection $\dG=\SL(2)$. We will show that $\Loc_{\SL(2)}(\PP^{1}, S)$ is isomorphic to its linearized version and its de Rham version which traditionally appears in the formulation of the Geometric Langlands correspondence.

\subsubsection{Linearized version} 

Let $\Lin_{\SL(2)}(\PP^1, S)$ denote the linearized version of $\Loc_{\SL(2)}(\PP^1, S)$
defined by the presentation
$$
\xymatrix{
\Lin_{\SL(2)}(\PP^1, S) := T^{*}((\PP^{1})^{S}/\SL(2)) = (\tilde\calN^\vee)^{S, \sum=0} /\SL(2)
}
$$
where we regard $\tdN$ as the Springer resolution of the nilpotent cone in $\frg^{\vee}=\sl(2)$, and
 impose that the sum of the Lie algebra elements be zero.
Thus a point of $\Lin_{\SL(2)}(\PP^1, S)$ is a triple of pairs $ (B_s,\ell_s)_{s\in S}$ consisting of a matrix $B_s\in \sl(2)$ and an eigenline $B_s(\ell_s) \subset \ell_s$ with trivial eigenvalue $B_s|_{\ell_s} = 0$, and the matrices satisfy the equation $B_0 + B_1 +  B_\infty = 0$ inside of $\sl(2)$.

The local equations for $(\tilde\calN^\vee)^{S, \sum=0}$ are exactly the same as those derived above for
$(\tilde\calN^\vee)^{S, \prod=1}$  except  now $B_0, B_1$, $B_\infty$ are nilpotent rather than unipotent matrices
\begin{equation*}
\xymatrix{
B_0=\mat{-ax}{a}{-ax^{2}}{ax} & B_1=\mat{-by}{b}{-by^{2}}{by} 
}
\end{equation*}
\begin{equation*}
\xymatrix{
-B_\infty=B_0+B_1=\mat{-ax-by}{a+b}{-ax^{2}-by^{2}}{+ax+by}
}
\end{equation*}
with the requirement that $B_\infty$ is of the form $\mat{0}{*}{0}{0}$ imposing the equations
\begin{equation*}
\xymatrix{
ax+by=0 &  ax^{2}+by^{2}=0
}
\end{equation*}

Thus we can construct an $\SL(2)$-equivariant isomorphism
\begin{equation*}
\xymatrix{
 (\tilde\calN^\vee)^{S, \prod=1}  \ar[r]^-\sim & (\tilde\calN^\vee)^{S, \sum=0} 
}
\end{equation*}
by the assignment
\begin{equation*}
\xymatrix{
(A_0, A_1, A_\infty,  \ell_0, \ell_1, \ell_\infty)
\ar@{|->}[r] &  (A_0-1, A_1-1, 2-A_0-A_1, \ell_{0}, \ell_{1}, \ell_{\infty})
}\end{equation*}
Note that $A_\infty-1\neq 2-A_0-A_1$ as they differ in local coordinates by
\begin{equation*}
\mat{0}{0}{ab(x-y)}{0}
\end{equation*}
 though nevertheless $(2-A_0-A_1)\ell_{\infty}=0$.

We could just as well choose either of  the alternative isomorphisms given by the assignments
 \begin{equation*}
\xymatrix{
(A_0, A_1, A_\infty,  \ell_0, \ell_1, \ell_\infty)
\ar@{|->}[r] &  (2-A_1-A_\infty, A_1-1, A_\infty-1, \ell_{0}, \ell_{1}, \ell_{\infty})
}\end{equation*}
 \begin{equation*}
\xymatrix{
(A_0, A_1, A_\infty,  \ell_0, \ell_1, \ell_\infty)
\ar@{|->}[r] &  (A_0 - 1, 2-A_0-A_\infty, A_\infty-1, \ell_{0}, \ell_{1}, \ell_{\infty})
}\end{equation*}
They give different isomorphisms reflecting the fact that $(\tilde\calN^\vee)^{S,\prod=1}$ has automorphisms that infinitesimally move points in its non-reduced component.

%%%%%%%%%%%%%%%%%%%%%%%%%%%%%%%%%%%%%%%%%%%%%%%%%

\subsubsection{de Rham moduli} 
Let $\dRLoc_{\SL(2)}(\PP^1, S)$ denote the  de Rham version of $\Loc_{\SL(2)}(\PP^1, S)$
classifying data
$((\calE, \tau, \{\ell_{s}\}_{s\in S}, \nabla)$ where $\calE$ is a rank 2 vector bundle on $\PP^{1}$
 equipped with  a line  $\ell_{s}\subset\calE_{s}$ at each $s\in S$, and a meromorphic
connection 
\begin{equation*}
\xymatrix{
\nabla: \calE\ar[r] & \calE\otimes\Om_{\PP^1}(S)
}
\end{equation*}
 with regular singularity at each $s\in S$, whose residue $\Res_{s}\nabla$ is trivial when restricted to  $\ell_{s}$,
 and
 \begin{equation*}
\xymatrix{
\tau : \calO_{\PP^1} \ar[r]^-\sim & \det(\calE)
}
\end{equation*}
is a $\nabla$-flat
  trivialization of the determinant.

\begin{lemma} There is canonical isomorphism from the de Rham moduli to linearized moduli
\begin{equation*}
\xymatrix{
  \dRLoc_{\SL(2)}(\PP^1, S) \ar[r]^-\sim &\Lin_{\SL(2)}(\PP^1, S) 
}
\end{equation*}
\end{lemma}

\begin{proof}
First, for any $((\calE, \tau, \{\ell_{s}\}_{s\in S}, \nabla)\in\dRLoc_{\SL(2)}(\PP^1, S)$, we have $\calE\simeq\calO^2_{\PP^{1}}$.
Otherwise, there is an embedding $\calO_{\PP^{1}}(n)\incl \calE$ with quotient $\calO_{\PP^{1}}(-n)$, for some $n>0$. The composition
\begin{equation*}
\xymatrix{
\calO_{\PP^{1}}(n)\ar@{^(->}[r] &  \calE\ar[r]^-{\nabla} & \calE\otimes\Om_{\PP^{1}}(S)\ar[r] & \calE/\calO_{\PP^{1}}(n)\otimes\Om_{\PP^{1}}(S)\simeq \calO_{\PP^{1}}(1-n)
}
\end{equation*}
is $\calO_{\PP^{1}}$-linear, hence must be zero since $n>1-n$. Thus $\nabla$ restricts to a connection on $\calO_{\PP^{1}}(n)$ without poles (because the residues of $\nabla$ are nilpotent), which is impossible since $n\neq0$.

Next, fix an isomorphism $(\calE, \tau)\simeq (\calO^{2}_{\PP^{1}}, \tau_0)$ with the trivial bundle (such choices form an $\SL(2)$-torsor). 
The trivial bundle carries the  de Rham connection $d$, and  any  $((\calE, \tau, \{\ell_{s}\}_{s\in S}, \nabla)\in\dRLoc_{\SL(2)}(\PP^1, S)$
is equivalent to one of the form
$(\calO^{2}_{\PP^{1}},\{\ell_{s}\}_{s\in S}, \tau_{0},\nabla=d+\varphi)\in\dRLoc_{\SL(2)}(\PP^1, S)$
 where $\varphi:\calO^{2}_{\PP^{1}}\to\calO^{2}_{\PP^{1}}\otimes\Om_{\PP^{1}}(S)$ is a traceless $\calO$-linear map whose restriction to $\ell_{s}$, for each $s\in S$,  is trivial.

Now, define the sought-after isomorphism by the $\SL(2)$-equivariant assignment
\begin{equation*}
\xymatrix{
(\calO^{2}_{\PP^{1}},\{\ell_{s}\}_{s\in S}, \tau_{0}, \nabla=d+\varphi)
\ar[r] & 
(\Res_{0}\varphi,\ell_{0},\Res_{1}\varphi, \ell_{1},\Res_{\infty}\varphi,\ell_{\infty})\in (\tdN)^{S}
}
\end{equation*}
whose image lies in $(\tdN)^{S,\sum=0}$ thanks to the residue theorem for curves 
\begin{equation*}
\xymatrix{
\Res_{0}\varphi+\Res_{1}\varphi + \Res_{\infty}\varphi=0.
}
\end{equation*}
\end{proof}

\begin{cor}\label{c:spec mod same} The Betti moduli $\Loc_{\SL(2)}(\PP^1, S)$, its linearized version $\Lin_{\SL(2)}(\PP^1, S)$,
and the de Rham moduli $\dRLoc_{\SL(2)}(\PP^1, S)$ 
are all isomorphic as stacks over the classifying stack of $\SL(2)$.
\end{cor}

%%%%%%%%%%%%%%%%%%%%%%%%%%%%%%%%%%%%%%%%%%%%%%%%%%%%

\subsection{Coherent sheaves}

Given a stack $Z$, recall we write $\Coh(Z)$ to denote the dg derived category of coherent complexes on $Z$.
We  abuse terminology and use the term coherent sheaves
to refer to its objects.

%%%%%%%%%%%%%%%%%%%%%%%%%%%%%%

\subsubsection{Affine Hecke action} Let $\dG$ be any reductive group over $\QQ$. Fix $s\in S$, then the monoidal category $\Coh^{\dG}(\St_{\dG})$ acts  on $\Coh(\Loc_{\dG}(\PP^1, S))$ as follows. Recall the curve $\XX_{s}=\PP^{1}_{-}\coprod_{\PP^{1}\setminus\{s\}}\PP^{1}_{+}$ in Section~\ref{ss:Hk mod} with the point $s$ doubled. The moduli stack $\Loc_{\dG}(\XX_{x}, S_{\pm})$ can be similarly defined as $\Loc_{\dG}(\PP^{1}, S)$, with $\dB$-reductions at both $s_{-}$ and $s_{+}$. The Steinberg stack $\St_{\dG}/\dG$ can be identified with the moduli stack $\Loc_{\dG}(\DD_{s}, \{s_{-}, s_{+}\})$ of $\dG$-local systems on the doubled disk $\DD_{s}$ with unipotent monodromy and $\dB$-reductions at $s_{-}$ and $s_{+}$. We have a diagram
\begin{equation}\label{eq: aff hecke diag}
\xymatrix{
& \ar[dl]_-{p_-} \Loc_{\dG}(\XX_s, S_{\pm}) \ar[dr]^-{p_+} \ar[d]^-\kappa& \\
\Loc_{\dG}(\PP^{1}, S)  & \Loc_G(\DD_s, \{s_-, s_+\})=\St_{\dG}/\dG &   \Loc_{\dG}(\PP^{1}, S)  
}
\end{equation}

Passing to quasi-coherent sheaves, one obtains the affine Hecke action 
\begin{equation*}
\xymatrix{
\star_{s}: \Coh^{\dG}(\St_{\dG})
 \otimes  \QCoh(\Loc_{\dG}(\PP^{1}, S))
 \ar[r] & \QCoh(\Loc_{\dG}(\PP^{1}, S))}
\end{equation*}
\begin{equation*}
\xymatrix{
\calK\star_{s}\calF = (p_{+})_!((p_{-})^*\calF \otimes \kappa^*(\calK)).}
\end{equation*}
which preserves the subcategory $\Coh(\Loc_{\dG}(\PP^{1}, S))$ because $p_{+}$ is proper. 

Natural generalizations of the above constructions provide $\Coh(\Loc_{\dG}(\PP^{1}, S))$ the requisite coherences of a $ \Coh^{\dG}(\St_{\dG})$-module
structure.

%%%%%%%%%%%%%%%%%%%%%%%%%%%%%%%%%%%%%%%%%%%%%%%%%%%%

\subsubsection{$\dG= \SL(2)$}
The center $Z(\SL(2)) \simeq \mu_2$ acts trivially on $(\tdN)^{S,\prod=1}$, therefore it acts on the underlying coherent sheaf of  each object in $\Coh(\Loc_{\SL(2)}(\PP^1, S))$. This provides a direct sum decomposition
$$
\xymatrix{
\Coh(\Loc_{\SL(2)}(\PP^1, S)) \simeq \Coh^{\triv} (\Loc_{\SL(2)}(\PP^1, S)) \oplus \Coh^\alt(\Loc_{\SL(2)}(\PP^1, S))
}
$$
determined by whether the action of $\mu_2 \simeq Z(\SL(2))$ is trivial or by the alternating representation.

For any $s\in S$, the corresponding Atkin-Lehner involution $\cO^{cl}_{\St_{\dG}}(-1,0)\in\Coh^{\dG}(\St_{\dG})$ exchanges the two summands.

%%%%%%%%%%%%%%%%%%%%%%%%%%%%%%%%%%%%%%%%%%%%%%%%%%%%

\subsubsection{$\dG= \PGL(2)$}\label{sss:spec PGL2}

The decomposition 
into connected components ~\eqref{Loc 2 cpts}
provides a direct sum decomposition
$$
\xymatrix{
\Coh(\Loc_{\PGL(2)}(\PP^1, S)) \simeq \Coh(\Loc^{\ov 0}_{\PGL(2)}(\PP^1, S))  \oplus \Coh(\Loc^{\ov 1}_{\PGL(2)}(\PP^1, S))
}
$$

By Lemma ~\ref{l:Loc odd}, the second
summand admits an equivalence
$$
\xymatrix{
\Coh(\Loc^{\ov 1}_{\PGL(2)}(\PP^1, S))\simeq \Vect.
}
$$

%%%%%%%%%%%%%%%%%%%%%%%%%%%%%%%%%%%%%%%%%%%%%%%%%%%%

\subsubsection{Relation between $\dG= \SL(2)$ and $\dG=\PGL(2)$}\label{sss:coh shv rel}

%whose base-change to $\Bun_{\PGL(2)}(\PP^1, S)_{ev}$ is a $B\mu_2$-torsor, or in other words a $\mu_2$-gerbe.
Pullback along the map $p$ in~\eqref{eq:Loc p} provides an equivalences
\begin{equation*}
\xymatrix{
p^*:\Coh(\Loc^{\ov 0}_{\PGL(2)}(\PP^1, S))\ar[r]^-\sim &  \Coh^{\triv}(\Loc_{\SL(2)}(\PP^1, S))
}
\end{equation*}
as $\Perf(\Loc_{\PGL(2)}(\PP^{1}, S))$-module categories.

Thus using the  prior decompositions and Atkin-Lehner involutions, we see that to understand any of the above categories, it suffices to understand for example  $\Coh^{\triv}(\Loc_{\SL(2)}(\PP^1, S))$.

%%%%%%%%%%%%%%%%%%%%%%%%%%%%%%%%%%%%%%%%%%%%%%%%%%%%
%%%%%%%%%%%%%%%%%%%%%%%%%%%%%%%%%%%%%%%%%%%%%%%%%%%%
%%%%%%%%%%%%%%%%%%%%%%%%%%%%%%%%%%%%%%%%%%%%%%%%%%%%

\section{Langlands duality} 

In this section we give the proof of our main theorem. For most of this section we focus on $G=\PGL(2)$ and $\dG= \SL(2)$. We will establish results  in this case
first, and then use them to deduce the case of  $G=\SL(2)$ and $\dG= \PGL(2)$.

%%%%%%%%%%%%%%%%%%%%%%%%%%%%%%%%%%%%%%%%%%%%%%%%%%%%

\subsection{Dictionary: matching objects}

Before proceeding to the construction and proof of the equivalence, let us record here various
distinguished objects
that will be matched by it.

Let $U^{\ov{0}}, U^{\ov{1}} \subset \Bun_{\PGL(2)}(\PP^1, S)$ denote the open substacks classifying parabolic bundles with respectively  underlying bundle 
$\calE \simeq \calO_{\PP^1} \oplus \calO_{\PP^1}$, $\calE \simeq \calO_{\PP^1}(1) \oplus \calO_{\PP^1}$.
In what follows, all of the sheaves will be understood to be extensions by zero off of $U^{\ov{0}}, U^{\ov{1}}$.

\subsubsection{$U^{\ov 1}$} Within $U^{\ov{1}}$, 
consider the open substacks
\begin{equation*}
\xymatrix{
 c_1(*)    \ar@{^(->}[r]^-j & 
  c_1(*) \cup c_1(\vn)   \ar@{^(->}[r]^-i & 
U^{\ov{1}}
}\end{equation*}
classifying respectively bundles $\calE \simeq \calO_{\PP^1}(1) \oplus \calO_{\PP^1}$ with generic lines $\ell_0, \ell_1, \ell_\infty$,
and more generally, lines $\ell_0, \ell_1, \ell_\infty$ with none
contained within $\calO_{\PP^1}(1)$.

We have the following distinguished objects:

\begin{equation*}
\xymatrix{
\Wh_{S} = i_! j_* \const{c_1(*)} \ar@{<->}[r] & \calO_{\Loc_{\SL(2)}(\PP^1, S)} & \mbox{(by construction)} 
}
\end{equation*}
\begin{equation*}
\xymatrix{
\Eis_{-1} = i_! \const{c_1(\vn)}[-1] \ar@{<->}[r] & \calO_{\tdN_\Delta} & \mbox{(by Prop.~\ref{p:Eis})} 
}
\end{equation*}
\begin{equation*}
\xymatrix{
\Eis_{1} = \const{c_1(S)}[-3] \ar@{<->}[r] & \calO_{\tdN_\Delta}(2) & \mbox{(by Prop.~\ref{p:Eis})}
}
\end{equation*}

\subsubsection{$U^{\ov 0}$}\label{sss:U0}  For a point $c_{0}(R)\subset U^{\ov 0}$, recall from ~\eqref{eq:IC k R} the notation $\IC_{0}(R)$ for the IC-sheaf of its closure. On $U^{\ov{0}}$, we have the following distinguished objects:

\begin{equation*}
\xymatrix{
\IC_0(\vn) \ar@{<->}[r] & \calO_{(\PP^1)^3}(-1, -1, -1) & \mbox{(by Prop.~\ref{prop: 2 to 3 pt compatibility})}
}
\end{equation*}
\begin{equation}\label{IC01}
\xymatrix{
\IC_0(0, 1) \ar@{<->}[r] & \calO_{\Lambda_{0, 1}}(0, 0, -1) & \mbox{(by Prop.~\ref{prop: 2 to 3 pt compatibility})}
}
\end{equation}
\begin{equation*}
\xymatrix{
\IC_0(0, \infty) \ar@{<->}[r] & \calO_{\Lambda_{0, \infty}}(0, -1, 0) & \mbox{(by Prop.~\ref{prop: 2 to 3 pt compatibility})}
}
\end{equation*}
\begin{equation}\label{IC1inf}
\xymatrix{
\IC_0(1, \infty) \ar@{<->}[r] & \calO_{\Lambda_{1, \infty}}(-1, 0, 0) & \mbox{(by Prop.~\ref{prop: 2 to 3 pt compatibility})}
}
\end{equation}
\begin{equation*}
\xymatrix{
\Eis_{0} = \const{c_0(S)}[-2] \ar@{<->}[r] & \calO_{\tdN_\Delta}(1)& \mbox{(by Prop.~\ref{p:Eis})}
}
\end{equation*}

We will also use the object $J_1\star_1 \Wh_{S}$. Consider the open substacks
\begin{equation*}
\xymatrix{
 c_0(\vn)    \ar@{^(->}[r]^-j & 
  c_0(\vn) \cup c_0(0, \infty)   \ar@{^(->}[r]^-i & 
U^{\ov{0}}
}\end{equation*}
classifying  bundles $\calE \simeq \calO^{2}_{\PP^1}$ with respectively generic lines $\ell_0, \ell_1, \ell_\infty$,
and more generally, lines $\ell_0, \ell_1, \ell_\infty$ with the only possible coincidence
$\ell_0 = \ell_\infty$. Then we have
\begin{equation*}
\xymatrix{
%\Wh(0, 1,0) = 
J_1\star_1 \Wh_{S}  \simeq i_* j_! \const{c_0(\vn)} \ar@{<->}[r] & \calO_{\Loc_{\SL(2)}(\PP^1, S)}(0, 1,0) & \mbox{(by calculating $J_1\star_1$)}
}
\end{equation*}

\subsection{Construction of functor}

%%%%%%%%%%%%%%%%%%%%%%%%%%%%%%%%%%%%%%%%%%%%%%%%%%%%

To construct the functor in ~\eqref{eq: 1}, we first construct an action of the monoidal category $\Perf(\Loc_{\dG}(\PP^{1},S))$ (under $\otimes$) on the automorphic category $\Sh(\Bun_{G}(X,S))$. Of course this action should be given by ``integrating'' the local Hecke actions, but the key property that makes the integration possible is the following local constancy result.

%\marginpar{Moved old 3.4 to here, since the statement alreadys involves the spectral side.}

\begin{lemma}\label{l:hk loc const}
For $G= \PGL(2)$ or $\SL(2)$, the $ \calH_G^\sph$-module structure on  $\Sh_!(\Bun_G(\PP^1, S))$ 
is  locally constant in $x\in \PP^1 \setminus S$.
\end{lemma}

\begin{proof}
We will give an elementary proof using our prior parameterization of objects, though it is possible to give an infinitesimal argument of more general applicability.

Observe that it suffices to prove the assertion for $G=\PGL(2)$, and a Hecke operator given by a lower modification (which is a monoidal generator for $\cH^{\sph}_{\PGL(2)}$). 

%\marginpar{proof is rewritten; original argument has a gap, or at least is too sketchy.}

For $x\in \PS$, the Hecke correspondence of a lower modification at $x$ is given by the diagram
\begin{equation}\label{px+-}
\xymatrix{\Bun_{\PGL(2)}(\PP^{1}, S) & \Bun_{\PGL(2)}(\PP^{1}, S\cup\{x\})\ar[l]_-{p_{x-}} \ar[r]^-{p_{x+}} &\Bun_{\PGL(2)}(\PP^{1}, S)}
\end{equation}
Here $p_{x-}$ sends a point $(\cE, \ell_{0}, \ell_{1}, \ell_{\infty}, \ell_{x})$ to $(\cE, \ell_{0}, \ell_{1}, \ell_{\infty})$, and $p_{x+}$ sends it to $(\cE', \ell'_{0}, \ell'_{1},\ell'_{\infty})$, where $\cE'$ fits into a short exact sequence $\calE'\to \calE \to \ell_x$, and $\ell'_{s}=\ell_{s}$ for $s\in S$ after identifying $\cE'_{s}$ with $\cE_{s}$.

We need to show that for any $\cF\in \Sh_{!}(\Bun_{\PGL(2)}(\PP^{1}, S))$, the stalk of $p_{x+!}p^{*}_{x-}\cF$ at any point $b'\in\Bun_{\PGL(2)}(\PP^{1}, S)$ is locally constant as $x$ varies in $\PP^{1} \setminus S$. It suffices to check this for $\cF=\cF_{k}(R)$, one of the basis objects. Let $b=c_{k}(R)$ and let $H_{x}(b,b')=(p_{x-}, p_{x+})^{-1}(b,b')\subset \Bun_{\PGL(2)}(\PP^{1}, S\cup\{x\})$. As $x$ varies, the $H_{x}(b,b')$ form a family $h_{b,b'}: H(b,b')\to \PS$. Since the stalk of $p_{x+!}p^{*}_{x-}\cF_{k}(R)$ at $b'$ is simply $\cohoc{*}{H_{x}(b,b'),\QQ}$, it suffices to show that $h$ is a fibration.

Therefore we fix $b,b'\in \Bun_{\PGL(2)}(\PP^{1}, S)$, viewed as classifying spaces of their respective automorphism groups. Consider the two projections restricted from ~\eqref{px+-}:
\begin{equation*}
\xymatrix{b & H_{x}(b,b')\ar[l]-_{h_{x-}}\ar[r]^-{h_{x+}} & b'}
\end{equation*}
The fibers of $h_{x-}$ and $h_{x+}$ are subsets of $\PP^{1}$, hence $\dim\Aut(b)$ and $\dim\Aut(b')$ differ at most by $1$. We have two cases:
\begin{enumerate}
\item If $\dim\Aut(b)$ and $\dim\Aut(b')$ differ by $1$, one of the arrows $h_{x-}$ or $h_{x+}$ has to be an isomorphism. Therefore, in this case, $H(b,b')\simeq b\times (\PS)$ or $H(b,b')\simeq b'\times (\PS)$, hence $h_{b,b'}$ is a trivial fibration.

\item If $\dim\Aut(b)=\dim\Aut(b')$. Inspecting the list of points in $\Bun_{\PGL(2)}(\PP^{1}, S)$ given in Section~\ref{sss:Bun PGL2 pts}, we see this happens only for the following pairs $(b,b')$.
\begin{itemize}
\item $(b,b')=(c_{0}(S-\{s\}), c_{1}(s))$ or $(c_{1}(s), c_{0}(S-\{s\}))$ for $s\in S$. In this case both $h_{x-}$ and $h_{x+}$ are isomorphisms, therefore $h_{b,b'}$ is a trivial fibration.
\item $(b,b')=(c_{1}(\vn), c_{2}(\vn))$ or $(c_{2}(\vn), c_{1}(\vn))$. In this case both $h_{x-}$ and $h_{x+}$ are isomorphisms, therefore $h_{b,b'}$ is a trivial fibration.
\item $(b,b')=(c_{0}(\vn), c_{1}(*))$ or $(c_{1}(*), c_{0}(\vn))$. In this case, $\Aut(b)=\Aut(b')=1$. If $b=(\cE=\cO^{2}_{\PP^{1}}, \ell_{0},\ell_{1}, \ell_{\infty})$, then 
\begin{equation*}
H(b,b')=\{(x,\ell_{x})\in (\PS)\times \PP^{1}\mid \mbox{there is no map $\cO_{\PP^{1}}(-1)\to \cE$ containing $\ell_{0}, \ell_{1}, \ell_{\infty}$ and $\ell_{x}$}\}. 
\end{equation*}
One can check that $H(b,b')\subset (\PS)\times\PP^{1}$ is the complement of the graph of an open embedding $\PS\incl \PP^{1}$. Therefore $h_{b,b'}$ is an $\AA^{1}$-fibration.
\end{itemize}
\end{enumerate}
This completes the proof of the lemma.

%%%%%%%%%% old version %%%%%%%%%%%%
%Fix a configuration $b=(\calE, \ell_0, \ell_1, \ell_\infty) \in \Bun_{\PGL(2)}(\PP^1, S)$. Then for $x\in \PP^1 \setminus S$, the Hecke correspondence of a
% lower modification at $x$ is given by the correspondence
%$$
%\xymatrix{
%b=\{(\calE, \ell_0, \ell_1, \ell_\infty)\} & \ar[l]_-{{}^{b}p_{x-}} \{(\calE, \ell_0, \ell_1, \ell_\infty, \ell_x)\} \ar[r]^-{{}^{b}p_{x+}} & \{(\calE', \ell_0, \ell_1, \ell_\infty)\}
%}
%$$
%where $\ell_x\subset\calE_x$ is a line, and we have a short exact sequence $\calE'\to \calE \to \ell_x$. 
%
%Fix a configuration $b'=(\calE', \ell_0, \ell_1, \ell_\infty) \in \Bun_{\PGL(2)}(\PP^1, S)$, let ${}^{b}\Sigma_{x}={}^{b}p^{-1}_{x+}(b)\subset\PP(\cE_{x})$ be the space of possible $\ell_x$ giving rise to $b$. As $x$ varies in $\PP^{1}\setminus S$, the ${}^{b}\Sigma_{x}$ form a family ${}^{b}\Sigma\subset \PP(\cE)|_{\PP^{1}\setminus S}$. Thus we must show that ${}^{b}\Sigma\to \PP^{1}\setminus S$  is a fibration, which is simple to check using our prior parameterization of objects:

%\begin{enumerate}
%
%
%\item
%For $\calE \simeq \calO_{\PP^1} \oplus \calO_{\PP^1}$,
%${}^{b}\Sigma_{x}$ depends only on  
%which of the lines $\ell_0, \ell_1, \ell_\infty$ coincides with $\ell_x$.
%
%
%
%\item
%For $k\geq 1$, and $\calE \simeq \calO_{\PP^1}(k) \oplus \calO_{\PP^1}$,
%${}^{b}\Sigma_{x}$ depends only on whether  $\ell_x$ lies in  the summand $\calO_{\PP^1}(k)$ of $\cE$.
%
%\end{enumerate}

\end{proof}

\begin{remark}
The above lemma  extends to the statement~\cite{NYhecke}  that spherical Hecke modifications at a point of a curve acting on sheaves 
on $\Bun_G(X, S)$ with nilpotent singular support are locally constant in the point of the curve.
All codirections are nilpotent in special cases such as considered in the lemma.
\end{remark}

Note the natural inclusion and projection
\begin{equation*}
\xymatrix{
\Loc_{\dG}(\PP^{1},S)\ar[r]^{\sim} & (\tilde \calN^\vee)^{S, \prod = 1} /\dG\ar[r] & (\tilde \calN^\vee)^S/\dG\ar[r] & (\tilde \calN^\vee/\dG)^S.
}
\end{equation*}
%where $(\tilde\calN^\vee)^{S,\prod=1} $ denotes  the product of $S$ cyclically ordered copies with the equation on the group elements $\prod = 1$  imposed inside of $\dG$.

Passing to perfect complexes,  we obtain a composite pullback functor 
\begin{equation*}
\xymatrix{
\Perf(\tilde \calN^\vee/\dG)^{\otimes S}  \ar[r]^-\sim & \Perf((\tilde \calN^\vee/\dG)^{S})   \ar[r] & \Perf ( (\tilde \calN^\vee)^S/\dG)
\ar[r] & \Perf (\Loc_{\dG}(\PP^{1}, S)).
}
\end{equation*}

Recall from~\eqref{eq:PerfN S action} that we have an action of $\Perf(\tilde \calN^\vee/\dG)^{\otimes S}$ on $\Sh_!(\Bun_G(\PP^1, S))$ coming from Wakimoto operators at each $s\in S$.

\begin{cor}
For $G= \PGL(2)$ or $\SL(2)$,  the natural $\Perf(\tilde \calN^\vee/\dG)^{\otimes S}$-action
on 
$\Sh_!(\Bun_G(\PP^1, S))$  in ~\eqref{eq:PerfN S action}  factors through the pullback
\begin{equation*}
\xymatrix{
\Perf(\tilde \calN^\vee/\dG)^{\otimes S}   %\ar[r] & \Perf (\tilde \calN^\vee)^S/\dG)
   \ar[r] & \Perf (\Loc_{\dG}(\PP^{1}, S)).
   }
\end{equation*}
\end{cor}

\begin{proof}
Set $U = \PP^1 \setminus S$, and fix a basepoint $u_0 \in U$.
Fix a collection of paths $\theta_s:[0,1] \to \PP^1(\CC)$, $s\in S$,  with $\theta_s(0) = u_0$, $\theta_s((0, 1)) \subset U$ pairwise disjoint,
and $\theta_s(1) =  s\in S$.
Each $\theta_s$ provides a counterclockwise loop $\gamma_s\in \pi_1(U, u_0)$ obtained by following 
$\theta_{s}$ towards $s$, then taking a small counterclockwise loop around $s$, and
finally following the inverse of $\theta_{s}$ back to $u_0$.
We can choose the  collection of paths $\gamma_s$, $s\in S$ so that the relation $\gamma_{0}\gamma_{1}\gamma_{\infty}=1$ holds in $\pi_{1}(U,u_{0})$.

For each $s\in S$, using the loop $\gamma_{s}$, and the local constancy of the spherical Hecke action proved in Lemma~\ref{l:hk loc const}, we may identify the $\cH^{\sph}_{G}$-action on $\Sh_{!}(\Bun_{G}(\PP^{1}, S))$ given by $\Hecke^{\sph}_{u_{0}}$ with the action given by the nearby cycles $\Psi_{s}(\Hecke^{\sph}_{D_{s}})$, which in turn, by ~\eqref{eq:Z comp},  is identified with the action of $\cH^{\sph}_{G}$ via Gaitsgory's central functor and the affine Hecke action. In other words, the action of $\Rep(\dG)$ on $\Sh_{!}(\Bun_{G}(\PP^{1}, S))$ by $\Hecke^{\sph}_{u_{0}}$ is identified with the composition of the monoidal functor $\Rep(\dG)\to \Perf(\tdN/\dG)\xr{\Phi^{\aff}\circ\D_{*}} \cH^{\aff}_{G}$ and the  $\star_{s}$-action of the latter.  

%\marginpar{Proof expanded starting here.}

The above discussion shows that the $\Perf(\pt/\dG)=\Rep(\dG)$-actions coming from all three factors of $\Perf(\tdN/\dG)$ are canonically identified, therefore the  $\Perf(\tdN/\dG)^{\otimes S}$-action factors through 
\begin{equation*}
\xymatrix{\Perf(\tdN/\dG)^{\otimes S}\otimes_{\Perf(\pt/\dG)^{\otimes S}}\Perf(\pt/\dG)}.
\end{equation*}
By \cite[Theorem 4.7]{BFN}\footnote{In \cite[Theorem 4.7]{BFN},  the equivalence $\QCoh(Y_{1})\otimes_{\QCoh(Y)}\QCoh(Y_{2})\isom\QCoh(X_{1}\times_{Y} X_{2})$ is proved for perfect stacks $X_{1}, X_{2}$ and $Y$. By taking compact objects on both sides, we get the version for perfect complexes  $\Perf(Y_{1})\otimes_{\Perf(Y)}\Perf(Y_{2})\isom\Perf(X_{1}\times_{Y} X_{2})$.}, the above tensor product is equivalent to
\begin{equation*}
\xymatrix{\Perf((\tdN/\dG)^{S}\times_{(\pt/\dG)^{S}}\pt/\dG)\ar[r]^-{\sim} & \Perf((\tdN/\dG)^{S}/\dG).}
\end{equation*}
Thus we have shown that the action of $\Perf(\tdN/\dG)^{\otimes S}$ factors through $\Perf((\tdN/\dG)^{S}/\dG)$.

Finally, we consider the pullback along the map ~\eqref{mult S}
\begin{equation*}
\xymatrix{\Perf(\dG/\dG)\ar[r] & \Perf((\tdN)^{S}/\dG)}
\end{equation*}
The $\Perf((\tdN)^{S}/\dG)$-action on $\Sh_{!}(\Bun_{G}(\PP^{1}, S))$ then restricts to an action of $\Perf(\dG/\dG)$, whose further restriction to $\Rep(\dG)$ is given by $\Hecke^{\sph}_{u_{0}}$. Such an action of $\Perf(\dG/\dG)$ is equivalent to the following data: for each $V\in \Rep(\dG)$, the endo-functor $\Hecke^{\sph}_{u_{0}}(V,-)$ of $\Sh_{!}(\Bun_{G}(\PP^{1}, S))$ is equipped with an automorphism $m_{V}$ satisfying $m_{V\otimes V'}=m_{V}\circ  m_{V'}$ (and further coherences). By construction, $m_{V}$ is given by the product of the monodromies of $\Hecke^{\sph}_{u_{0}}(V,\cF)$ along $\gamma_{0},\gamma_{1}$ and $\gamma_{\infty}$.  The equation $\gamma_{0}\gamma_{1}\gamma_{\infty}=1$ implies that $m_{V}$ is isomorphic to the identity automorphism for all $V$ (in a way compatible with tensor structures), hence the $\Perf(\dG/\dG)$-action factors through 
\begin{equation*}
\xymatrix{\Perf(\dG/\dG)\ar[r]^{\ev_{1}} & \Perf(\pt/\dG)\ar[r]^{\sim} & \Rep(\dG)}
\end{equation*}
and the $\Perf((\tdN)^{S}/\dG)$-action further factors through
\begin{equation*}
\Perf((\tdN)^{S}/\dG)\otimes_{\Perf(\dG/\dG)}\Perf(\pt/\dG).
\end{equation*}
Again by ~\cite[Theorem 4.7]{BFN}, the above tensor product is equivalent to
\begin{equation*}
\xymatrix{\Perf((\tdN)^{S}/\dG\times_{\dG/\dG}\pt/\dG)=\Perf((\tdN)^{S,\prod=1}/\dG)=\Perf(\Loc_{\dG}(\PP^{1}, S)).}
\end{equation*}
This proves the desired factorization of the $\Perf(\tdN/\dG)^{\otimes S}$-action on $\Sh_{!}(\Bun_{G}(\PP^{1}, S))$.
\end{proof}

%%%%%%%%%%%% old version %%%%%%%%%%%
%Now thanks to Lemma~\ref{l:hk loc const}, and Gaitsgory's nearby cycles functor, applying spherical Hecke modifications along the collection of paths $\gamma_s$, $s\in S$, provides a factorization of the natural $\Perf(\tilde \calN^\vee/\dG)^{\otimes S}$-action
% through the monoidal functors
%\begin{equation*}
%\xymatrix{
%\Perf(\tilde \calN^\vee/\dG)^{\otimes S}  \ar[r]^-\sim & \Perf((\tilde \calN^\vee/\dG)^{S})   \ar[r] & 
%\Perf ((\tilde \calN^\vee)^S/\dG).
%}
%\end{equation*}
%
%Finally, the equation $\gamma_{0}\gamma_{1}\gamma_{\infty}=1$ implies 
%the further factorization through
%\begin{equation*}
%\xymatrix{
%\Perf ((\tilde \calN^\vee)^S/\dG)
%   \ar[r] & \Perf ((\tilde \calN^\vee)^{S, \prod = 1}/\dG)=\Perf(\Loc_{\dG}(\PP^{1}, S))
%   }
%\end{equation*}
%by results of~\cite{BFN}.

%For $G= \PGL(2)$, $\dG = \SL(2)$,  recall the tensor action
%\begin{equation*}
%\xymatrix{
% \Perf ( \Loc_{\SL(2)}(\PP^1, S)) \otimes \Sh(\Bun_{\PGL(2)}(\PP^1, S))
%   \ar[r] &
%    \Sh(\Bun_{\PGL(2)}(\PP^1, S))}
%\end{equation*}
%and that it preserves $\Sh_!( \Bun_{\PGL(2)}(\PP^1, S))$.

\subsubsection{The functor $\Phi$}
From now on we let $G=\PGL(2)$ and $\dG=\SL(2)$.

By making $\Perf ( \Loc_{\SL(2)}(\PP^1, S))$ act on the Whittaker sheaf $\Wh_{S} \in  \Sh_!(\Bun_{\PGL(2)}(\PP^1, S))$,  we obtain a functor
\begin{equation*}
\xymatrix{
 \Phi_\Perf:\Perf ( \Loc_{\SL(2)}(\PP^1, S)) 
   \ar[r] &
    \Sh_!(\Bun_{\PGL(2)}(\PP^1, S)).}
\end{equation*}
Since $\Sh(\Bun_{\PGL(2)}(\PP^1, S))$ is cocomplete, we can take the continuous extension of $\Phi_{\Perf}$ to get a functor
\begin{equation*}
\xymatrix{
 \Phi:\QCoh ( \Loc_{\SL(2)}(\PP^1, S)) 
   \ar[r] &
   \Sh(\Bun_{\PGL(2)}(\PP^1, S)).}
\end{equation*}

Consider as well the restriction  
\begin{equation*}
\xymatrix{
\Phi_\Coh =  \Phi |_{\Coh ( \Loc_{\SL(2)}(\PP^1, S)). }
}
  \end{equation*}

%
%Introduce the full dg subcategory
%\begin{equation}
%\xymatrix{
%  \Sh_\dag(\Bun_{\PGL(2)}(\PP^1, S))
%\subset  \Sh(\Bun_{\PGL(2)}(\PP^1, S))
%}
%   \end{equation}
%of compact objects, or equivalently, constructible complexes that are extensions by zero off of finite-type substacks, and whose stalks at all points are torsion modules over the equivariant cohomology
%of the automorphisms of the points. 

We will defer the proof of the following until Proposition~\ref{p:Coh!} below but mention it here for clarity.

\begin{prop}\label{prop: extension images}

%\begin{enumerate}
%\item
%The functor $\Phi_\Perf$  lands in the full dg subcategory
%$  \Sh_\dag(\Bun_{\PGL(2)}(\PP^1, S)).
%%\subset  \Sh_!(\Bun_{\PGL(2)}(\PP^1, S))
%$

%\item
The  functor $\Phi_\Coh$
 lands in the full dg subcategory
$  \Sh_!(\Bun_{\PGL(2)}(\PP^1, S)).
$
%\end{enumerate}
\end{prop}
%%%%%%%%%%%%%%%%%%%%%%%%%%%%%%%%%%%%%%%%%%%%%%%%%%%%

\subsection{Compatibilities with changing level structure}

In this section, we will see why it is importance we act upon the Whittaker sheaf to construct the functor
$ \Phi$
and its elaborations. 

%We could deduce the main assertion Proposition~\ref{prop: 2 to 3 pt compatibility} from the 
%general affine Hecke equivariance of Proposition~\ref{p:Hk equiv}
%but we will in any case give a direct proof.

%\marginpar{Does Prop ~\ref{prop: 2 to 3 pt compatibility} formally follow from Prop ~\ref{p:Hk equiv}?}

 %%%%%%%%%%%%%%%%%%%%%%%%%%%%%%%%%%%%%%%%%%%%%%%%%%%%%

\subsubsection{Changing level structure on the automorphic side}

On the automorphic side, for $s\in S$, consider the natural $\PP^1$-fibration
 \begin{equation*}
\xymatrix{ 
\pi_{s}: \Bun_{\PGL(2)}(\PP^1, S) \ar[r] & \Bun_{\PGL(2)}(\PP^1, S\setminus \{s\})
}
\end{equation*}
where we forget the flag at $s\in S$.
It provides an adjoint triple
\begin{equation*}
\xymatrix{
\Sh_!(\Bun_{\PGL(2)}(\PP^1, S))
\ar[rr]^-{\pi_{s*}=\pi_{s!}} &&
 \Sh_{!}(\Bun_{\PGL(2)}(\PP^1, S\setminus \{s\})) \ar@<-2.5ex>[ll]_-{\pi^{*}_{s}}\ar@<2.5ex>[ll]_-{\pi^{!}_{s}}}
\end{equation*}

\subsubsection{Changing level structure on the spectral side} 

We seek the corresponding adjoint triple on the spectral side.

First, introduce the intermediate stack
\begin{equation*}
\xymatrix{ 
 \Loc_{\SL(2)}(\PP^1, S\setminus \{s\}, \{s\})  =  \Loc_{\SL(2)}(\PP^1, S\setminus \{s\})  \times_{\{s\}/\SL(2)} \PP^1/\SL(2)
}
\end{equation*}
classifying 
an $\SL(2)$-local system on $\PP^1 \setminus (S\setminus \{s\})$ with $B^\vee$-reductions near $S\setminus \{s\}$ with trivial induced $T^\vee$-monodromy, and an additional $B^\vee$-reduction at $s\in \PP^1$.

Next, consider the natural correspondence
\begin{equation*}
\xymatrix{ 
& \Loc_{\SL(2)}(\PP^1, S\setminus \{s\},  \{s\}) \ar[dl]_-{q_{s}} \ar[d]_-{\kappa_{s}} \ar[rd]^-{p_{s}}& \\
\Loc_{\SL(2)}(\PP^1, S) & \PP^1/\SL(2) &  \Loc_{\SL(2)}(\PP^1, S\setminus \{s\}) 
}
\end{equation*}
where $p_{s}$ is the evident $\PP^1$-fibration forgetting the flag at $s\in \PP^1$, $\kappa_{s}$ forgets all of the data except the flag at $s\in \PP^1$,
and $q_{s}$ is the evident inclusion fitting into the Cartesian square
\begin{equation*}
\xymatrix{ 
\ar[d] \Loc_{\SL(2)}(\PP^1, S) & \Loc_{\SL(2)}(\PP^1, S\setminus \{s\}, \{s\}) \ar[l]_-{q_{s}}  \ar[d] \\
 (T^*\PP^1) /\SL(2)& \ar[l] \PP^1/\SL(2) \\
}
\end{equation*}
Note in particular that the pullback $q_{s}^*$ preserves coherent complexes since up to base change it is given by tensoring with the perfect complex $\calO_{\PP^1} = \mathit{Cone}( \calO_{ T^*\PP^1 }(2) \to  \calO_{T^*\PP^1})$.

Passing to coherent complexes, define the adjoint triple 
%\begin{equation}
%\xymatrix{ 
%\eta: \Coh(\Loc_{\SL(2)}(\PP^1, S))
%\ar[r] & 
%\Coh(\Loc_{\SL(2)}(\PP^1, \{0, 1\}))
% &
% \eta(\calF) = p_*(\kappa^*\calO_{\PP^1}(-1) \otimes i^*\calF)}
%\end{equation}
%It extends to an adjoint triple
\begin{equation*}
\xymatrix{
 \Coh(\Loc_{\SL(2)}(\PP^1, S))
\ar[rr]^-{\eta_{s}} &&
 \Coh(\Loc_{\SL(2)}(\PP^1, S\setminus \{s\})) \ar@<-2.5ex>[ll]_-{\eta^\ell_{s}}\ar@<2.5ex>[ll]_-{\eta^r_{s}}
 }
\end{equation*}
\begin{equation*}
\xymatrix{
 \eta_{s}(\calF) = p_{s*}(q_{s}^*\calF\otimes \kappa_{s}^*\calO_{\PP^1}(-1)) 
 }
\end{equation*}

\begin{equation*}
\xymatrix{
 \eta^\ell_{s} = q_{s*}(p_{s}^*\calF\otimes \kappa_{s}^*\calO_{\PP^1}(-1)[-1])  &   \eta^r_{s} = q_{s*}(p_{s}^*\calF\otimes \kappa_{s}^*\calO_{\PP^1}(-1)[1]) 
 }
\end{equation*}

%
%Since $p$ is a $\PP^1$-fibration, there is a natural equivalence 
%\begin{equation}
%\xymatrix{
%p^{!}\simeq p^{*}\otimes\omega_{\PP^1}\simeq p^{*}\otimes\calO_{\PP^1}(-2)[1]
%}\end{equation}
%where $\omega_{\PP^1} \simeq \calO_{\PP^1}(-2)[1]$ is the dualizing complex.
%
%
%To find a middle ground between $p^{!}$ and $p^{*}$,
% introduce the functor 
%\begin{equation}
%\xymatrix{
%p^{\bullet}=p^{*}\otimes \calO_{\PP^1} (-1): 
% \Coh(\Loc_{\SL(2)}(\PP^1, \{0, \infty\}))
% \ar[r] &  \Coh(\Loc_{\SL(2)}(\PP^1, \{0, \infty\}, 1))
%}\end{equation}
%
%
%Similarly, since $i$ is equivalent via base-change to the embedding of the zero-section of the cotangent bundle of $\PP^1$,
%there is a natural equivalence 
%\begin{equation}
%\xymatrix{
%i^{!}\simeq i^{*}\otimes\omega_{\PP^1}\simeq i^{*}\otimes\calO_{\PP^1}(-2)[1]
%}\end{equation}
%
%To find a middle ground between $i^{!}$ and $i^{*}$,
% introduce the functor 
%\begin{equation}
%\xymatrix{
%i^{\bullet}=i^{*}\otimes\calO_{\PP^1}(-1): \Coh(\Loc_{\SL(2)}(\PP^1, \{0, \infty\}))
%\ar[r] & 
% \Coh(\Loc_{\SL(2)}(\PP^1, \{0, \infty\}, \{1\}))
%}\end{equation}
%

Finally, recall from Section \ref{ss:2 pts} that there is an equivalence (denoted by $\Phi_{0,\infty}$ there)\footnote{Note that the definition of $\Phi_{0,\infty}$ in Section \ref{ss:2 pts} is asymmetric with respect to $0$ and $\infty$: its definition uses the Hecke action at $0$. Therefore, in defining $\Phi_{S\setminus\{s\}}$, we need to make a choice of one of the two points in $S\setminus\{s\}$. The results involving $\Phi_{S\setminus\{s\}}$ will be valid for any such choice. In fact, one can show that for different choices of points in $S\setminus\{s\}$, the resulting functors are canonically isomorphic to each other, but we do not need this statement in the sequel.}
\begin{equation*}
\xymatrix{
\Phi_{S\setminus \{s\}}:\Coh( \Loc_{\SL(2)}(\PP^1, S\setminus \{s\})) 
   \ar[r]^-\sim &
    \Sh_!(\Bun_{\PGL(2)}(\PP^1, S\setminus \{s\})).}
\end{equation*}

\begin{prop}\label{prop: 2 to 3 pt compatibility}
For each pair of vertical arrows $(\eta^{\ell}_{s}, \pi^{*}_{s}), (\eta_{s}, \pi_{s!})$ and $(\eta^{r}_{s}, \pi^{!}_{s})$, the following diagram commutes by a canonical isomorphism
\begin{equation*}
\xymatrix{
    \Coh( \Loc_{\SL(2)}(\PP^1,S))  \ar[d]_{\eta_{s}}\ar[r]^-{\Phi} & 
      \Sh_!(\Bun_{\PGL(2)}(\PP^1, S))  \ar[d]_{\pi_{s!}}
\\
  \ar@<2.5ex>[u]^-{\eta^\ell_{s}} \ar@<-2.5ex>[u]^{\eta^{r}_{s}} \Coh( \Loc_{\SL(2)}(\PP^1, S\setminus \{s\}))  \ar[r]^-{\Phi_{S\setminus \{s\}}} & 
      \Sh_!(\Bun_{\PGL(2)}(\PP^1, S\setminus \{s\}))  \ar@<2.5ex>[u]^-{\pi^*_{s}} \ar@<-2.5ex>[u]^{\pi^{!}_{s}}
}
\end{equation*}
%The same is true when $\eta^{\ell}_{s}$ and $\pi^{*}$ are replaced with $\eta^{r}$ and $\pi^{!}$.
\end{prop}

\begin{proof} 
(1) We first prove the commutativity for the pair $(\eta^{\ell}_{s}, \pi^{*}_{s})$.

By construction, both compositions 
\begin{equation*}
\xymatrix{
\Phi\circ \eta^\ell_{s}, \pi^* _{s}\circ \Phi_{S\setminus\{s\}}:\Coh( \Loc_{\SL(2)}(\PP^1,S\setminus \{s\})) \ar[r] &
  \Sh_!(\Bun_{\PGL(2)}(\PP^1, S))
}
\end{equation*}
are naturally equivariant for the tensor action of  $\Perf(\tdN/\dG)$ at the points $S\setminus \{s\}$. Therefore  it suffices to give a natural isomorphism when evaluated on the structure sheaf
\begin{equation}\label{pi eta on cO}
\xymatrix{
\pi^*_{s}( \Phi_{S\setminus\{s\}}(\calO_{\Loc})) \ar[r]^-\sim & \Phi( \eta^\ell_{s}(\calO_{\Loc^{s}})) 
}
\end{equation}
where we use short-handed notation $\cO_{\Loc}$ (resp. $\cO_{\Loc^{s}}$) to mean the structure sheaf of $\Loc_{\SL(2)}(\PP^{1},S)$ (resp. $\Loc_{\SL(2)}(\PP^{1},S\setminus\{s\})$).\footnote{There is a general argument which works for any $G$, but here we give a more down-to-earth argument for $G=\PGL(2)$.}
%write $ \calO_{(2)} = \calO_{\Loc_{\SL(2)}(\PP^1, S\setminus \{s\})}$.

On the one hand, note the isomorphism 
\begin{equation*}
\xymatrix{
 \eta^\ell_{s}(\calO_{\Loc^{s}})   \simeq \mathit{Cone}(\calO_{\Loc}(1,0,0)\ar[r] & \calO_{\Loc}(-1,0,0))[-1]
}
\end{equation*}
where we order the twists with $s$ as the first component. We can recast this as an isomorphism
\begin{equation*}
\xymatrix{
 \eta^\ell_{s}(\cO_{\Loc^{s}})   \simeq \mathit{Cone}(\cO_{\tdN}(1)\star_s \cO_{\Loc}\ar[r] & \cO_{\tdN}(-1)\star_s \cO_{\Loc})[-1].
}
\end{equation*}
%where $J_k\star_s$ denotes the action of a Wakimoto operator at the point $s\in \PP^1$.
Moreover, the morphism of the cone is induced by the natural morphism of Wakimoto kernels
\begin{equation*}
\xymatrix{
\calO_{\tdN}(1)  \ar[r] & \calO_{\tdN}(-1)
}
\end{equation*}
Thus we have an isomorphism
\begin{equation*}
\xymatrix{
 \Phi(\eta^\ell_{s}(\cO_{\Loc^{s}}))   %\simeq \mathit{Cone}(J_{1}\star_1 \Wh\ar[r] & J_{-1}\star_1 \Wh)
 \simeq \mathit{Cone}((J_{1}\ar[r] & J_{-1})\star_s \Wh_{S})[-1].
 }
\end{equation*}
Expanding in terms of the standard basis, we have 
the reformulation
\begin{equation*}
\xymatrix{
\Phi( \eta^\ell_{s}(\cO_{\Loc^{s}}))  \simeq \mathit{Cone}((T_{0*} T_{1/2}\ar[r] & T_{1/2} T_{0!})\star_s \Wh_{S})[-1].
}
\end{equation*}
Thanks to the distinguished triangles in $\cH^{\aff}_{\PGL(2)}$ given by
\begin{equation*}
\xymatrix{
\d \ar[r] & T_{0!} \ar[r] & \Avg &
\Avg \ar[r] & T_{0*} \ar[r] & \d
}
\end{equation*}
%corresponding to the distinguished triangle
%\begin{equation}
%\xymatrix{
% \calO_{\tilde \calN^\vee} \ar[r] & \calO_{\St_{\dG}}^{cl} (-1, -1) \ar[r] & \calO_{\PP^1 \times \PP^1}(-1, -1)
%}
%\end{equation}
and the fact that $\Avg\star_s\Wh_{S}=0$ as seen in Corollary~\ref{c:Wh asp}, we have the further reformulation
\begin{equation*}
\xymatrix{
\Phi( \eta^\ell_{s}(\cO_{\Loc^{s}})) 
\simeq \Avg  \star_s T_{1/2} \star_s \Wh_{S} = \pi^{*}_{s}\pi_{s*}(T_{1/2} \star_s \Wh_{S})[1].
}
\end{equation*}

Finally, we have an isomorphism 
\begin{equation*}
\xymatrix{
T_{1/2}\star_{s} \Wh_{S} \simeq i_! j_* \const{ c_0(\vn)}
}\end{equation*}
in terms of the open substacks
\begin{equation*}
\xymatrix{
 c_0(\vn)    \ar@{^(->}[r]^-j & 
  c_0(\vn) \cup c_0(0, \infty)   \ar@{^(->}[r]^-i & 
\Bun_{\PGL(2)}(\PP^1, S)
}\end{equation*}
classifying trivial bundles $\calO^{2}_{\PP^1}$ with distinct lines $\ell_0, \ell_1, \ell_\infty$,
and with $\ell_1$ alone distinct. 
From this, we observe an isomorphism 
\begin{equation*}
\xymatrix{
\pi_{s*}(T_{1/2}\star_{s} \Wh_{S})[1] \simeq \Wh_{S\setminus\{s\}}
}\end{equation*}
from which \eqref{pi eta on cO} follows.

(2) The proof for the pair $(\eta^{r}_{s}, \pi^{!}_{s})$ is completely the same as that for $(\eta^{\ell}_{s}, \pi^{*}_{s})$.

(3) By adjunction and the known canonical isomorphism $\pi^{*}_{s}\circ\Phi_{S\setminus\{s\}}\simeq \Phi\circ\eta^{\ell}_{s}$, we get a natural transformation $\Phi_{S\setminus\{s\}}\Rightarrow \pi_{s*} \circ \Phi\circ\eta^{\ell}_{s}$. Precomposing with $\eta_{s}$, we get a natural transformation
\begin{equation*}
\th: \Phi_{S\setminus\{s\}}\circ\eta\Rightarrow \pi_{s*}\circ \Phi\circ\eta^{\ell}_{s}\circ\eta\Rightarrow  \pi_{s*}\circ \Phi.
\end{equation*}
We will show that $\th$ is an equivalence. Note it suffices to show that $\pi^{*}_{s}\th$ is an equivalence because $\pi^{*}_{s}$ is conservative. 

Observe that
\begin{equation*}
\eta^{\ell}_{s}\eta(\cF)=\cO_{\PP^{1}\times\PP^{1}}(-1,-1)[-1]\star_{s}\cF, \quad \textup{ for }\cF\in \Coh(\Loc_{\SL(2)}(\PP^{1}, S)).
\end{equation*}
By the affine Hecke equivariance of $\Phi$ (see Prop \ref{p:Hk equiv}), and the fact that $\cO_{\PP^{1}\times\PP^{1}}(-1,-1)$ corresponds to $\Avg$ under the equivalence $\Phi^{\aff}$, we have
\begin{equation*}
\Phi\circ\eta^{\ell}_{s}\circ\eta =\Phi\circ(\cO_{\PP^{1}\times\PP^{1}}(-1,-1)[-1]\star_{s}) \simeq \Avg[-1]\star_{s}\Phi.
\end{equation*}
On the other hand, the above equivalence of functors is the composition
\begin{equation*}
\xymatrix{\Phi\circ \eta^{\ell}_{s}\circ \eta \ar[rr]^-{\sim} && \pi^{*}_{s}\circ \Phi_{S\setminus\{s\}}\circ\eta \ar[rr]^-{\pi^{*}_{s}\th} &&  \pi^{*}_{s}\circ\pi_{s*}\circ \Phi=\Avg[-1]\star_{s}\Phi}
\end{equation*} 
where the first equivalence is the composition of the identity of $\eta$ and the equivalence
established above in (1). Therefore $\pi^{*}_{s}\th$ is an isomorphism. This completes the proof.  
\end{proof}

This proposition allows us to calculate the image of $\IC_0(\vn), \IC_0(1,\infty), \IC_0(0,\infty)$ and $\IC_0(0,1)$ under $\Phi$, as listed in Section~\ref{sss:U0}. For example, $\IC_{0}(0,\infty)\cong \pi^{*}_{1}\Eis_{0,\{0,\infty\}}$  (the Eisenstein series $\Eis_{0}$ for $\PP^{1}\setminus\{0,\infty\}$).

\subsection{Compatibility with Eisenstein series}

%\subsubsection{Automorphic Eisenstein series}
%On the automorphic side, consider the induction diagram
%\begin{equation}
%\xymatrix{
%\Bun_T(\PP^1) & \ar[l]_-p \Bun_B(\PP^1)  \ar[r]^-q & \Bun_{\PGL(2)}(\PP^1, S) 
%}\end{equation}
%where $p$ is the usual projection, and $q$ assigns to a $B$-bundle the induced $G$-bundle with its given $B$-reduction
%remembered along $S$. We can think of objects of $\Bun_T(\PP^1)$ as line bundles on $\PP^1$, and objects of $\Bun_B(\PP^1)$ as inclusions $\calL \subset \calE$ of a line bundle into a rank 2 vector bundle on $\PP^1$ up to simultaneous tensoring with a line bundle. Then $p$ is given by $\calL^{\otimes 2} \otimes (\det \calE)^{-1}$, and $q$ is given by $\calL|_S \subset \calE$.

%Recall that $p$ induces an isomorphism 
%\begin{equation}
%\xymatrix{
%\pi_0(\Bun_B(\PP^1)) \ar[r]^-\sim & \pi_0(\Bun_T(\PP^1)) \simeq   \Lambda_T \simeq \ZZ
%}
%\end{equation}

In the case $G=\PGL(2)$, we have $\L_{T}=\ZZ$. For $n\in \ZZ$, recall the subdiagram
\begin{equation*}
\xymatrix{
\Bun^n_T(\PP^1) & \ar[l]_-{p_n} \Bun^n_B(\PP^1)  \ar[r]^-{q_n} & \Bun^{\bar n}_{\PGL(2)}(\PP^1, S)
}\end{equation*}
where we fix $n= 2\deg(\calL) - \deg(\calE)$ (and $\bar n = n \mod 2$).  
%stands for  $ev$ when $n$ is even, and $odd$ when $n$ is odd. 

Recall the Eisenstein series sheaf
\begin{equation*}
\xymatrix{
\Eis_n = q_{n!} \const{ \Bun^n_B(\PP^1)}[-n-2]
}\end{equation*}
%Note here $-n-2$ is the dimension of $\Bun^n_B(\PP^1)=\pt/\Aut(\cO(n)\oplus \cO)$.

To describe it, recall we write  $j:c_n(S) \to \Bun^{\bar n}_G(\PP^1, S)$ for the point where $\calE \simeq \calO_{\PP^1}(n) \oplus \calO_{\PP^1}$ with lines $\ell_0, \ell_1, \ell_\infty\subset \calO_{\PP^1}(n)$, and write $\calF_n(S) = j_! \const{c_n(S)}  \in \Sh_!( \Bun^{\bar n}_G(\PP^1, S))$ for the extension by zero of the constant sheaf.
Recall also the special point $j:c_1(\vn) \to \Bun^{\bar 1}_G(\PP^1, S)$ where 
$\calE \simeq \calO_{\PP^1}(1) \oplus \calO_{\PP^1}$ with collinear lines $\ell_0, \ell_1, \ell_\infty\subset \calO_{\PP^1}$, and 
$\calF_1(\vn) = j_! \const{c_1(\vn)}  \in \Sh_!( \Bun^{\bar 1}_G(\PP^1, S))$ is the extension by zero of the constant sheaf.

\begin{lemma}\label{lem: easy eis}
\begin{enumerate}
\item When $n\geq 0$, we have an isomorphism
\begin{equation*}
\xymatrix{
p_n: \Bun^n_B(\PP^1)\ar[r]^-\sim & c_{n}(S) \subset \Bun^{\bar n}_{\PGL(2)}(\PP^1, S)
}\end{equation*}
and hence an isomorphism
\begin{equation*}
\xymatrix{
\Eis_n \simeq \calF_n(S)[-n-2].
}\end{equation*}
\item When $n= -1$, we have an isomorphism
\begin{equation*}
\xymatrix{
p_1: \Bun^{-1}_B(\PP^1)\ar[r]^-\sim & c_1(\vn) \subset \Bun^{\bar 1}_{\PGL(2)}(\PP^1, S)
}\end{equation*}
and hence an isomorphism
\begin{equation*}
\xymatrix{
\Eis_{-1} \simeq \calF_1(\vn)[-1]
}\end{equation*}

\end{enumerate}
\end{lemma}

\begin{proof}
For  $\calE \simeq \calO_{\PP^1}(n) \oplus \calO_{\PP^1}$ with $n\geq -1$, and lines $\ell_0, \ell_1, \ell_\infty\subset \calO_{\PP^1}(n)$, there exists a unique inclusion $\calO_{\PP^1}(n) \subset \calE$ such that $\calO_{\PP^1}(n)|_S$ coincides with the given lines. (In fact, for $n\geq 1$, there exists a unique inclusion independently of the lines.) 
\end{proof}

\subsubsection{Spectral Eisenstein series}

We seek the objects on the spectral side corresponding to Eisenstein sheaves.

Consider the substack  $\Loc_{B^\vee}(\PP^1, S) \subset \Loc_{\SL(2)}(\PP^1, S)$ classifying $B^\vee$-local systems on $\PP^1 \setminus S$ with trivial induced $T^\vee$-monodromy near $S$ (which in this case implies trivial induced 
$T^\vee$-monodromy globally).
It admits the presentation as a quotient
\begin{equation*}
\xymatrix{
\Loc_{B^\vee}(\PP^1, S) \simeq \tdN_{\Delta}/\SL(2)
}\end{equation*}
of the reduced subscheme $\tdN_{\Delta} \subset (\tdN)^{S, \prod=1}$ of the irreducible component $\L_{S}$ from the list of Section~\ref{sss:Loc comp}. In particular, we have the natural $\SL(2)$-equivariant projection $\pi: \tdN_{\Delta}/\SL(2) \to \PP^1/\SL(2)$.

For $n\in\ZZ$, define the spectral Eisenstein series coherent sheaf to be
\begin{equation*}
\xymatrix{
\calO_\Delta(n) = \calO_{\tdN_{\Delta}/\SL(2)}(n)   \simeq \pi^*\calO_{\PP^1/\SL(2)}(n).
}\end{equation*}

\begin{prop}\label{p:Eis} 
For $n\in\ZZ$, we have an isomorphism
\begin{equation*}
\xymatrix{
\Phi(\calO_\Delta(n+1))\simeq \Eis_n.
}\end{equation*}
\end{prop}

\begin{proof} We denote $(\tdN)^{S,\prod=1}$ simply by $\L$ during the proof. We will denote objects in $\Coh(\Loc_{\SL(2)}(\PP^{1},S))\simeq \Coh^{\SL(2)}(\L)$  by their pullbacks to $\L$.

By the construction of $\Phi$, we have $\Phi(\calO_\Delta(n+1))=\Phi(\cO_{\L}(n,0,0)\otimes_{\cO_{\L}}\cO_{\D}(1))\simeq J_{n}\star_{0}\Phi(\cO_{\D}(1))$; on the other hand, by Lemma \ref{l:J on Eis}, $J_{n}\star_{0}\Eis_{0}\simeq \Eis_{n}$. Therefore it suffices to show that
\begin{equation*}
\xymatrix{
\Phi(\calO_\Delta(1))\simeq \Eis_0.
}\end{equation*}

%First, it is elementary to show 
%\begin{equation}
%\xymatrix{
%J_k \star \Eis_n \simeq \Eis_{n+k} & n\geq -1, k\geq 0
%}\end{equation}
%Hence, since $J_k$ is invertible, to prove the proposition, it suffices to construct a single isomorphism

 One direct strategy would be to write $\calO_\Delta(1)$ as a complex of vector bundles, then apply $\Phi$ to the complex,
 and show the resulting complex is isomorphic to $\Eis_0$. Unfortunately, since $\calO_\Delta(1)$ is coherent but not perfect, this would involve infinite complexes. To avoid this complication, we will instead bootstrap off of Proposition~\ref{prop: 2 to 3 pt compatibility} and  express $\calO_\Delta(1)$ in terms of the structure sheaf $\calO_{\L}$ and objects coming from two points of ramification.

First, by construction we have 
\begin{equation*}
\xymatrix{
\Phi(\calO_{\L}(0, 1, 0))\simeq \Phi(J_1 \star_1 \calO_{\L})\simeq J_1 \star_1 \Phi( \calO_{\L}) \simeq J_1 \star_1 \Wh_{S}
}\end{equation*}

Let us describe this sheaf explicitly.
Consider the open substacks
\begin{equation*}
\xymatrix{
 c_0(\vn)    \ar@{^(->}[r]^-j & 
  c_0(\vn) \cup c_0(0, \infty)   \ar@{^(->}[r]^-i & 
U^{\ov 0}  \ar@{^(->}[r]^-u &  \Bun^{\ov 0}_{\PGL(2)}(\PP^1, S)
}\end{equation*}
classifying  bundles $\calE \simeq \calO^{2}_{\PP^1} $ with respectively distinct lines $\ell_0, \ell_1, \ell_\infty$,
 more generally, lines $\ell_0, \ell_1, \ell_\infty$ with the only possible coincidence
$\ell_0 = \ell_\infty$, and finally most generally, any configuration of lines $\ell_0, \ell_1, \ell_\infty$. Then a simple calculation, for example via the identity $J_1  = T_{0*} T_{1/2}$,  shows that
\begin{equation*}
\xymatrix{
J_1 \star_1 \Wh_{S} \simeq u_!i_* j_! \const{ c_0(\vn)}. 
}\end{equation*}

From here on, we will only consider the open substack $U^{\ov 0}\subset \Bun_{\PGL(2)}(\PP^1, S)$, and all sheaves  will be understood to be extensions by zero off of $U^{\ov 0}$.

Let $Y$ be the preimge of the partial diagonals $\Delta_{0,1}\cup\Delta_{1,\infty}\subset (\PP^{1})^{S}$ in $(\tdN)^{S,\prod=1}$. Under the local coordinates introduced in Section~\ref{sss:Loc comp}, $Y$ is given locally by the equation $xy=0$. Therefore $Y=\L_{0,1}\cup \L_{1,\infty}\cup \wt\L_{S}$ , where $\wt\L_{S}$ denotes the non-reduced component (5) in Section~\ref{sss:Loc comp} whose reduced structure is $\L_{S}\simeq \tdN_{\Delta}$. Since $\Delta_{0,1}\cup \Delta_{1,\infty}$ have ideal sheaves $\cO_{(\PP^{1})^{S}}(-1,-1,0)\otimes\cO_{(\PP^{1})^{S}}(0,-1,-1)=\cO_{(\PP^{1})^{S}}(-1,-2,-1)$ within $(\PP^{1})^{S}$, the ideal sheaf $\cI_{Y}$ is a quotient of $\cO_{\L}(-1,-2,-1)$. Using local coordinates, we see that the ideal sheaf of $Y$ in $\L$ is generated by one equation $(a+b)$, which defines the components $\L_{\vn}$ and $\L_{0,\infty}$. This gives  in the heart of $\Coh^{\SL(2)}(\L)$
a short exact sequence
\begin{equation*}
\xymatrix{
0\ar[r] & \calO_{\L_{\vn}\cup \L_{0,\infty}}(-1,-2 ,-1)  \ar[r] &  \calO_{\L}  \ar[r] &  
\calO_{Y}\ar[r] & 0
}\end{equation*}
and its twist
\begin{equation}\label{eq: spectral triangle 1}
\xymatrix{
0\ar[r] & \calO_{\L_{\vn}\cup \L_{0,\infty}}(-1, -1, -1)  \ar[r] &  \calO_{\L}(0, 1, 0)  \ar[r] &  
\calO_{Y}(0, 1, 0)\ar[r] & 0
}\end{equation}

By a similar process, using a Koszul-like resolution of $\cO_{\L}$ as a quotient of $\cO_{Y}$ (locally defined by the equations $x=0, y=0$), we get a filtration of $\cO_{Y}(0,1,0)$ by $\SL(2)$-equivariant coherent subsheaves with associated-graded (from sub to quotient)
\begin{equation*}
\calO_{\D}(1), \quad \calO_{\L_{0,1}}(0, 0, -1)\oplus \calO_{\L_{1,\infty}}(-1, 0, 0), \quad \calO_{\D}(1).
\end{equation*}
In particular, $\cO_{Y}(0,1,0)$ carries an endomorphism $\ep: \cO_{Y}(0,1,0)\surj \cO_{\D}(1)\incl \cO_{Y}(0,1,0)$ such that $\ep^{2}=0$.

% Let us apply $\Phi$ to the triangle \eqref{eq: spectral triangle 1}.

Next we consider the automorphic side.
Consider the respective open and closed substacks
\begin{equation*}
\xymatrix{
a: A = \{ \ell_0 \not = \ell_\infty\}   \ar@{^(->}[r] & U  &
b: B = \{ \ell_0 = \ell_1\}  \cup\{ \ell_1 = \ell_\infty\}    \ar@{^(->}[r] & U  
}\end{equation*} 
We have a short exact sequence of perverse sheaves
\begin{equation}\label{eq: auto triangle 1}
\xymatrix{
0\ar[r] & a_! \const{A} \ar[r] &  i_* j_! \const{c_{0}(\vn)}  \ar[r] &  b_{!}T_B \ar[r] & 0
}\end{equation}
where $T_B$ is a perverse sheaf on $B$. It is easy to see that $b_{!}T_{B}$ has a filtration (as a perverse sheaf) with associated-graded (from sub to quotient)
\begin{equation*}
\Eis_{0}, \quad \IC_0(0, 1) \oplus \IC_0(1, \infty), \quad \Eis_{0}.
\end{equation*}
In particular, $b_{!}T_{B}$ carries an endomorphism $\ep': b_{!}T_{B}\surj \Eis_{0}\incl b_{!}T_{B}$ such that $\ep'^{2}=0$.

%fits into the  short exact sequence 
%\begin{equation}\label{eq: auto triangle 2}
%\xymatrix{
%0\ar[r] & \QQ_B  \ar[r] &  T_B \ar[r] &   \IC_0(S) \simeq \Eis_0   \ar[r] & 0
%}\end{equation}
%and $\QQ_B$ fits into the short exact sequence 
%\begin{equation}\label{eq: auto triangle 2}
%\xymatrix{
% \IC_0(S) \ar[r] & \QQ_B  \ar[r] &  \IC_0(0, 1) \oplus \IC_0(1, \infty)
%}\end{equation}

Recall there is an isomorphism
\begin{equation*}
\xymatrix{
\Phi(\calO_{\L}(0, 1, 0)) \simeq  i_* j_! \const{c_{0}(\vn)}.
}\end{equation*}

\begin{claim} 
There is an isomorphism \begin{equation*}
\xymatrix{
\Phi(\calO_{\L_{\vn}\cup \L_{0,\infty}}(-1, -1, -1)) \simeq a_! \const{A}.
}
\end{equation*}
\end{claim}

\begin{proof}
In the case of $\PP^{1}$ with two punctures $0$ and $\infty$, we may identify $\Loc_{\SL(2)}(\PP^{1},\{0,\infty\})$ with the adjoint quotient $\St_{\SL(2)}/\SL(2)$ of the derived Steinberg variety. In the following we write $\Loc_{\SL(2)}(\PP^{1},\{0,\infty\})$ simply as $\Loc(0,\infty)$. Recall from Example~\ref{ex:aff hecke for pgl2}, $\Phi^{\aff}$ sends the twisted classical structure sheaf $\cO_{\St}^{cl}(-1,-1)$ to $T_{0!}$. Therefore, by the definition of $\Phi_{0,\infty}$ we have
\begin{eqnarray*}
\Phi_{0,\infty}(\cO^{cl}_{\Loc(0,\infty)}(-1,-1)) & =&\Phi^{\aff}(\cO^{cl}_{\St}(-1,-1))\star_{0}\Eis_{0,\{0,\infty\}}\\
 & \simeq & T_{0!}\star_{0}\Eis_{0,\{0,\infty\}}\\
 & \simeq & j_{0,\infty !}\const{U_{0,\infty}}[-1].
\end{eqnarray*}
Here $j_{0,\infty}: U_{0,\infty}\simeq \pt/T\incl \Bun_{\PGL(2)}(\PP^{1},\{0,\infty\})$ is the open point $\cE=\cO_{\PP^{1}}^{2}$ with two distinct lines $\ell_{0},\ell_{\infty}$. 

Since $\eta^{\ell}_{1}(\cO^{cl}_{\Loc(0,\infty)}(-1,-1))\simeq \calO_{\L_{\vn} \cup \L_{0,\infty}}(-1,-1 ,-1 )[-1]$, by Proposition~\ref{prop: 2 to 3 pt compatibility}, we have that
\begin{eqnarray*}
\Phi(\calO_{\L_{\vn} \cup \L_{0,\infty}}(-1,-1 ,-1 )) &\simeq& \Phi(\eta_{1}^\ell(\cO^{cl}_{\Loc(0,\infty)}(-1,-1)))[1] \\
&\simeq &\pi^{*}_{1}\Phi_{0,\infty}(\cO^{cl}_{\Loc(0,\infty)}(-1,-1))[1]\\
&\simeq& \pi_{1}^{*}j_{0,\infty !}\const{U_{0,\infty}}\\
&\simeq& a_!\const{A}.
\end{eqnarray*}
\end{proof}

Taking direct sums, we obtain an isomorphism
\begin{equation*}
\xymatrix{
\Phi(\calO_{\L_{\vn}\cup \L_{0,\infty}}(-1, -1, -1) \oplus \calO_{\L}(0, 1, 0) )  \simeq 
 a_! \const{A} \oplus  i_* j_! \const{c_{0}(\vn)}.
}\end{equation*}

\begin{claim}
The functor $\Phi$ induces a quasi-isomorphism
\begin{equation*}
\xymatrix{
\End(\calO_{\L_{\vn}\cup \L_{0,\infty}}(-1, -1, -1) \oplus \calO_{\L}(0, 1, 0) ) 
\ar[r]^-\sim & \End( a_! \const{A} \oplus  i_* j_! \const{c_{0}(\vn)}).
}
\end{equation*}

\end{claim}

\begin{proof}
First, one can  calculate
\begin{equation*}
\xymatrix{
\End( \calO_{\L}(0, 1, 0) )  \simeq \QQ, & \End(  i_* j_! \const{c_{0}(\vn)})  \simeq \QQ.
}
\end{equation*}
Since both are generated by the identity morphism, $\Phi$ must induce a quasi-isomorphism on them.

Next, we have seen in the previous Claim that

\begin{equation*}
\xymatrix{
\calO_{\L_{\vn} \cup \L_{0,\infty}}(-1,-1 ,-1 ) \simeq \eta^\ell_{1}(\cO^{cl}_{\Loc(0,\infty)}(-1,-1)))[1], 
}
\end{equation*}

\begin{equation*}
\xymatrix{
a_!\const{A} \simeq  \pi_{1}^{*}j_{0,\infty !}\const{U_{0,\infty}}. 
}
\end{equation*}

Thus for any object $\cM$, we have a commutative diagram
\begin{equation*}
\xymatrix{
\Hom(\calO_{\L_{\vn} \cup \L_{0,\infty}}(-1,-1 ,-1 ), \cM) \ar[r]^-\Phi \ar[d]_-\sim & 
\Hom( a_!\const{A}, \Phi(\cM)) \ar[d]_-\sim \\
\Hom(\cO^{cl}_{\Loc(0,\infty)}(-1,-1)[1], \eta_1\cM) \ar[r]^-{\Phi_{0, \infty}} & 
\Hom( j_{0,\infty !}\const{U_{0,\infty}},\pi_{1*}  \Phi(\cM)) 
}
\end{equation*}
where the vertical equivalences are by adjunction. Since the bottom arrow is an equivalence, the top arrow must be as well.
In particular, we can apply this for $\cM  \simeq \calO_{\L_{\vn}\cup \L_{0,\infty}}(-1, -1, -1) \oplus \calO_{\L}(0, 1, 0) $.

Finally, a similar argument using the respective right adjoints $\eta_1^r, \pi^!$,  shows  for any object $\cM$, that $\Phi$ induces
an equivalence
\begin{equation*}
\xymatrix{
\Hom(\cM, \calO_{\L_{\vn} \cup \L_{0,\infty}}(-1,-1 ,-1 )) \ar[r]^-\sim  & 
\Hom( \Phi(\cM), a_!\const{A}).  
}
\end{equation*}
Again, we can apply this for $\cM  \simeq \calO_{\L_{\vn}\cup \L_{0,\infty}}(-1, -1, -1) \oplus \calO_{\L}(0, 1, 0) $.

This concludes the proof of the claim.
\end{proof}

\begin{claim} The functor $\Phi$ applied to the sequence \eqref{eq: spectral triangle 1} gives  the sequence  \eqref{eq: auto triangle 1}. In particular, we have an isomorphism 
\begin{equation}\label{OY TB}
\Phi(\cO_{Y}(0,1,0))\simeq b_{!}T_{B}.
\end{equation}
Moreover,
the functor $\Phi$ takes the endomorphism
$\ep$ of $\cO_{Y}(0,1,0)$  to a nonzero multiple of the endomorphism $\ep'$  of $b_{!}T_{B}$.
\end{claim} 
\begin{proof} 
We have seen that
 \begin{equation*}
\xymatrix{
\Phi(\calO_{\L_{\vn}\cup \L_{0,\infty}}(-1, -1, -1)) \simeq a_! \const{A},
 & \Phi(\calO_{\L}(0, 1, 0)) \simeq  i_* j_! \const{c_{0}(\vn)}.
}\end{equation*}

One can  calculate
\begin{equation*}
\xymatrix{
\Hom(\calO_{\L_{\vn}\cup \L_{0,\infty}}(-1, -1, -1),  \calO_{\L}(0, 1, 0)) \simeq \QQ \oplus \QQ[-1], 
}\end{equation*}
\begin{equation*}
\xymatrix{
\Hom(a_!\const{A}, i_* j_! \const{c_{0}(\vn)}) \simeq \QQ \oplus \QQ[-1]. 
}\end{equation*}
Note that each is one-dimensional in degree $0$.

By the previous Claim,  the first morphism of \eqref{eq: spectral triangle 1} is taken to a nonzero morphism. 
Since this morphism and the first morphism of \eqref{eq: auto triangle 1}  are nonzero elements of a one-dimensional vector space, each is a nonzero scale of the other. This implies $\Phi$ takes the sequence  \eqref{eq: spectral triangle 1} to the sequence  \eqref{eq: auto triangle 1}, and in particular, passing to cones gives the isomorphism \eqref{OY TB}. 

Furthermore, the previous Claim also implies the functor $\Phi$ induces a quasi-isomorphism on endomorphisms of the cones
\begin{equation}\label{eq endos}
\xymatrix{
\End(\cO_{Y}(0,1,0) ) 
\ar[r]^-\sim & \End( b_{!}T_{B}).
}
\end{equation}
One can  calculate  the degree $0$ endomorphisms  on both sides of \eqref{eq endos} to see each is isomorphic to the dual numbers with respective  generators $\ep$ and $\ep'$.  Thanks to the quasi-isomorphism~\eqref{eq endos}, this implies
$\Phi$ takes 
$\ep$  to a nonzero multiple of $\ep'$. 

This completes the proof of the claim.
\end{proof}

To complete the proof of the Proposition, 
introduce the quotient categories
\begin{equation*}
\xymatrix{
\ov C=\QCoh(\Loc_{\SL(2)}(\PP^{1},S))/\jiao{\cO_{\L_{0,1}}(0,0,-1), \cO_{\L_{1,\infty}}(-1,0,0)},
}
\end{equation*}
\begin{equation*}
\xymatrix{
\ov \Sh = \Sh(\Bun_{\PGL(2)}(\PP^{1},S))/\jiao{\IC_{0}(0,1)\oplus \IC_{0}(1,\infty)}.
}
\end{equation*}
By \eqref{IC01} and \eqref{IC1inf}, $\Phi$ induces a continuous functor
\begin{equation*}
\xymatrix{
\ov \Phi: \ov C\ar[r] &  \ov \Sh.
}
\end{equation*}
Let $\cK$ be the image of $\cO_{Y}(0,1,0)$ in $\ov C_1$; let $\cT$ be the image of $b_{!}T_{B}$ in $\ov \Sh$. By \eqref{OY TB}, we have
\begin{equation*}
\ov \Phi(\cK)\simeq \cT. 
\end{equation*}

%Finally, to complete the proof of the proposition, 
%let us introduce the more refined quotient categories
%\begin{equation*}
%\xymatrix{
%\ov C_1=\QCoh(\Loc_{\SL(2)}(\PP^{1},S))/\jiao{\cO_{\L_{0,1}}(0,0,-1), \cO_{\L_{1,\infty}}(-1,0,0)}
%}
%\end{equation*}
%\begin{equation*}
%\xymatrix{
%\ov \Sh_1 = \Sh(\Bun_{\PGL(2)}(\PP^{1},S))/\jiao{\IC_{0}(0,1)\oplus \IC_{0}(1,\infty)}
%}
%\end{equation*}
%By \eqref{IC01} and \eqref{IC1inf}, $\Phi$ induces a continuous functor
%\begin{equation*}
%\xymatrix{
%\ov \Phi_1: \ov C_1\ar[r] &  \ov \Sh_1.
%}
%\end{equation*}
%Let $\cK_1$ be the image of $\cO_{Y}(0,1,0)$ in $\ov C_1$; let $\cT_1$ be the image of $b_{!}T_{B}$ in $\ov \Sh_1$. By \eqref{OY TB}, we have
%\begin{equation*}
%\ov \Phi_1(\cK_1)\simeq \cT_1. 
%\end{equation*}
Inside of $\ov C$, the image of $\cO_{\D}(1)$ is represented by the infinite complex (the last nonzero entry is in degree $0$)
\begin{equation}\label{K res}
\cdots\to\cK\xr{\ov\ep}\cK\xr{\ov\ep} \cK\xr{0}0\to\cdots.
\end{equation}
where $\ov\ep$ is the endomorphism of $\cK$ induced by the endomorphism $\ep$ of $\cO_{Y}(0,1,0)$.

Inside $\ov \Sh$, the image of $\Eis_{0}$ is represented by the infinite complex of perverse sheaves (the last nonzero entry is in degree $0$)
\begin{equation}\label{T res}
\cdots\to\cT\xr{\ov\ep'}\cT\xr{\ov\ep'} \cT\xr{0}0\to\cdots.
\end{equation}
where $\ov\ep'$ is the endomorphism of $\cT$ induced by the endomorphism $\ep'$ of $b_{!}T_{B}$.

By the previous claim, the continuous functor $\ov \Phi$ sends \eqref{K res} to \eqref{T res}.  Therefore the image of $\Phi(\cO_{\D}(1))$ in $\ov \Sh$ is the same as the image of $\Eis_{0}$. In particular, 
\begin{equation}\label{OD in IC1}
\Phi(\cO_{\D}(1))\subset \jiao{\IC_{0}(0,1),\IC_{0}(1,\infty), \Eis_{0}}.
\end{equation}
The same argument can be applied when the point $1\in S$ is replaced by $0$ or $\infty$, and we get
\begin{eqnarray}\label{OD in IC0}
\Phi(\cO_{\D}(1))\subset \jiao{\IC_{0}(0,1),\IC_{0}(0,\infty), \Eis_{0}},\\
\label{OD in ICinf}\Phi(\cO_{\D}(1))\subset \jiao{\IC_{0}(0,\infty),\IC_{0}(1,\infty), \Eis_{0}}.
\end{eqnarray}
Since the intersection of the categories on the right sides of \eqref{OD in IC1}, \eqref{OD in IC0} and \eqref{OD in ICinf} consists of sheaves supported at the point $c_{0}(S)$, we conclude that $\Phi(\cO_{\D}(1))$ is supported at $c_{0}(S)$.

Finally, using the compatibility of $\Phi$ with changing levels, we can calculate the push forward of $\Phi(\cO_{\D}(1))$ under $\pi_{1}: \Sh(\Bun_{\PGL(2)}(\PP^{1},S))\to \Sh(\Bun_{\PGL(2)}(\PP^{1},\{0,\infty\}))$. By Proposition \ref{prop: 2 to 3 pt compatibility} we have
\begin{equation}\label{pi Phi}
\pi_{1*}\Phi(\cO_{\D}(1))\simeq \Phi(\eta_{1}(\cO_{\D}(1))).
\end{equation}
Since we will be changing the level structure, we use $\tdN_{\D,\{0,\infty\}}$ and $\cO_{\D,\{0,\infty\}}$ to denote the analogues of $\tdN_{\D}$ and $\cO_{\D}$ when $S$ is replaced by $\{0,\infty\}$. We have a commutative diagram where the left parallelogram is derived Cartesian
\begin{equation*}
\xymatrix{ & \tdN_{\D,\{0,\infty\}}/\SL(2)\ar[r]^-{\th'} \ar[drr]^-{p'_{1}}\ar[dl]_-{q'_{1}} & \Loc_{\SL(2)}(\PP^{1}, \{0,\infty\},\{1\})\ar[dl]_-{q_{1}}\ar[dr]^-{p_{1}}\\
\tdN_{\D}/\SL(2)\ar[r]^-{\th} & \Loc_{\SL(2)}(\PP^{1},S) & & \Loc_{\SL(2)}(\PP^{1}, \{0,\infty\})}
\end{equation*}
Then we have
\begin{equation}\label{eta OD}
\eta_{1}(\cO_{\D}(1))\simeq p_{1*}(q_{1}^{*}\th_{*}\cO_{\D}(1)\otimes \cO_{\L}(0,-1,0)) \simeq p_{1*}(\th'_{*}q_{1}'^{*}\cO_{\D})\simeq\cO_{\D,\{0,\infty\}}
\end{equation}
By Lemma \ref{l:2 pts cal}, $\Phi_{0,\infty}$ sends $\cO_{\D,\{0,\infty\}}$  (which is the same as $\D^{-}_{*}\cO_{\tdN}$ in the notation of  Lemma \ref{l:2 pts cal}) to the Eisenstein sheaf $\Eis_{0,\{0,\infty\}}$.  Combining \eqref{pi Phi} and \eqref{eta OD}, we have
\begin{equation*}
\pi_{1*}\Phi(\cO_{\D}(1))\simeq \Eis_{0,\{0,\infty\}}.
\end{equation*}
Since $\Phi(\cO_{\D}(1))$ is supported on $c_{0}(S)$ which is mapped isomorphically onto its image under $\pi_{1}$, we conclude that $\Phi(\cO_{\D}(1))\simeq \Eis_{0}$. This completes the proof of the Proposition.
\end{proof}

%\begin{remark}
%In fact, we will have more generally an isomorphism
%\begin{equation}
%\xymatrix{
%\Phi(\calO_\Delta(n+1))\simeq \Eis_n
%}\end{equation}
% for any $n\in \ZZ$.
% 
%\end{remark}

%%%%%%%%%%%%%%%%%%%%%%%%%%%%%%%%%%%%%%%%%%%%%%%%%%%%%

\subsection{Newforms} 
%On the automorphic side, for $s\in S$, recall the natural $\PP^1$-fibration
% \begin{equation}
%\xymatrix{ 
%\pi: \Bun_{\PGL(2)}(\PP^1, S) \ar[r] & \Bun_{\PGL(2)}(\PP^1, S\setminus \{s\})
%}
%\end{equation}
%where we forget the flag at $s\in S$,
%and the resulting adjoint triple
%\begin{equation}
%\xymatrix{
%\Sh_!(\Bun_{\PGL(2)}(\PP^1, S))
%\ar[rr]^-{\pi_{s*}=\pi_{s!}} &&
% \Sh_{!}(\Bun_{\PGL(2)}(\PP^1, S \setminus \{s\}) \ar@<-2.5ex>[ll]_-{\pi_s^{*}}\ar@<2.5ex>[ll]_-{\pi_s^{!}}}
%\end{equation}
%in which $\pi_s^! \simeq \pi_s^*[1]$.

On the automorphic side, for $s\in S$, define $\Sh_s \subset \Sh_!( \Bun_{\PGL(2)}(\PP^1, S))$ to be the full subcategory generated by the image of $\pi_s^*$.

Define the dg category of newforms to be the dg quotient 
\begin{equation*}
\xymatrix{
\Sh^{\new} = \Sh_!( \Bun_{\PGL(2)}(\PP^1, S))/\langle \Sh_0, \Sh_1, \Sh_\infty\rangle
}
\end{equation*}
where we kill all ``old forms" coming from fewer points of ramification.

%Similarly, on the spectral side, recall 
% the natural correspondence
%\begin{equation}
%\xymatrix{ 
%& \Loc_{\SL(2)}(\PP^1, S\setminus \{s\}, \{s\}) \ar[dl]_-{i} \ar[d]_-\kappa \ar[rd]^-{p}& \\
%\Loc_{\SL(2)}(\PP^1, S) & \PP^1/\SL(2) &  \Loc_{\SL(2)}(\PP^1, S\setminus \{s\}) 
%}
%\end{equation}
%and the resulting adjoint triple 
%%\begin{equation}
%%\xymatrix{ 
%%\eta: \Coh(\Loc_{\SL(2)}(\PP^1, S))
%%\ar[r] & 
%%\Coh(\Loc_{\SL(2)}(\PP^1, \{0, 1\}))
%% &
%% \eta(\calF) = p_*(\kappa^*\calO_{\PP^1}(-1) \otimes i^*\calF)}
%%\end{equation}
%%It extends to an adjoint triple
%\begin{equation}
%\xymatrix{
% \Coh(\Loc_{\SL(2)}(\PP^1, S))
%\ar[rr]^-{\eta} &&
% \Coh(\Loc_{\SL(2)}(\PP^1, S\setminus \{s\})) \ar@<-2.5ex>[ll]_-{\eta^\ell}\ar@<2.5ex>[ll]_-{\eta^r}
% }
%\end{equation}
%\begin{equation}
%\xymatrix{
% \eta(\calF) = p_*(\kappa^*\calO_{\PP^1}(?)[?] \otimes i^*\calF) 
% }
%\end{equation}
%\begin{equation}
%\xymatrix{
% \eta^\ell = i_*(\kappa^*\calO_{\PP^1}(?)[?] \otimes p^*\calF)  &   \eta^r = i_*(\kappa^*\calO_{\PP^1}(?)[?] \otimes p^*\calF) 
% }
%\end{equation}

Similarly, on the spectral side, define $C_s \subset  \Coh(\Loc_{\SL(2)}(\PP^1, S))$ to be the full subcategory generated by the image of $\eta^\ell_{s}$.

Define the dg quotient category 
\begin{equation*}
\xymatrix{
C^{\new} =  \Coh(\Loc_{\SL(2)}(\PP^1, S))/\langle C_0, C_1, C_\infty\rangle.
}
\end{equation*}

%
%Since $H(\calG)$ is natural in  $\calG\in\Coh^{\dG}(\calS(\LL))$, it suffices to show that $H(\calG)$ is an isomorphism for $\calG$ in a generating set of objects for the triangulated category $\Coh^{\dG}(\calS(\LL))$.
%
\begin{lemma}\label{l:gennew}
\begin{enumerate}
\item $\Sh^{\new}$ is generated by $\Eis_{n}$, for $ n\geq-1$.
\item $C^{\new}$ is generated by $\calO_\Delta(n)$, for $n\geq0$.
\end{enumerate}
\end{lemma}
\begin{proof}

(1) Set $\Sh^\old = \langle \Sh_0, \Sh_1, \Sh_\infty\rangle$. We only need to exhibit a set of generators for $\Sh_{!}(\Bun_{\PGL(2)}(\PP^1, S))$ whose members are either in $\Sh^{\old}$, or a (shifted) Eisenstein sheaf. Such a set of generators is given by:
\begin{itemize}
\item For $n\geq 2$, $\IC_{n}(R) \in \Sh^\old$  when $R\not = S$;  by Lemma~\ref{lem: easy eis}, we also have $\calF_n(S)\simeq \Eis_n[-n-2]$.
\item For $n=1$,  $\IC_1(R) \in \Sh^\old$  when $R\not = \vn$ or $S$; 
by Lemma~\ref{lem: easy eis}, we also have $ \calF_1(\vn)\simeq \Eis_{-1}[-1] $
and $ \calF_1(S)\simeq \Eis_1[-3] $.
\item For $n=0$, $\IC_0(R) \in \Sh^\old$  when $R\not = S$;  by Lemma~\ref{lem: easy eis}, we also have $ \calF_0(S)\simeq \Eis_{0}[-2]$.
\end{itemize}

%Recall the basis introduced in Section~\ref{}. 

\medskip

(2) Let $C^{\old}=\jiao{C_{0},C_{1},C_{\infty}}$ and let $C'=\jiao{C^{\old}, \cO_{\Delta}(n) ; n\ge0}$. Our goal is to show that  $C'=\Coh(\Loc_{\SL(2)}(\PP^{1},S))$.

We will use the following well-known fact. Let $Y$ be a stack of finite type, $i:Z\incl Y$  a closed substack and  $j:U=Y-Z\incl Y$  the open complement of $Z$. Then $j^{*}$ induces an equivalence $\Coh(Y)/\Coh_{Z}(Y)\simeq \Coh(U)$, where $\Coh_{Z}(Y)$ is the dg subcategory of $\Coh(Y)$ generated by the image of $i_{*}: \Coh(Z)\to \Coh(Y)$. 

Using the above fact and induction, one can show the following statement which we label (\dag)

(\dag): Suppose a stack $Y$ of finite type is stratified into a union of finitely many strata $Y_{\a} \subset Y$, for $\alpha$ in some index set $A$. Suppose for each $\a \in A$, we have a collection of objects $\cF^{(i)}_{\a}\in\Coh(\ov Y_{\a})$, for $i$ in some index set $I_{\a}$, such that $\{\cF^{(i)}_{\a}|_{Y_{\a}};i\in I_{\a}\}$ generate $\Coh(Y_{\a})$. Then the collection $\{\cF_{\a}; \a \in A, i\in I_{\a}\}$ generate $\Coh(Y)$.

We will also use the following additional simple observation we  label (\ddag):
 
(\ddag):  Let $Y$ be an affine scheme with an action of an affine group $H$. Then $\Coh^{H}(Y)$ is generated by objects of the form $V\otimes\cO_{Y}$, where $V$ runs over all finite-dmensional irreducible representations of $H$, and the $H$-equivariant structure on $V\otimes \cO_{Y}$ is given by the diagonal action of $H$.

In Section~\ref{sss:Loc comp} we listed the irreducible components of $(\tdN)^{S,\prod=1}$, and denoted  their reduced structure by $\L_{R}$, for subsets $R\subset S$ such that $\#R\ne1$. We know that $\L_{R}$ is the conormal bundle to the partial diagonal $\D_{R}\subset (\PP^{1})^{S}$.

Let us now stratify $(\tdN)^{S,\prod=1}$ by taking the intersections of the components $\L_{R}$.  By the above statements (\dag) and (\ddag),  it is enough to exhibit a collection of objects in $C'$ on the closure of each stratum whose restrictions to that stratum generate the $\dG$-equivariant derived category of coherent sheaves on that stratum.

The $3$-dimensional strata are the opens $\L^{\circ}_{R}=\L_{R}\setminus \cup_{R'\ne R}\L_{R'}$, for $R\subset S, \#R\ne1$.  Let us describe the quotients $\L^{\circ}_{R}/\dG$, along with a set of objects of $C'$ whose restrictions generate $\Coh^{\dG}(\L^{\circ}_{R})$. 
\begin{enumerate}
\item $\L_{\varnothing}^{\circ}/\dG\simeq \pt/\mu_{2}$, where $\mu_{2}$ is the center of $\dG$. By (\ddag), $\Coh^{\dG}(\L^{\circ}_{R})$ is generated by two elements $\cO_{\L^{\circ}_{R}}$ and $\sgn\otimes \cO_{\L^{\circ}_{R}}$, where $\sgn$ denotes the sign representation of $\mu_{2}$. Therefore, for $R=\varnothing$, 
the restrictions of $\cO_{\L_{\varnothing}}(-1,-1,0), \cO_{\L_{\varnothing}}(-1,0,0)\in C_{0} \subset C^\old \subset C'$ 
to $\L^{\circ}_{\varnothing}$ generate $\Coh^{\dG}(\L^{\circ}_{\varnothing})$. 

\item When $\#R=2$, $\L_{R}^{\circ}/\dG\simeq (\dN\setminus \{1\})/T^{\vee}\simeq \pt/\mu_{2}$. Note $\L_{R}\simeq \PP^{1}\times \tdN$. By the same argument as in the previous case,  
the restrictions of $\cO_{\PP^{1}}(-1)\boxtimes\cO_{\tdN}(1), \cO_{\PP^{1}}(-1)\boxtimes\cO_{\tdN}\in C^{\old}\subset C'$   to $\L^{\circ}_{R}$ generate $\Coh^{\dG}(\L^{\circ}_{R})$. 

\item $\L_{S}^{\circ}/\dG\simeq Y'/\dB$, where $Y'=(\dN\setminus \{1\})^{S,\prod=1}$, and the action of $\dB$ factors through $\dT$. Therefore $\L_{S}^{\circ}/\dG\simeq (\AA^{1}\setminus \{0,1\})\times (\pt/(\Ga\times \mu_{2}))$. Again by (\ddag), $\Coh^{\dG}(\L^{\circ}_{S})$ is generated by two elements $\cO_{\L^{\circ}_{S}}$ and $\sgn\otimes \cO_{\L^{\circ}_{S}}$. Therefore, the restrictions of $\cO_{\Delta}$ and $\cO_{\Delta}(1) \in C'$  to $\L^{\circ}_{S}$ generate $\Coh^{\dG}(\L^{\circ}_{S})$. 
\end{enumerate}

The $1$-dimensional stratum is the intersection of all $\L_{R}$ given by  the diagonal $\Delta_{S}\subset (\PP^{1})^{S}$.  
We will return to it momentarily.

For $\#R=2$, we have $\Delta_{R}=\L_{\varnothing}\cap \L_{R}$,  and set $\Delta_{R}^{\circ}=\Delta_{R}\setminus \Delta_{S}$; we also set  $\Theta_{R}=\L_{S}\cap \L_{R}\simeq\tdN$,  and $\Theta^{\circ}_{R}=\Theta_{R}\setminus \Delta_{S}$. Then 
the 2-dimensional strata are $ \Delta^{\circ}_{R}, \Theta^{\circ}_{R},$ for $\#R=2$. 
Let us describe their quotients by $\dG$, along with a set of objects of $C'$ whose restrictions generate 
 equivariant coherent sheaves.

\begin{enumerate}
\item $\Delta^{\circ}_{R}/\dG\simeq \pt/\dT$. Write $\D_{R}$ as $\PP^{1}\times \PP^{1}$. By (\ddag), $\Coh^{\dG}(\Delta^{\circ}_{R})$ is generated by the restrictions of $\cO_{\PP^{1}\times\PP^{1}}(-1,n)$, for all $n\in\ZZ$, which all lie in $C^{\old}\subset C'$.

\item $\Theta^{\circ}_{R}/\dG\simeq \pt/(\Ga\times \mu_{2})$. Note the canonical projection $\Theta_{R}\simeq \tdN \to \PP^{1}$, providing the line bundles $\cO_{\Theta_{R}}(n)$, for $n\in \ZZ$. By (\ddag), $\Coh^{\dG}(\Theta^{\circ}_{R})$ is generated by  the restrictions of $\cO_{\Theta_{R}}$ and $\cO_{\Theta_{R}}(1)$. Note that $\Theta_{R}\subset \L_{S}=\tcN_{\Delta}$ is a $\dG$-invariant line sub-bundle in the two-dimensional vector bundle $\tcN_{\Delta} \simeq \cO_{\PP^{1}}(-2)^{\oplus 2}$ over $\PP^{1}$. Therefore we have an exact sequence of $\dG$-equivariant coherent sheaves $0\to \cO_{\Delta}(2)\to \cO_{\Delta}\to \cO_{\Theta_{R}}\to 0$. This shows that $\cO_{\Theta_{R}}\in C'$. Similarly, $0\to \cO_{\Delta}(n+2)\to \cO_{\Delta}(n)\to \cO_{\Theta_{R}}(n)\to 0$ implies $\cO_{\Theta_{R}}(n)\in C'$, for any $n\ge0$. 
\end{enumerate}

Finally let us show that $\cO_{\Delta_{S}}(n)\in C'$, for all $n\in\ZZ$. This will  complete the proof by  providing a generating set for $\Coh^{\dG}(\Delta_{S})$, where recall $\Delta_S \subset (\PP^1)^S$ is the closed 1-dimensional stratum.  Since $\Delta_{S}$ is the zero section of $\Theta_{R}$, for any $\#R=2$, we have a $\dG$-equivariant exact sequence $0\to \cO_{\Theta_{R}}(n+2)\to \cO_{\Theta_{R}}(n)\to \cO_{\Delta_{S}}(n)\to 0$. Since we have already shown that $\cO_{\Theta_{R}}(n)\in C'$, for all $n\ge0$, we also have $\cO_{\Delta_{S}}(n)\in C'$, for all $n\ge0$. Now pick any $\#R=2$ and write $\Delta_{R}=\PP^{1}\times\PP^{1}$ and regard $\Delta_{S}$ as the diagonal. Consider the $\dG$-equivariant exact sequence
\begin{equation}\label{nth nghb}
\xymatrix{
0\ar[r] & \cO_{\Delta_{R}}(-1,-n-2)\ar[r] &\cO_{\Delta_{R}}(n,-1)\ar[r] &\cE_{n} \ar[r] & 0
}\end{equation}
obtained by restricting $\cO_{\Delta_{R}}(n,-1)$ to the $n$-th infinitesimal neighborhood of the diagonal $\Delta_{S}$. Then $\cE_{n}$ (which is topologically supported on $\Delta_{S}$) is a successive extension of $\cO_{\Delta_{S}}(n-1), \cO_{\Delta_{S}}(n-3), \cdots,\cO_{\Delta_{S}}(-n-1)$ (each time the twisting decreases by $2$). We have $\cE_{n}\in C^{\old}$, for any $n\ge0$, by \eqref{nth nghb} because the first two terms are in $C^{\old}$. We have already shown that $\cO_{\Delta_{S}}(n)\in C'$, for any $n\ge 0$. Using that $\cE_{n}\in C'$,  for any $n\ge0$, we conclude that $\cO_{\Delta_{S}}(n)\in C'$, for all $n<0$. This completes the proof.
\end{proof}

%%%%%%%%%%%%%%%%%%%%%%%%%%%%%%%%%%%%%%%%%%%%%%%%%%%%

\subsection{Equivalence}

\begin{prop}\label{p:Coh!} 
Proposition~\ref{prop: extension images} holds:
 $\Phi_\Coh$  lands in
$  \Sh_!(\Bun_{\PGL(2)}(\PP^1, S)).
%\subset  \Sh_!(\Bun_{\PGL(2)}(\PP^1, S))
$

Moreover, $\Phi_\Coh$ is essentially surjective onto $  \Sh_!(\Bun_{\PGL(2)}(\PP^1, S)).
%\subset  \Sh_!(\Bun_{\PGL(2)}(\PP^1, S))
$

\end{prop}

\begin{proof}
We continue with the notation introduced in the previous section.

By Proposition \ref{prop: 2 to 3 pt compatibility}, we have $\Phi|_{C_{s}}:C_{s}\to \Sh_{s}$, for $s\in S$. Moreover, it is essentially surjective since $\Phi_{S\setminus\{s\}}$ is essentially surjective. 

Therefore $\Phi$ induces a functor 
\begin{equation*}
\xymatrix{
\Phi^{\new}:C^{\new}\ar[r] &  \Sh(\Bun_{\PGL(2)}(\PP^1, S))/\jiao{\Sh_{0},\Sh_{1},\Sh_{\infty}}.
}\end{equation*}

 It suffices to show that $\Phi^{new}$ has image exactly equal to $\Sh^{\new}$. By Proposition \ref{p:Eis}, 
 we have 
 \begin{equation*}
\xymatrix{
\Phi^\new(\calO_\Delta(n+1))\simeq \Eis_n, & n\geq -1.
}\end{equation*}
Thus by Lemma \ref{l:gennew}(2), the image of $\Phi^{\new}$ lies in $\Sh^{\new}$, and by Lemma \ref{l:gennew}(1), it is essentially surjective onto $\Sh^{\new}$.
\end{proof}

Now we are ready to prove our main theorem for $G=\PGL(2)$.

\begin{theorem}
The functor $\Phi_\Coh$ provides an equivalence
\begin{equation*}
\xymatrix{
 \Coh ( \Loc_{\SL(2)}(\PP^1, S)) \ar[r]^-\sim & \Sh_!(\Bun_{\PGL(2)}(\PP^1, S))
}
  \end{equation*}
compatible with the affine Hecke actions at $s\in S$. 

It restricts to equivalences
\begin{equation*}
\xymatrix{
 \Coh^{\triv} ( \Loc_{\SL(2)}(\PP^1, S)) \ar[r]^-\sim & \Sh_!(\Bun^{\ov 1}_{\PGL(2)}(\PP^1, S))
}
  \end{equation*}
\begin{equation}\label{eq:Phi alt}
\xymatrix{
 \Coh^{\alt} ( \Loc_{\SL(2)}(\PP^1, S)) \ar[r]^-\sim & \Sh_!(\Bun^{\ov 0}_{\PGL(2)}(\PP^1, S)).
}
  \end{equation}
\end{theorem}

\begin{proof} Compatibility of $\Phi_{\Coh}$ with the affine Hecke actions follow from Proposition \ref{p:Hk equiv}.

Thanks to Proposition~\ref{p:Coh!}, it remains to show:
for $\calF , \calG\in  \Coh ( \Loc_{\SL(2)}(\PP^1, S))$, the natural homomorphism
\begin{equation}\label{eq:Hom FG}
\xymatrix{
\Hom_{  \Coh ( \Loc_{\SL(2)}(\PP^1, S))}(\calF,\calG)\ar[r] & \Hom_{ \Sh_!(\Bun_{\PGL(2)}(\PP^1, S))}(\Phi\calF,\Phi\calG)
}
\end{equation}
is a quasi-isomorphism.

We will make a series of reductions. We introduce the abbreviation $\Sh_! = \Sh_{!}(\Bun_{\PGL(2)}(\PP^1, S))$ and $\Loc=\Loc_{\SL(2)}(\PP^1, S)$. 

First, by continuity, we may assume that $\calF=\calO_{\Loc}(a,b,c)\otimes V$, i.e., the tensor of a line bundle and an $\SL(2)$-representation. Then the left hand side of ~\eqref{eq:Hom FG} takes the form
\begin{equation*}
\xymatrix{
\Hom_{  \Coh ( \Loc)}(\calF,\calG)\simeq \Gamma(\Loc_{\SL(2)}(\PP^1, S), \calO_{\Loc}(-a,-b,-c)\otimes V^\vee \otimes \calG )
}
\end{equation*}

Second,  by construction, we have
%From the construction of the actions of $\Coh^{\dG}(\tdN)$ on $D$ we have the adjunctions
%\begin{equation}
%(*V, *V^{\vee}), (*_{x}J_{a}, *_{x}J_{-a}) \textup{for }V\in \Rep(\dG), a\in \ZZ, x\in\{0,1,\infty\}.
%\end{equation}
%Therefore
\begin{eqnarray*}
\Hom_{\Sh_!}(\Phi\calF,\Phi\calG)&\simeq &
\Hom_{Sh_!}( \Hecke^{\sph}_{u_{0}}(V,  J_a \star_0 J_b \star_1 J_c \star_\infty \Wh_{S}), \Phi\calG)\\
&\simeq&
\Hom_{Sh_!}(  \Wh_{S}, \Hecke^{\sph}_{u_{0}}(V^\vee,  J_{-a} \star_0 J_{-b} \star_1 J_{-c} \star_\infty \Phi\calG))\\
&\simeq&
\Hom_{Sh_!}(  \Wh_{S},  \Phi(\calO_{\Loc}(-a,-b,-c)\otimes V^\vee \otimes \calG))
\end{eqnarray*}
where $u_{0}\in \PS$ is a base point.

Thus we may reduce to the case $\calF=\calO_{\Loc}$, and would like to show that the natural map
\begin{equation*}
\xymatrix{
\Gamma(\Loc_{\SL(2)}(\PP^1, S), \calG ) \ar[r] & \Hom_{ \Sh_!}(\Wh_{S},\Phi\calG)
}
\end{equation*}
is a quasi-isomorphism

Now the global sections functor $\Gamma(\Loc_{\SL(2)}(\PP^1, S),-)$ factors through $C^{\new}$ since objects in $C_{s}$, for $s\in S$, have a factor $\calO_{\PP^{1}}(-1)$ whose global sections must vanish. 

Similarly, since $\pi_{s!}\Wh_{S}=0$ by Lemma \ref{l:Avg Wh 0}, the functor $\Hom_{\Sh_!}(\Wh_{S},-)$ factors through $\Sh^{\new}$.  Furthermore, by Proposition \ref{prop: 2 to 3 pt compatibility}, we have $\Phi(C_{s})\subset \Sh_{s}$, for $s\in S$. Thus the functor $\Hom_{\Sh_!}(\Wh_{S},\Phi(-))$ factors through $C^{\new}$.
 
 %Since both the source and the target of the natural transformation factor through $C^{\new}$. 
 Hence by Lemma \ref{l:gennew}(2), it suffices to assume $\calG = \calO_{\Delta}(n)$, for $n\geq0$. 
 
 For $n=0$, both sides are canonically quasi-isomorphic to $\QQ$ and we claim the morphism is a quasi-isomorphism.
 Equivalently, applying Wakimoto symmetry, we claim the induced morphism  
\begin{equation}\label{eq:nontrivial morph}
\xymatrix{
 \Hom_{  \Coh ( \Loc)}(\calO(0, 1, 0),\calO_{\tdN}(1))\ar[r] & \Hom_{ \Sh_!}(J_1\star_1 \Wh_S, \Eis_{0}) 
}
\end{equation}
is a quasi-isomorphism. Returning to the proof and notation of Proposition~\ref{p:Eis}, observe the  left hand side of ~\eqref{eq:nontrivial morph} is generated by the composition
\begin{equation*}
\xymatrix{
 \calO(0, 1, 0) \ar[r]^-\sigma &  \calO_Y(0, 1, 0) \ar[r]^-{\tilde \epsilon} & \calO_{\tdN}(1)
 }
\end{equation*}
such that the endomorphism $\epsilon:\calO_Y(0, 1, 0) \to \calO_Y(0, 1, 0)$ therein is the composition of the surjection $\tilde \epsilon$ and the inclusion 
 $ \calO_{\tdN}(1) \hookrightarrow \calO_Y(0, 1, 0)$.
 Similarly, the  right hand side of ~\eqref{eq:nontrivial morph} is generated by the composition
\begin{equation*}
\xymatrix{
J_1\star_1 \Wh_S \ar[r]^-{\sigma'} & b_! T_B \ar[r]^-{\tilde \epsilon'} & \calO_{\tdN}(1)
 }
\end{equation*}
such that  the endomorphism $\epsilon': b_! T_B \to  b_! T_B$ is the composition of the surjection $\tilde \epsilon'$ and the inclusion 
 $ \Eis_0 \hookrightarrow  b_! T_B$. Moreover, in the two claims in the proof of Proposition~\ref{p:Eis},
 we saw that $\Phi(\sigma) = \sigma'$, and $\Phi(\epsilon)$ is a non-zero multiple of $\epsilon'$.
 Thus $\Phi(\tilde \ep)$ is a non-zero multiple of $\tilde\ep'$, since both lie in one-dimensional spaces, 
 and we have confirmed \eqref{eq:nontrivial morph} is a quasi-isomorphism.

 For $n>0$, both sides of ~\eqref{eq:nontrivial morph} vanish. On the one hand, $\Gamma(\Loc_{\SL(2)}(\PP^1, S),\calO_\Delta(n))$ is a direct sum of the $\SL(2)$-invariants in $\Gamma(\PP^{1},\calO_{\PP^1}(n+2i))$, for $i\geq0$, and hence vanishes for $n>0$. On the other hand,
 the support of $\Phi(\calO_\Delta(n))=\Eis_{n-1}$ is disjoint from the support of $\Wh_{S}$, and hence they are orthogonal.
\end{proof}

By invoking the identifications and symmetries for the automorphic and spectral categories recorded in Sections~\ref{sss:cons shv rel} and ~\ref{sss:coh shv rel}, we can conclude from the theorem 
an additional equivalence. 

\begin{cor}
There is an equivalence
\begin{equation*}
\xymatrix{
 \Coh^{\SL(2)-\alt} (\Loc_{\PGL(2)}(\PP^1, S)) \ar[r]^-\sim & \Sh_!(\Bun_{\SL(2)}(\PP^1, S)).
}
  \end{equation*}
compatible with affine Hecke actions at $s\in S$. Here we write $\Coh^{\SL(2)-\alt}(\Loc_{\PGL(2)}(\PP^1, S))$ for the dg category of $\SL(2)$-equivariant coherent complexes on $(\tilde\calN^\vee)^{S, \prod=1}$,
where the equation $\prod = 1$ is imposed inside of $\PGL(2)$, and such that the center $\mu_2 \simeq Z(\SL(2)) \subset \SL(2)$ acts by the alternating
representation on coherent complexes.

%\marginpar{Expanded here and added a proof.}

It restricts to equivalences
\begin{equation}\label{SL2 Phi ev}
\xymatrix{   \Coh^{\SL(2)-\alt}(\Loc^{\ov 0}_{\PGL(2)}(\PP^1, S))\ar[r]^-{\sim} &  \Sh_{!}^{\triv}(\Bun_{\SL(2)}(\PP^1, S))     
}
\end{equation}
\begin{equation}\label{eq:cusp}
\xymatrix{   \Coh^{\SL(2)-\alt}(\Loc^{\ov 1}_{\PGL(2)}(\PP^1, S))\ar[r]^-{\sim} &  \Sh_{!}^{\alt}(\Bun_{\SL(2)}(\PP^1, S)).      
}
\end{equation}
\end{cor}

\begin{proof}
The equivalence~\eqref{SL2 Phi ev} follows by combining \eqref{eq:Phi alt} and \eqref{eq:PGL ev SL2}, and the fact that
\begin{equation*}
\Coh^{\SL(2)-\alt} (\Loc^{\ov 0}_{\PGL(2)}(\PP^1, S))=\Coh^{\alt} (\Loc_{\SL(2)}(\PP^1, S)).
\end{equation*}

By Section~\ref{sss:auto SL2} and ~\ref{sss:spec PGL2} that both sides of ~\eqref{eq:cusp} are equivalent to $\Vect$; the equivalence~\eqref{eq:cusp} then follows immediately.
\end{proof}

\begin{remark} The sheaf $\cF_{0}(\vn)^{\alt}$ on $\Bun_{\SL(2)}(\PP^1, S)$ is a cuspidal Hecke eigensheaf with eigenvalue given by the unique ``odd'' $\PGL(2)$-local system on $\PS$ given in Lemma \ref{l:Loc odd}. See Remark~\ref{r:Legendre} for a description of this local system. 
\end{remark}

\begin{remark}
Though we will not discuss the details here, the above equivalences further restrict to equivalences from those coherent sheaves with nilpotent singular support
to those constructible sheaves that are point-wise compact
\begin{equation*}
\xymatrix{
 \Coh_\calN ( \Loc_{\SL(2)}(\PP^1, S)) \ar[r]^-\sim & \Sh_\dag(\Bun_{\PGL(2)}(\PP^1, S))
}
  \end{equation*}
  \begin{equation*}
\xymatrix{
\Coh^{\SL(2)-\alt}_\calN (\Loc_{\PGL(2)}(\PP^1, S)) \ar[r]^-\sim & \Sh_\dag(\Bun_{\SL(2)}(\PP^1, S)).
}
  \end{equation*}

\end{remark}

%%%%%%%%%%%%%%%%%%%%%%%%%%%%%%%%%%%%%%%%%%%%%%%%%%%%

\subsection{Unipotently monodromic version} We record here the monodromic form of the prior equivalence. Its construction
and proof are similar.

Let $\Bun_{\PGL(2)}(\PP^1, \wt S)$ denote the moduli of $\PGL(2)$-bundles on $\PP^1$ with $N$-reductions at the points
of $S = \{0, 1, \infty\}$. Note the natural map $\pi:\Bun_{\PGL(2)}(\PP^1, \wt S)\to \Bun_{\PGL(2)}(\PP^1, S)$ is a $T^S=T^{3}$-torsor.

Let $\Sh_!^{mon}(\Bun_{\PGL(2)}(\PP^1, \wt S))$ denote the full dg subcategory of 
$\Sh_!(\Bun_{\PGL(2)}(\PP^1, \wt S))$ generated by pullbacks along $\pi$.

Let $\Loc_{\SL(2)}(\PP^1, \widetilde S)$ denote the Betti moduli of $\SL(2)$-local systems on $\PP^1 \setminus S$ with $B^\vee$-reductions near $S$ with arbitrary induced $T^\vee$-monodromy.
Thus it admits a  presentation
$$
\xymatrix{
\Loc_{\SL(2)}(\PP^1, \widetilde S) \simeq (\widetilde {\SL(2)})^{S, \prod=1} /\SL(2)
}
$$
where, $\wt {\SL(2)}$ is the Grothendieck alteration of $\SL(2)$, and  the equation on the product of the group elements $\prod = 1$ is imposed inside of $\SL(2)$.

Let $\Coh _{\Loc_{\SL(2)}(\PP^{1},S)}( \Loc_{\SL(2)}(\PP^1,  \wt S)) $ be the full subcategory of  $\Coh( \Loc_{\SL(2)}(\PP^1,  \wt S))$ consisting of coherent complexes set-theoretically supported on the substack $\Loc_{\SL(2)}(\PP^{1},S)$.
 
%Let $\Loc_{\SL(2)}(\PP^1,  \widehat S)$ denote the completion of $\Loc_{\SL(2)}(\PP^1, \widetilde S)$ along  $\Loc_{\SL(2)}(\PP^1, S)$.

%\marginpar{Reformulated the theorem. Original version with the completion is incorrect.}

\begin{theorem}
There is an equivalence
\begin{equation*}
\xymatrix{
\wt\Phi_\Coh: \Coh_{\Loc_{\SL(2)}(\PP^{1},S)}( \Loc_{\SL(2)}(\PP^1,  \wt S)) \ar[r]^-\sim & \Sh^{mon}_!(\Bun_{\PGL(2)}(\PP^1, \wt S))
}
  \end{equation*}
compatible with Hecke modifications. 
\end{theorem}

The proof is similar to the equivariant version with the following changes.

The monodromic version of the Whittaker sheaf $\widehat\Wh_{S}$ corresponds to the structure sheaf of the completion of $\Loc_{\SL(2)}(\PP^1,  \wt S))$ along $\Loc_{\SL(2)}(\PP^1, S)$. It can be constructed as follows. 
Consider the diagram of Cartesian squares of open substacks
\begin{equation*}
\xymatrix{
\ar[d] \wt {c_1(*) }   \ar@{^(->}[r]^-{\wt j} &  \ar[d]_-\pi
  \wt{ c_1(*)} \cup \wt {c_1(\vn)}   \ar@{^(->}[r]^-{\wt i} & 
\ar[d] \Bun^{\ov 1}_{\PGL(2)}(\PP^1, \wt S)\\
 c_1(*)    \ar@{^(->}[r]^-j & 
  c_1(*) \cup c_1(\vn)   \ar@{^(->}[r]^-i & 
\Bun^{\ov 1}_{\PGL(2)}(\PP^1, S)
}\end{equation*}
where the vertical maps are $T^S$-torsors.
In particular, since $ c_1(*)  $ is simply a point, $\wt{ c_1(*)  }$ is a itself a $T^S$-torsor.
Then the free-monodromic Whittaker sheaf 
is given by
\begin{equation*}
\xymatrix{
\widehat\Wh_{S} = \wt{i}_!  \wt{j}_* \calL_{ c_1(*)}[2\#S\cdot \dim T]= \wt{i}_!  \wt{j}_* \calL_{ c_1(*)}[6]\in  \Sh^{mon}_!(\Bun^{\ov 1}_{\PGL(2)}(\PP^1, \wt S))
}\end{equation*}
where $\calL_{ c_1(*)}$ denotes the free-monodromic unipotent local system on $\wt{ c_1(*) }$: its monodromy representation is the completion of the regular representation of $\pi_{1}(\wt{ c_1(*) })\cong \pi_{1}(T^{S})$ at the augmentation ideal. By construction we have
\begin{equation*}
\pi_{!}\widehat\Wh_{S}\simeq \Wh_{S}.
\end{equation*}

The functor $\wt\Phi_\Coh$ is constructed by acting on the monodromic Whittaker sheaf.
Its essential surjectivity follows from that of the equivariant case, and its fully faithfulness comes down to the
calculation 
\begin{equation*}
\Hom(\widehat \Wh_{S}, \pi^{!}\Eis_{-1})\simeq \Hom(\pi_{!}\widehat \Wh_{S}, \Eis_{-1})\simeq \Hom(\Wh_{S},\Eis_{-1})\simeq \QQ.
\end{equation*}

%%%%%%%%%%%%%%%%%%%%%%%%%%%%%%%%%%%%%%%%%%%%%%%%%%%%
%%%%%%%%%%%%%%%%%%%%%%%%%%%%%%%%%%%%%%%%%%%%%%%%%%%%
%%%%%%%%%%%%%%%%%%%%%%%%%%%%%%%%%%%%%%%%%%%%%%%%%%%%

\end{document}